\documentclass[11pt, amscd, psamsfonts]{amsart}

\usepackage{amsthm,   amsmath}
\usepackage{amssymb} 
\usepackage{enumerate}

\usepackage[all]{xy}
\usepackage{verbatim}

\usepackage{graphicx}
\usepackage{amsfonts}
\usepackage{mathrsfs}
\usepackage{color}
\usepackage{soul}
\usepackage{hyperref}

  \usepackage[margin=1.40in]{geometry}

\usepackage{amsthm}

\newtheorem{theorem}{Theorem}[subsection]

\newtheorem{cor}[theorem]{Corollary}
\newtheorem{conjecture}[theorem]{Conjecture}

\newtheorem{lemma}[theorem]{Lemma}
\newtheorem{proposition}[theorem]{Proposition}

\newtheorem{suptheorem}{Theorem}[section]
\newtheorem{supcor}[suptheorem]{Corollary}
\newtheorem{suplemma}[suptheorem]{Lemma}
\newtheorem{supproposition}[suptheorem]{Proposition}

\theoremstyle{remark}
\newtheorem{remark}[theorem]{Remark}

\newtheorem{example}[theorem]{Example}

\newtheorem{supremark}[suptheorem]{Remark}

\theoremstyle{definition}
\newtheorem{definition}[theorem]{Definition}

\newtheorem{supdefinition}[suptheorem]{Definition}

\newcommand{\rZ}{{\rm Z}}
\newcommand{\rH}{{\rm H}}
\newcommand{\rE}{{\rm E}}
\newcommand{\quash}[1]{}  
\newcommand{\lps}{[\![}
\newcommand{\rps}{]\!]}

 \newcommand{\scr}{\mathscr }

\newcommand\rK{{\rm K}}

\newcommand\bu{\bullet}
\newcommand\und{\underline}
\newcommand\E{\mathcal{E}}
\newcommand\calF{\mathcal{F}}
\newcommand\calG{\mathcal{G}}
\newcommand\calC{\mathcal{C}}

\usepackage{marginnote}

\DeclareFontEncoding{OT2}{}{}

\newcommand{\mr}[1]{\mathrm{#1}}
\newcommand{\mf}[1]{\mathfrak{#1}}
\newcommand{\mc}[1]{\mathcal{#1}}
\newcommand{\mb}[1]{\mathbb{#1}}

\newcommand{\Z}{\mb{Z}}
\newcommand{\Q}{\mb{Q}}
\newcommand{\zp}{\mb{Z}_p}
\newcommand{\qp}{\mb{Q}_p}

\newcommand{\cont}{\mr{cont}}

\newcommand{\fin}{\mr{fin}}

\newcommand{\La}{\Lambda}

\newcommand{\Ga}{\Gamma}

\newcommand{\Cl}{\mr{Cl}}
\newcommand{\Iw}{\mr{Iw}}

\newcommand{\invlim}[1]{\varprojlim_{#1} \,}
\newcommand{\dirlim}[1]{\varinjlim_{#1} \,}
\newcommand{\ps}[1]{[\![ #1 ]\!]}

\newcommand{\cotimes}[1]{\,\hat{\otimes}_{#1} \,}

\newcommand{\finite}{(\mr{finite})}

\DeclareMathOperator{\Hom}{Hom} 
 \DeclareMathOperator{\Gal}{Gal}
 \DeclareMathOperator{\Res}{Res}
\DeclareMathOperator{\coker}{coker} 
 
\DeclareMathOperator{\im}{im}

\DeclareMathOperator{\Ext}{Ext}
\DeclareMathOperator{\Spec}{Spec}
\DeclareMathOperator{\rank}{rank}

\def\Z{\mathbb{Z}}

\def\Q{\mathbb{Q}}

\def\O{{\mathcal{O}}}

\def\Box{{\bullet}}

\DeclareMathOperator{\cyc}{cyc}

\def\arrowdown#1#2{\Big\downarrow \rlap{$\vcenter{\hbox{$\scriptstyle#2$}}$}
{\hbox to -10pt{\hss{$\vcenter{\hbox{$\scriptstyle#1$}}$}}}}
\def\arrowup#1#2{\Big\uparrow \rlap{$\vcenter{\hbox{$\scriptstyle#2$}}$}
{\hbox to -10pt{\hss{$\vcenter{\hbox{$\scriptstyle#1$}}$}}}}

 \numberwithin{equation}{section}

\begin{document}
\title{Higher Chern classes in Iwasawa theory}

\author[\scriptsize  Bleher]{\scriptsize F. M. Bleher}
\address{F. M. Bleher, Dept. of Mathematics\\Univ. of Iowa\\Iowa City, IA 52242, USA}
\email{frauke-bleher@uiowa.edu}

\author[Chinburg]{T. Chinburg}
\address{T. Chinburg, Dept. of Mathematics\\ Univ. of Pennsylvania \\ Philadelphia, PA 19104, USA}
\email{ted@math.upenn.edu}

\author[Greenberg]{R. Greenberg}
\address{R. Greenberg\\Dept. of Mathematics\\Univ. of Washington\\Box 354350\\ Seattle, WA 98195, USA}
\email{greenber@math.washington.edu}

\author[Kakde]{M. Kakde}
\address{M. Kakde\\Dept. of Mathematics\\Indian Institute of Science\\Bangalore 560012, India}
\email{maheshkakde@iisc.ac.in}

\author[Pappas]{G. Pappas}
\address{G. Pappas\\ Dept. of
Mathematics\\
Michigan State Univ.\\
E. Lansing, MI 48824, USA}
\email{pappas@math.msu.edu}

\author[Sharifi]{R. Sharifi}
\address{R. Sharifi\\Dept. of Mathematics\\UCLA\\Box 951555\\Los Angeles, CA 90095, USA}
\email{sharifi@math.ucla.edu}

\author[Taylor]{M. J. Taylor}
\address{M. J. Taylor\\  Merton College\\ 
Oxford, OX1 4JE, UK}
\email{martin.taylor@merton.ox.ac.uk}



\begin{abstract} \
We begin a study of $m$th Chern classes and $m$th characteristic symbols for Iwasawa modules which are supported in codimension at least $m$. This extends the classical theory of  characteristic ideals and their generators for  Iwasawa modules which are torsion, i.e., supported in codimension at least $1$. We apply this to an Iwasawa module constructed from an inverse limit of $p$-parts of ideal class groups of abelian extensions of an imaginary quadratic field. When this module is pseudo-null, which is conjecturally always the case, we determine its second Chern class and show that it has a characteristic symbol given by the Steinberg symbol of two Katz $p$-adic $L$-functions. 
 \end{abstract}

\maketitle

 \tableofcontents
\newpage

\section{Introduction}
\label{s:intro}

The main conjecture of Iwasawa theory in its most classical form
asserts the equality of two ideals in a formal power series ring.
The first is defined through the action of the abelian Galois group of the $p$-cyclotomic
tower over an abelian base field on a limit of $p$-parts of class groups in the tower.
The other is generated by a power series that interpolates values of Dirichlet $L$-functions.
This conjecture was proven
by Mazur and Wiles \cite{MazurWiles} and has since been generalized in a multitude of ways.
It has led to the  development of a wide range of new methods in
number theory, arithmetic geometry and the theory of modular forms: 
see for example \cite{Green1}, \cite{KatoICM2006}, \cite{Iwasawa2012} and their references. 
 As we will explain in Section \ref{s:conj}, classical main conjectures pertain to the first
 Chern classes of various complexes of modules over Iwasawa algebras. 
In this paper, we begin a study of the higher
 Chern classes of such complexes and their relation
 to analytic invariants such as $p$-adic $L$-functions.  
 This can
 be seen as studying the behavior in higher codimension of
 the natural  complexes.
 
Higher Chern classes appear implicitly in some of the earliest 
work of Iwasawa \cite{IwBAMS}.
Let $p$ be an odd prime, and let $F_\infty$ denote a 
$\mathbb{Z}_p$-extension of a number field $F$.  
Iwasawa showed that for sufficiently large $n$, the order of the $p$-part of the ideal class group of the cyclic extension of degree $p^n$ in $F_\infty$
  is 
 \begin{equation}
 \label{eq:growthrate}
 p^{\mu p^n + \lambda n + \nu}
 \end{equation}
  for some constants $\mu$, $\lambda$ and $\nu$. Let $L$ be the maximal abelian
  unramified pro-$p$ extension of $F_\infty$.  Iwasawa's theorem is proved
  by studying the structure of $X = \mathrm{Gal}(L/F_\infty)$ as a module for the
  Iwasawa algebra $\Lambda = \mathbb{Z}_p\ps{\Gamma} \cong \mathbb{Z}_p\ps{t}$ associated to 
   $\Gamma = \mathrm{Gal}(F_\infty/F) \cong \mathbb{Z}_p$. 
Here, $\Lambda $ is a dimension two unique factorization domain with a unique codimension two prime ideal $(p, t)$, which has residue field $\mathbb{F}_p$.    The focus of classical
Iwasawa theory is on the invariants $\mu$ and $\lambda$, which pertain to the support of $X$
in codimension $1$ as a torsion finitely generated $\Lambda $-module. More precisely, 
$\mu$ and $\lambda$ are determined by the first Chern class of $X$ as a $\Lambda $-module,
as will be explained in  Subsection \ref{ss:firstclass}.  Suppose now
  that $\mu = 0 = \lambda$.  Then $X$ is either zero or supported in codimension $2$ (i.e., $X$ is pseudo-null), and  
  $$
  \nu \in \mathbb{Z} = \rK_0(\mathbb{F}_p) 
  $$
may be identified with the (localized) second Chern class of $X$ as a $\Lambda $-module.  In general,
 the relevant Chern class is associated to the codimension of the support of an Iwasawa module.
 This class can be thought of as the leading term in the algebraic description of the module.
 When one is dealing with complexes of modules, the natural codimension is that of the support
 of the cohomology of the complex.  
 
 \subsection{Chern classes and characteristic symbols.}
 \label{s:chernclass}
 
There is a general theory of  localized Chern classes due to Fulton-MacPherson \cite[Chapter 18]{FultonIntersectionTheory} based on MacPherson's 
graph construction (see also \cite{RobertsBook}).
Moreover, Gillet developed a sophisticated theory of  
Chern classes in $\rK$-cohomology with supports in
\cite{GilletRR}.  This pertains to suitable complexes of modules over a Noetherian
scheme which are exact off a closed subscheme and requires certain assumptions, including 
Gersten's conjecture. 
In this paper, we will
restrict to a special situation that can be examined by simpler tools.  
Suppose that  $R$ is a local commutative Noetherian ring and that $\calC^\bullet$ is a bounded complex of finitely generated $R$-modules which is exact in codimension less than $m$.  We now describe an $m$th Chern class which can be associated to $\calC^\bullet$. In our 
applications, $R$ will be an Iwasawa algebra. 

Let $Y = \mathrm{Spec}(R)$, and let
$Y^{(m)}$ be the set of codimension $m$ points of $Y$, i.e., height $m$ prime ideals
of $R$. Denote by $\rZ^m(Y)$   the group of cycles of codimension $m$ in $Y$, i.e.
the free abelian group  generated by $y\in Y^{(m)}$:
$$
\rZ^m(Y) =  \bigoplus\nolimits_{y \in Y^{(m)}} \mathbb{Z}\cdot y.
$$
 For $y\in Y^{(m)}$, let $R_y$ denote the localization of $R$ at $y$, and set  $\calC^\bullet_y=\calC^\bu\otimes_R R_y$.
Under our condition on $\calC^\bu$, the cohomology groups ${\rm H}^i(\calC^\bu_y)={\rm H}^i(\calC^\bu)\otimes_R R_y$ are finite length $R_y$-modules. We then define a (localized)
Chern class $c_m(\mathcal{C}^\bullet )$ in the group 
$ 
\rZ^m(Y) 
$ 
by letting the component at $y$ of $c_m(\calC^\bullet)$ be the 
alternating sum of the lengths 
$$
\sum\nolimits_i (-1)^i {\rm length}_{R_y}{\rm H}^i(\calC^\bu_y).
$$
If the codimension of the support of some ${\rm H}^i(\calC^\bu_y)$ is exactly $m$, the Chern class $c_m(\calC^\bullet)$ is what we referred to earlier, just
before the start of Subsection  \ref{s:chernclass}, as the leading term 
of $\calC^\bullet$ as a complex of $R$-modules.  This is a very special case of the construction in \cite{RobertsBook} and 
 \cite[Chapter 18]{FultonIntersectionTheory}. In particular, if $M$ is a 
finitely generated $R$-module which is supported in codimension at least $m$, we have $$c_m(M)=
\sum_{y\in Y^{(m)}}  {\rm length}_{R_y}(M_y)\cdot y.
$$

We would now like to relate $c_m(\calC^\bullet)$ to analytic invariants.
Suppose that $R$ is a regular integral domain, and let $Q$ be the fraction field of $R$.  When $m = 1$ one can use the divisor homomorphism
$$\nu_1 \colon Q^\times \to \rZ^1(Y) =  \bigoplus\nolimits_{y \in Y^{(1)}} \mathbb{Z}\cdot y.$$
In the language of the classical main conjectures, an element $f \in Q^\times$
such that $\nu_1(f) = c_1(M)$ is a characteristic power series for $M$
when $R$ is a formal power series ring.  A main conjecture for $M$ posits
that there is such an $f$ which can be constructed analytically, e.g. via
$p$-adic $L$-functions.

The key to generalizing this is to observe that $Q^\times$ is the first
Quillen $\rK$-group $\rK_1(Q )$ and $\nu_1$ is a tame symbol map.  
To try to relate $c_m(M)$ to analytic
invariants for arbitrary $m$, one can consider elements of $\rK_m(Q )$
which can be described by symbols involving $m$-tuples of elements
of $Q $ associated to $L$-functions.  The homomorphism $\nu_1$
is replaced by a homomorphism $\nu_m$ involving compositions of tame symbol maps.  
We now describe one way to do this.

Suppose that $\und\eta = (\eta_0,\ldots,\eta_{m})$ 
is a sequence of points of $Y$ with ${\rm codim}(\eta_i)=i$ and such that 
$\eta_{i+1}$ lies in the closure $\overline{\eta}_i$ of $\eta_i$ for all $i < m$. 
Denote by $P_m(Y)$ the set of all such sequences. Let $k(\eta_i)=Q(R/\eta_i)$ be the residue field of $\eta_i$.
Composing successive tame symbol maps (i.e., connecting maps of  localization sequences), we obtain homomorphisms
$$
	\nu_{\und\eta} \colon \rK_m(Q)=\rK_m(k(\eta_0))\to \rK_{m-1}(k(\eta_1))\to\cdots\to \rK_0(k(\eta_m))=\Z.
$$
Here, $\rK_i$ denotes the $i$th Quillen $\rK$-group.
We combine these in the following way to give a homomorphism 
$$
	\nu_m \colon \bigoplus_{\und\eta'\in P_{m-1}(Y)} \rK_m(Q)\to \rZ^m(Y) =  \bigoplus\nolimits_{y \in Y^{(m)}} \mathbb{Z}\cdot y.
$$
Suppose $a = (a_{\und\eta'})_{\und\eta'\in P_{m-1}(Y)}$.  We define
the component of $\nu_m(a)$ at $y$ to be the sum  of $\nu_{(\eta'_0,\eta'_1,\ldots , \eta'_{m-1}, y)}(a_{\und\eta'})$ 
over all the sequences $$\und\eta'  = (\eta'_0,\eta'_1,\ldots , \eta'_{m-1}) \in P_{m-1}(Y)$$ such that $y$ is in the closure of $\eta'_{m-1}$.

If $M$ is a finitely generated $R$-module supported in codimension at least $m$ as above, then we refer to any element in
$\bigoplus_{\und\eta'\in P_{m-1}(Y)} \rK_m(Q)$ that $\nu_m$ maps to $c_m(M)$ as a {\sl characteristic symbol} for $M$.
This generalizes the notion of a characteristic power series of a torsion module 
in classical Iwasawa theory, which can be reinterpreted as the case $m=1$.

We focus primarily on the case in which 
$m= 2$ and  $R$ is a formal power
series ring $A\lps t_1,\ldots,t_r\rps$ over a mixed characteristic complete discrete
valuation ring $A$.  In this case, we show that the symbol map $\nu_2$ 
gives an isomorphism 
\begin{equation}
\label{eq:niceisom}
  \frac{\prod'_{\eta_1 \in Y^{(1)}} \rK_2(Q)}{\rK_2(Q) \prod_{\eta_1\in Y^{(1)}}\rK_2(R_{\eta_1})} \xrightarrow{\ \sim\ } \rZ^2(Y).
\end{equation}
This uses the fact that Gersten's conjecture holds for $\rK_2$ and $R$.
In the numerator of \eqref{eq:niceisom}, the restricted product $\prod'_{\eta_1\in Y^{(1)}} \rK_2(Q)$ is the subgroup of the direct product in which all but a finite number of components belong to $\rK_2( R_{\eta_1})\subset \rK_2(Q)$. In the denominator, we have the product of the subgroups $\prod_{\eta_1\in Y^{(1)}}\rK_2( R_{\eta_1})$ and  $\rK_2( Q)$, the second group embedded diagonally in $\prod'_{\eta_1\in Y^{(1)}} \rK_2( Q)$.
 The significance of this formula is that it shows that one can specify elements of $\rZ^2(Y)$
 through a list of elements of $\rK_2(Q)$, one for each codimension one prime $\eta_1$ of $R$,
 such that the element for $\eta_1$ lies in $\rK_2( R_{\eta_1})$ for all but finitely many $\eta_1$.  
 
\subsection{Results.}

Returning to Iwasawa theory, an optimistic hope one might have is that
under certain hypotheses, the second Chern class of an Iwasawa module or complex thereof
can be described using \eqref{eq:niceisom} and  Steinberg symbols in $\rK_2(Q)$ with arguments that
are $p$-adic $L$-functions.  Our main result, Theorem \ref{thm:imagquad},   is of
exactly this kind.    
In it, we work under the assumption of a conjecture of Greenberg which predicts that 
certain Iwasawa modules over multi-variable power series rings are pseudo-null,
i.e., that they have trivial support in codimension $1$.  We recall this conjecture and some
evidence for it found by various authors in 
Subsection \ref{ss:GreenbergsConjecture}.  

More precisely,   we consider in Subsection \ref{ss:imaginaryquad} an  
imaginary quadratic field $E$, and we assume that $p$ is an odd prime that splits into two primes 
$\mf{p}$ and $\bar{\mf{p}}$ of $E$.   
Let $\widetilde{E}$ denote the compositum of all $\zp$-extensions of $E$.  Let $\psi$ be a one-dimensional $p$-adic
character of
the absolute Galois group of $E$ of finite order prime to $p$, and denote by $K$ (resp., $F$) the compositum of the fixed field of $\psi$ with $\widetilde{E}(\mu_p)$ (resp., $E(\mu_p)$). 
We consider the Iwasawa module $X=\mathrm{Gal}(L/K)$, where $L$ is the maximal abelian unramified pro-$p$ extension of $K$. Set $\mathcal{G}=\mathrm{Gal}(K/E)$, and
let $\omega$ be its Teichm\"uller character.   Let $\Delta = \Gal(F/E)$, which we may also view as the largest subgroup of $\mc{G}$
of prime-to-$p$ order.  For simplicity in this discussion, we suppose that $\psi \neq 1, \omega$.  

The Galois group $\mathcal{G}$ has an open maximal pro-$p$ subgroup $\Gamma$ isomorphic to $\zp^2$. 
Greenberg has conjectured that $X$ is pseudo-null as
a module for $\Lambda = \zp\ps{\Gamma} \cong \zp\lps t_1, t_2 \rps$. Our goal is to obtain information about $X$ 
and its eigenspaces $X^{ \psi } = \mc{O}_{\psi} \otimes_{\zp[\Delta]} X$,  
where $\mc{O}_{\psi}$ is the $\zp$-algebra generated by the values of $\psi$, and $\zp[\Delta] \to \mc{O}_{\psi}$ is the surjection induced by $\psi$.
When Greenberg's conjecture 
is true, the characteristic ideal giving the first Chern class of $X^{ \psi }$  is trivial. It thus 
makes good sense to consider the second Chern class, which gives information about the
height $2$ primes in the support of $X^{ \psi }$.

Consider the Katz $p$-adic $L$-functions $\mc{L}_{\mf{p}, \psi }$
and $\mc{L}_{\bar{\mf{p}},\psi }$ in the fraction field $Q$ of the ring $R = 
e_\psi \cdot W\ps{\mathcal{G}}\cong W\ps{t_1, t_2}$, where $e_\psi \in W\ps{\mathcal{G}}$ is the idempotent associated to $\psi$ and $W$ denotes the Witt vectors of an algebraic closure of ${\mathbb F}_p$. 
We can now define an analytic element $c_2^{\rm an}$ in the group $\rZ^2(\Spec(R))$ of \eqref{eq:niceisom} in the following way.  Let $c_2^{\rm an}$ be 
the image of the element on the left-hand side of \eqref{eq:niceisom} with component at $\eta_1$
the Steinberg symbol
$$
\{ \mc{L}_{\mf{p}, \psi }, \mc{L}_{\bar{\mf{p}}, \psi } \}\in \rK_2(Q),
$$
if $\mc{L}_{\mf{p}, \psi }$ is not a unit at $\eta_1$, and with other components trivial.
This element $c_2^{\rm an}$ does not depend on the ordering of $\mf{p}$ and $\bar{\mf{p}}$ 
(see Remark \ref{rem:symmetry}).

Our main result, Theorem \ref{thm:imagquad}, is that if $X$ is pseudo-null, then 
\begin{equation}
\label{eq:c2formula}
c_2^{\rm an} = c_2(X_W^{ \psi })  + c_2((X_W^{ \omega\psi^{-1} })^{\iota}(1)) 
\end{equation}
where  $X_W^{ \psi }$ and $(X_W^{ \omega\psi^{-1} })^{\iota}(1)$ are the $R$-modules defined
as follows: $X_W^{ \psi }$ is the completed tensor product $W \cotimes{ \mc{O}_{\psi}} X^{ \psi }$,
while $(X_W^{ \omega \psi^{-1} })^{\iota}(1)$ is the Tate twist of the module which results from 
 $X_W^{ \omega \psi^{-1} }$ by letting $g \in \mathcal{G}$
act by $g^{-1}$.

In \eqref{eq:c2formula}, one needs to take completed tensor products of Galois modules with $W$ because 
the analytic invariant $c_2^{\rm an}$ is only defined over $W$.  
Note that the right-hand side of \eqref{eq:c2formula} concerns
two different components of $X$, namely those associated to $\psi$
and $\omega \psi^{-1}$.  
It frequently occurs that exactly one of the two
is nontrivial:  see Example \ref{imquadexs}. In fact, one consequence of our main result is a codimension two
elliptic counterpart of the Herbrand-Ribet Theorem (see Corollary \ref{cor:Herbrand-Ribet}): the eigenspaces
$X^{ \psi }$ and $X^{ \omega \psi^{-1} }$ are both trivial
if and only
if one of $\mc{L}_{\mf{p}, \psi }$ or $ \mc{L}_{\bar{\mf{p}}, \psi } $
is a unit power series. 

 One can also interpret the
right-hand side of \eqref{eq:c2formula} in the following way.  Let $\Omega = \zp\ps{\mathcal{G}}$ and let 
$\epsilon \colon \Omega \to \Omega$
be the involution induced by the map $g \to \chi_{\text{cyc}}(g) g^{-1}$ on $\mathcal{G}$ where
$\chi_{\text{cyc}} \colon \mathcal{G} \to \mathbb{Z}_p^{\times}$ is the cyclotomic character.  
Then $(X^{ \omega\psi^{-1} })^{\iota}(1)$ is canonically isomorphic to the $\psi$ component
$(X_\epsilon)^{ \psi }$ of the twist $X_\epsilon = \Omega \otimes_{\epsilon,\Omega} X$
of $X$ by $\epsilon$.  Thus $X_\epsilon$ is isomorphic to $X$ as a $\Z_p$-module but with 
the action of $\Omega$ resulting from precomposing with the involution $\epsilon \colon 
\Omega \to \Omega$.  Then \eqref{eq:c2formula} can be written 
\begin{equation}
\label{eq:c2formulasecond}
c_2^{\rm an} = c_2(W \cotimes{ \mc{O}_{\psi}} (X \oplus X_\epsilon)^{ \psi }). 
\end{equation}

We discuss two extensions of \eqref{eq:c2formula}.  
In Subsection \ref{ss:onecomplexplace}, we  explain how the algebraic part of our result for imaginary quadratic fields extends, under certain additional hypotheses on $E$ and $\psi$, to number 
fields $E$ with at most one complex place.  
In Section \ref{s:generalization}, we show how  when $E$ is imaginary
quadratic and $K$ is Galois over $\mathbb{Q}$, we can obtain information about the $X$ above as a module for 
the non-commutative Iwasawa algebra
$\zp\ps{\mathrm{Gal}(K/\mathbb{Q})}$. This involves  a  
``non-commutative  second Chern class"  in which, instead of lengths of modules, we consider classes in appropriate Grothendieck groups.  Developing counterparts of
our results for more general non-commutative Galois groups is a natural
goal in view of the non-commutative main conjecture concerning
first Chern classes treated in \cite{Coates1}.

\subsection{Outline of the proof.}
\label{s:strategy}

We now outline the strategy of the proof of \eqref{eq:c2formula}.
We first consider the Galois group $\mf{X} = \mathrm{Gal}(N/K)$,
with $N$ the maximal abelian pro-$p$ extension of $K$ that is unramified outside of
$p$. One has $X = \mf{X}/(I_{\mf{p}} + I_{\bar{\mf{p}}})$ for $I_{\mf{p}}$ 
the subgroup of $\mf{X}$ generated by inertia groups of primes of $K$ over $\mf{p}$
and $I_{\bar{\mf{p}}}$ defined similarly for the prime $\bar{\mf{p}}$.
A novelty of the proof is that it requires carefully analyzing the 
discrepancy between the rank one $\Omega$-module $\mf{X}$ 
and its free reflexive hull. 
  
The reflexive hull of a $\La$-module $M$ is $M^{**} = (M^{*})^{*}$ for $M^* = \mathrm{Hom}_{\Lambda}(M,\Lambda)$, and there is a canonical homomorphism $M \to M^{**}$.
Iwasawa-theoretic duality results tell us that
since $X$ is pseudo-null, the map $\mf{X} \to \mf{X}^{**}$ is injective with an explicit pseudo-null cokernel
(in particular, see Proposition \ref{r2case}).
We have a commutative diagram
\begin{equation}
    \label{eq:diagramnasty}
    		\SelectTips{cm}{} \xymatrix{		
    		I_{\mf{p}} \oplus I_{\bar{\mf{p}}} \ar[r] \ar[d] & I_{\mf{p}}^{**} \oplus 
		I_{\bar{\mf{p}}}^{**}
    		\ar[d]\\
    		\mf{X} \ar[r] & \mf{X}^{**} . \\
    		}
\end{equation}
Taking cokernels of the vertical homomorphisms in \eqref{eq:diagramnasty} yields a homomorphism
$$
	f \colon X \to \mf{X}^{**}/(\mr{im}(I_{\mf{p}}^{**})+  \mr{im}(I_{\bar{\mf{p}}}^{**})),
$$
where $\mr{im}$ denotes the image. 
A snake lemma argument then tells us that the cokernel of $f$ is the
Tate twist of an Iwasawa adjoint $\alpha(X)$ 
of $X$ which has the same class  as $X^{\iota}$ in the quotient of the Grothendieck group of the category of pseudo-null modules by the 
Grothendieck group of the category of finite modules.  Moreover, the map $f$ is injective in its $\psi$-eigenspace as $\psi \neq \omega$.

The $\psi$-eigenspaces of $\mf{X}$, $I_\mf{p}$ and $I_{\bar{\mf{p}}}$ are of rank one over  
$\Lambda_\psi = \mathcal{O}_\psi\ps{\Gamma}$.  They need not be free, but the key point is that their reflexive hulls are. 
The main conjecture for imaginary quadratic fields proven by Rubin \cite{Rubin1, Rubin2} (see also \cite{Kings}) 
implies that the $p$-adic $L$-function $\mc{L}_{\bar{\mf{p}}, \psi }$ in $\Lambda_{W} = W\ps{\Gamma}$ generates 
the image of the map
$$
	W \cotimes{\mathcal{O}_\psi} (I_{\mf{p}}^{ \psi })^{**} \to W \cotimes{\mathcal{O}_\psi}(\mf{X}^{ \psi })^{**} \cong \Lambda_{W},
$$
and similarly switching the roles of the two primes.  
Putting everything together, we have an exact sequence of $\Lambda_{W}$-modules:
\begin{equation} \label{keyexseq} 0 \xrightarrow{}   	X_W^{ \psi } \to
    	\frac{\Lambda_{W} }{\mc{L}_{\mf{p}, \psi}\Lambda_{W} +  \mc{L}_{\bar{\mf{p}},\psi} \Lambda_{W} } \xrightarrow{} \alpha(X_W^{ \omega\psi^{-1} })(1) \xrightarrow{} 0.
\end{equation}
The second Chern class of the middle term is $c_2^{\mathrm{an}}$, and the second Chern class of the last term
depends only on its class in the Grothendieck group, yielding \eqref{eq:c2formula}.

\subsection{Generalizations.}

We next describe several potential generalizations of \eqref{eq:c2formula} to other fields and Selmer groups of higher-dimensional Galois representations.  We intend them as motivation for further study, leaving details to future work.

Consider first the case of a CM field $E$ of degree $2d$, and suppose that the primes over $p$ in $E$ are split from
the maximal totally real subfield.  There is then a natural generalization of the analytic class on the left side of (\ref{eq:c2formulasecond}).  Fix a $p$-adic CM type $\mathcal{Q}$ for $E$, and let $\bar{\mathcal{Q}}$ be the conjugate type.  Then for $K$ and $F$ defined
as above using a $p$-adic character $\psi$ of prime-to-$p$ order of the absolute Galois 
group of $E$, one has Katz $p$-adic $L$-functions $\mc{L}_{\mathcal{Q}, \psi }$ and $\mc{L}_{\bar{\mathcal{Q}}, \psi }$ in the algebra $R = \Omega^\psi_W$ for $\Omega = \zp\ps{\Gal(K/E)}$.  Suppose that the quotient
\begin{equation}
\label{eq:higheranalytic}
 \frac{R}{R \mc{L}_{\mathcal{Q}, \psi } +  R \mc{L}_{\bar{\mathcal{Q}}, \psi }}
 \end{equation}
is pseudo-null over $R$.  A generalization of \eqref{eq:c2formula} would relate the second Chern class $c_{2,\mathcal{Q}}^{\mr{an}}$ of (\ref{eq:higheranalytic}) to a sum of second Chern classes of algebraic objects arising from Galois groups.  

The question immediately arises of how to extend our algebraic methods, as the unramified outside $p$ Iwasawa module $\mf{X}$ over $K$ has $\Omega$-rank $d$.  For this, we turn to the use of highest exterior powers, which has a rich tradition in the study of special values of $L$-functions, notably in conjectures of Stark and Rubin-Stark.  Let $I_{\mc{Q}}$ denote the subgroup of  $\mf{X}$ generated by the inertia groups of the primes in $\mc{Q}$.  Recall that the main conjecture states that the quotient $\mf{X}/I_{\mc{Q}}$ has first Chern class agreeing with the divisor of $\mc{L}_{\mc{Q},\psi}$.  To obtain a rank one object 
related to  the above $p$-adic $L$-functions as before, it is natural to consider the $d$th exterior power of the reflexive hull of $\mf{X}$, which we may localize at a height $2$ prime to ensure its freeness.  Under Greenberg's conjecture, the quotient $\bigwedge_{\Omega}^d\mf{X}/\bigwedge_{\Omega}^dI_{\mc{Q}}$ is pseudo-isomorphic to the analogous quotient for exterior powers of reflexive hulls, and the main conjecture becomes the statement that $c_1((\bigwedge_{\Omega}^d\mf{X}/\bigwedge_{\Omega}^dI_{\mc{Q}})^{\psi}_W)$ is the divisor of $\mc{L}_{\mathcal{Q},\psi}$.  

This suggests that the proper object for comparison with the analytic class $c_{2,\mathcal{Q}}^{\mr{an}}$ is no longer the second Chern class of the $\psi$-eigenspace of the unramified Iwasawa module $X$ but of the quotient of $d$th exterior powers
\begin{equation}
\label{eq:Zdefn}
	Z_{\mc{Q}} = \frac{\bigwedge_{\Omega}^d\mf{X}}{\bigwedge_{\Omega}^dI_{\mc{Q}} + 
	\bigwedge_{\Omega}^d I_{\mf{\bar{\mc{Q}}}}}.
\end{equation}
 Following the approach used when $d = 1$, it is natural to consider the difference between  $Z_{\mc{Q}}$ and the analogous quotient in which every term is replaced by its reflexive hull.  This is the approach taken in  \cite{ExteriorPower}, where roughly speaking,  
we show that the analytic class $c_{2,\mc{Q}}^{\mr{an}}$ is a sum of second Chern classes arising from $Z_{\mc{Q}}^{\psi}$ and $X^{ \omega\psi^{-1}}$.  
We also provide a Galois-theoretic interpretation of $Z_{\mc{Q}}$ for $d = 2$ as a quotient 
of the second graded quotient in the lower central series 
of the Galois group of the maximal pro-$p$, unramified outside $p$ extension of $K$.
 
We can also consider the case in which $E$ is imaginary quadratic but $p$ is inert, so that $\mf{X}$ is again
of rank one but the product of corresponding inertia groups over the completions of $K$ at $p$ has rank two.  Using Kobayashi's plus/minus Selmer conditions, Pollack and Rubin \cite[Section 4]{pollackrubin} define  two rank one submodules $I^{+}$ and $I^{-}$ of $\mf{X}$ that play the role of $I_{\mf{p}}$ and $I_{\bar{\mf{p}}}$.  Using the results in the present paper, one can get an analogue of (\ref{keyexseq}) with $\mc{L}_{\mf{p}}$ and $\mc{L}_{\bar{\mf{p}}}$ replaced by characteristic elements of $\mf{X}/I^+$ and $\mf{X}/I^-$.  As a two-variable main conjecture in this setting is lacking, this does not directly yield a relation of the second Chern class with $L$-functions.  

Finally, we turn to the Selmer groups of ordinary $p$-adic modular forms, which fit in one-variable families known as Hida families, parameterized by the weight of the form.  Theorem \ref{thm:imagquad} is a special case of this framework involving  CM newforms.  Hida families of residually irreducible newforms give rise to  Galois representations with Galois-stable lattices free over an integral extension $\mathcal{I}$ of an Iwasawa algebra in two variables, the so-called weight and cyclotomic variables.  These lattices are self-dual up to a twist.  We can use the cohomology of such a lattice to define objects analogous to $\mf{X}$ and $X$.  The former
is of rank one over $\mathcal{I}$, and the latter as before should be pseudo-null, with the dual Selmer group of the lattice providing an intermediate object between the two.  The dual Selmer group is expected to have first Chern class given by a Mazur-Kitagawa $p$-adic $L$-function.  In this case, the algebraic study goes through without serious additional complication, and the only obstruction to an exact sequence as in \eqref{keyexseq} is the identification of a second annihilator. 

The above examples may be just the tip of an iceberg.  In  \cite{greenberg:1994}, a main conjecture is formulated in a very general context where one considers a Galois representation over a complete Noetherian local ring $R$ with finite residue field of characteristic $p$. A main conjecture corresponds to a so-called {\em Panchishkin condition}.  It is not uncommon for there to be more than one choice of a Panchishkin condition and hence more than one main conjecture.  On the analytic side, the corresponding $p$-adic $L$-functions should often have divisors intersecting properly. On the algebraic side, the Pontryagin dual of the intersections of the corresponding Selmer groups should often be supported in codimension 2. It is then tempting to believe that the type of result we consider in this paper would have an analogue in this context and would involve taking highest exterior powers of appropriate $R$-modules. 

There are other situations where one can define more than one Selmer group and more than one $p$-adic $L$-function in natural ways, found for example in the work 
of Pollack \cite{pollack:2003}, Kobayashi \cite{kobayashi:2003}, Lei-Loeffler-Zerbes \cite{llz:2010}, Sprung \cite{sprung:2012}, Pottharst \cite{pottharst:2013}.
Some of these constructions involve what amounts to a choice of Panchishkin condition after a change of scalars.  This begs the question, that we do not address here, of how to define a suitable generalized notion of a Panchishkin condition.

\subsection{Organization of the paper.}

In Section \ref{s:chern}, we define Chern classes and characteristic symbols, and 
we explain how \eqref{eq:niceisom} follows from
certain proven cases of  Gersten's conjecture. 
In Section \ref{s:conj}, we recall the formalism
of some previous main conjectures in Iwasawa theory.  We also recall properties of Katz's 
$p$-adic $L$-functions and 
Rubin's results on the main conjecture over imaginary quadratic fields.  In Subsection \ref{ss:GreenbergsConjecture}, 
we recall Greenberg's conjecture and some evidence for it.

In Section \ref{s:FieldsWithR2Equal1}, we discuss various Iwasawa modules in some generality.
The emphasis is on working out Iwasawa-theoretic consequences of Tate, Poitou-Tate and Grothendieck
duality.
This requires the work in the Appendix, which concerns Ext-groups
and Iwasawa adjoints of modules over certain completed group rings.

We begin Section \ref{s:reflectiontype} with a discussion of reflection theorems of 
the kind we will need to discuss  Iwasawa theory in codimension two.  
In Subsection \ref{ss:rational}, we discuss codimension two phenomena in the most
classical case of the cyclotomic $\zp$-extension of an abelian extension of $\mathbb{Q}$.  Our main result over imaginary
quadratic fields is proven in Subsection \ref{ss:imaginaryquad} using the strategy discussed above.  The extension
 of the algebraic part of the proof to number fields with at most one complex place is given in Subsection
 \ref{ss:onecomplexplace}.  The non-commutative generalization over
 imaginary quadratic fields is proved
 in Section \ref{s:generalization}.

\subsection*{Acknowledgements} F.\,B. was partially supported by NSF FRG Grant No.\ DMS-1360621
and NSF Grant No.\ DMS-1801328.
 T.\,C. was partially supported by  NSF FRG Grant No.\ DMS-1360767, NSF FRG Grant No.\ DMS-1265290,
NSF SaTC grants No. CNS-1513671 and CNS-1701785, Simons Foundation Grant No.\ 338379
and 
NSF Grant No.\ DMS-1107263/1107367/1107452 ``RNMS: Geometric Structures and Representation 
Varieties" (the GEAR Network).
F.\,B. and T.\,C. would also like to thank  L'Institut des Hautes \'{E}tudes Scientifiques for support during 
the Fall of 2015. R.\,G. was partially supported by NSF FRG Grant No.\ DMS-1360902
and 
would also like to thank La Fondation Sciences
Math\'{e}matiques de Paris for its support during the Fall of 2015.
M.\,K. was partially supported by EPSRC First Grant No.\ EP/L021986/1 and would also like to thank TIFR, Mumbai, for its hospitality during the writing of parts of this paper. 
G.\,P.  was partially supported by  NSF FRG Grant No.\ DMS-1360733 and NSF Grant No.\ DMS-1701619.
R.\,S.  was partially supported by  NSF FRG Grant No.\ DMS-1360583, NSF Grant Nos.\ DMS-1401122/1615568 and DMS-1801963, and
Simons Foundation Grant No.\ 304824.
T.\,C., R.\,G., M.\,K. and R.\,S. would like to thank the Banff International Research Station for hosting the workshop 
``Applications of Iwasawa Algebras'' in March 2013.

\newpage

 \section{Chern classes and characteristic symbols}

 \label{s:chern}
 
\subsection{Chern classes}
\label{ss:generalsetup}  

We denote by $\rK'_m(R)$ and $\rK_m(R)$, the Quillen $\rK$-groups \cite{Quillen} of a ring $R$ defined using the categories of finitely generated and finitely generated projective $R$-modules, respectively.  If $R$ is regular and Noetherian, then we can identify $\rK_m(R)=\rK_m'(R)$.

Suppose that $R$ is a commutative local integral Noetherian ring.  Denote by ${\mathfrak m}$ the maximal ideal of $R$.
Set $Y=\Spec(R)$, and denote by $Y^{(i)}$ the set of points of $Y$ of codimension $i$, i.e., of 
prime ideals of $R$ of height $i$.  
Let $Q$ denote the fraction field $Q(R)$ of $R$, and denote by $\eta$ the generic point of 
$Y$. 

For $m\geq 0$, we set
$$
 Z^m(Y)=\bigoplus\nolimits_{y\in Y^{(m)}} \Z\cdot y,
$$
the right-hand side being the free abelian group generated by $Y^{(m)}$.

 Consider the  Grothendieck group $\rK'^{(m)}_0(R)=\rK'^{(m)}_0(Y)$ of bounded complexes $\E^\bu$  of finitely generated $R$-modules
which are  exact in codimension less than $m$,   as defined for example in \cite[I.3]{SouleAbramovich}. This is generated by
classes $[\E^\bu]$ of such complexes with relations given by
\begin{itemize}
\item[(i)] $[\E^\bu]=[\calF^\bu]$ if there is a quasi-isomorphism $\E^\bu\xrightarrow{\sim} \calF^\bu$,

\item[(ii)] $[\E^\bu]=[\calF^\bu]+[\calG^\bu]$ if there is an exact sequences of complexes
\[
0\to \calF^\bu\to \E^\bu\to \calG^\bu\to 0.
\]
\end{itemize}
If $M$ is a finitely generated $R$-module with support of codimension at least $m$, 
we regard it as a complex with only nonzero term $M$ at degree $0$.

Suppose that $\mathcal C^\bullet$ is a   bounded complex   of finitely generated  $R$-modules
which is  exact in codimension less than $m$. Then for each $y\in Y^{(m)}$, we can consider 
the complex of $R_y$-modules given by the localization $\mathcal C^\bullet_y=\calC^\bu\otimes_R R_y$.
The assumption on $\mathcal C^\bullet$ implies that all the homology groups
$ {\rm H}^i(\mathcal C^\bullet_y)$ 
are  $ R_y$-modules of finite length and $ {\rm H}^i(\mathcal C^\bullet_y)=0$ for all but a finite number 
of $y\in Y^{(m)}$. We set
$$
c_m(\mathcal C^\bullet)_y=\sum_{i}(-1)^i{\mathrm {length}}_{ R_y} {\rm H}^i(\mathcal C^\bullet_y)
$$
and 
$$
 c_m(\mathcal C^\bullet)=\sum_{y\in Y^{(m)}}c_m(\mathcal C^\bullet)_y\cdot y\in \rZ^m(Y).
$$

We can easily see that $c_m(\mathcal C^\bullet) $ only depends on the class $[\mathcal C^\bullet]$ 
in $\rK'^{(m)}_0(Y)$ and that it is additive, which is to say that it gives a group homomorphism 
$$
c_m \colon \rK'^{(m)}_0(Y)\to {\rm Z}^m(Y).
$$
The element $c_m(\mathcal C^\bullet)$ can also be thought of as a localized $m$th Chern class of $\mathcal C^\bullet$.
In particular, if $M$ is a finitely generated $ R$-module which is supported in codimension $\geq m$, then we have 
 $$
 c_m(M)=\sum_{y\in Y^{(m)}}  {\rm length}_{R_y}(M_y)\cdot y .
$$
In \cite{RobertsBook}, the element $c_m(M)$ is called the codimension-$m$ cycle associated to $M$ and is
denoted by $[M]_{\dim(R)-m}$.
The class $c_m$ can also be given as a very special case of the construction in \cite[Chapter 18]{FultonIntersectionTheory}. 

 In what follows,  we will show how to produce   elements of $\rZ^m(Y)$
 starting from elements in   $\rK_m( Q)$.

\subsection{Tame symbols and Parshin chains}
\label{ss:tamesymbols}

Suppose that $R$ is a discrete valuation ring with maximal ideal $\mathfrak m$,
fraction field $Q$ and residue field $k$.  Then, for all $m\geq 1$, 
the localization sequence of \cite[Theorem 5]{Quillen} produces connecting homomorphisms
$$
\partial_m \colon  \rK_m(Q)\to \rK_{m-1}(k).
$$
We will call these homomorphisms $\partial_m$ ``tame symbols''. 

If $m=1$, then $\partial_1(f)={\rm val}(f)\in \rK_0(k)=\Z$.
If $m=2$, then by Matsumoto's theorem, all elements in $\rK_2(Q)$ are 
finite sums of Steinberg symbols $\{f, g\}$ with $f, g\in Q^\times$ (see \cite{MilnorK}).  
We have
\begin{equation}\label{eq:tamesymbol}
\partial_2(\{f, g\})=(-1)^{{\rm val}(f){\rm val}(g)}\,\frac{f^{{\rm val}(g)}}{g^{{\rm val}(f)}}\, {\rm mod}\,\mathfrak m\, \in k^\times
\end{equation}
(see for example \cite{GraysonTame}, Cor. 7.13). In this case, by \cite{DennisStein}, localization gives a short exact sequence
\begin{equation}\label{eq:exactK2}
1\to \rK_2(R)\to \rK_2(Q)\xrightarrow{\partial_2}k^\times\to 1.
\end{equation}
 This exactness is a special case of Gersten's conjecture: see Subsection \ref{ss:gersten}.

  In what follows, we denote by $\eta_i$ a point in $Y^{(i)}$, i.e., a prime ideal of codimension $i$.
 Suppose that $\eta_i$ lies in the closure $\overline {\{\eta_{i-1}\}}$, so $\eta_i$ contains $\eta_{i-1}$, and consider $R/(\eta_{i-1})$. This is a local integral domain with fraction field $k(\eta_{i-1})$, and $\eta_i$ defines a height $1$ prime ideal in
 $R/(\eta_{i-1})$. The localization $R_{\eta_{i-1},\eta_{i}}=( R/(\eta_{i-1}))_{\eta_i}$ is a $1$-dimensional local ring with fraction field $k(\eta_{i-1})$ and residue field $k(\eta_{i})$. The localization sequence in $\rK'$-theory applied to $ R_{\eta_{i-1},\eta_{i}}$ still gives a connecting homomorphism
  $$
 \partial_{m}(\eta_{i-1},\eta_i) \colon \rK_m(k(\eta_{i-1}))\xrightarrow{ }\rK_{m-1}(k(\eta_i)).
  $$ 
  
  For $m=1$, by \cite[Lemma 5.16]{Quillen} (see also Remark 5.17 therein), or by \cite[Corollary 8.3]{GraysonTame}, the homomorphism
  $
  \partial_1(\eta_{i-1},\eta_i) \colon k(\eta_{i-1})^\times\to \Z
  $ is equal to ${\rm ord}_{\eta_i} \colon k(\eta_{i-1})^\times\to \Z$ where ${\rm ord}_{\eta_i}$ is the unique
  homomorphism with
  $$
 {\rm ord}_{\eta_i}(x)={\rm length}_{ R_{\eta_{i-1},\eta_i}}( R_{\eta_{i-1},\eta_i}/(x))
  $$
  for all $x\in   R_{\eta_{i-1},\eta_i}-\{0\}$.

  For any $n\geq 1$, we now consider the
 set $P_n(Y)$ of ordered sequences of points of $Y$ of the form $\und\eta=(\eta_0,\eta_1,\ldots, \eta_{n})$, with ${\rm codim}(\eta_i)=i$ and $\eta_i \in \overline{\{\eta_{i-1}\}}$, for all $i$. Such sequences are examples of ``Parshin chains'' \cite{ParshinCrelle}.  For $\und\eta=(\eta_0,\eta_1,\ldots, \eta_{n})\in P_n(Y)$, we define a homomorphism
 $$
\nu_{\und\eta} \colon \rK_n(Q)=\rK_n(k(\eta_0))\to \Z=\rK_0(k(\eta_n))
$$
as the composition of successive symbol maps:
\begin{multline*}
 \nu_{\und\eta}=\partial_{1}(\eta_{n-1},\eta_n)\circ\cdots \circ\partial_{n-1}(\eta_1, \eta_2)
 \circ \partial_{n}(\eta_0,\eta_1) \colon \\
\rK_n(k(\eta_0))\to \rK_{n-1}(k(\eta_{1}) )\to\cdots \to \rK_1(k(\eta_{n-1}))\to \rK_0(k(\eta_{n})).
\end{multline*}
Using this, we can define a homomorphism
 $$
 \nu_m \colon \bigoplus_{\und\eta'\in P_{m-1}(Y)} \rK_m(Q)\to \rZ^m(Y)=\bigoplus_{y\in Y^{(m)}}\Z\cdot y
 $$
 by setting the component of  $\nu_m( (a_{\und\eta'})_{\und\eta'})$ for $a_{\und\eta'}\in \rK_m(Q)$ that corresponds to $y\in Y^{(m)}$ to be the sum
 \begin{equation} \label{eq:summaps}
\nu_m( (a_{\und\eta'})_{\und\eta'})_y=\sum_{\und\eta' \ |\ y\in \overline{\{ \eta'_{m-1}\}}}\nu_{\und\eta'\cup y} (a_{\und\eta}).
 \end{equation}
 Here, we set 
 $$
 	\und\eta'\cup y=(\eta_0',\eta_1',\ldots, \eta_{m-1}')\cup y=(\eta_0',\eta_1',\ldots, \eta_{m-1}',  y).
$$ 
Only a finite number of terms in the sum are nonzero.
 \medskip

 For the remainder of the section, we assume that $R$ is in addition regular.
 
 For $m=1$, the map $\nu_m$ amounts to
 $$
 \nu_1 \colon \rK_1( Q)={Q}^\times\xrightarrow{} \rZ^1(Y)=\bigoplus_{y\in Y^{(1)}}\Z\cdot y
 $$
sending $f\in {Q}^\times$ to its divisor ${\rm div}(f)$.
 Since $ R$ is regular, it is a UFD, and $\nu_1$ gives 
  an isomorphism
 \begin{equation}\label{eq:adelicCH1}
 {\rm div}\colon {Q}^\times/{  R}^\times\xrightarrow{\sim} \rZ^1(Y).
 \end{equation}
 
 For $m=2$, the map
 $$
 \nu_2\colon \bigoplus_{\eta_1\in Y^{(1)}}\rK_2(Q)\to \rZ^2(Y)= \bigoplus_{y\in Y^{(2)}}\Z\cdot y,
 $$
 satisfies
 $$
 \nu_2(a)=\sum_{\eta_1\in Y^{(1)} }{\rm div}_{\eta_1}(\partial_2(a_{\eta_1})) 
 $$
for $a=(a_{\eta_1})_{\eta_1}$ with $a_{\eta_1}\in \rK_2(Q)$.
 Here, 
 $$
 {\rm div}_{\eta_1}(f)=\sum_{y\in \overline{\{\eta_1\}}}{\rm ord}_{y}(f)\cdot y
 $$
  is the divisor of the function $f\in k(\eta_1)^\times$ on $\overline{\{\eta_1\}}$. 
  \quash{Then $
 {\rm div}_{\eta_1}(f)$ is also equal to  the push down 
by $\widetilde{\bar\eta}_1\to \bar\eta_1\subset Y$ of the divisor of the function $f$
on the normalization $\widetilde{\bar\eta}_1$ of the closure $\bar\eta_1=\Spec(\scr R/(\eta_1))$. (See 
\cite[1.2]{FultonIntersectionTheory}.) }

\subsection{Tame symbols and Gersten's conjecture} 
\label{ss:gersten}

In this paragraph we suppose that Gersten's conjecture is true for $\rK_2$ and the 
integral regular local ring $ R$. By this, we mean that we assume that the sequence
\begin{equation}\label{eq:GConj}
1\to \rK_2( R)\to \rK_2(Q)\xrightarrow{\vartheta_2} \bigoplus_{\eta_1\in Y^{(1)}}
 k(\eta_1)^\times\xrightarrow{\vartheta_1} \bigoplus_{\eta_2\in Y^{(2)}}\Z\to 0
\end{equation}
is exact, where the component of $\vartheta_2$ at $\eta_1$ is the connecting homomorphism
$$
	\partial_2(\eta_0,\eta_1)\colon \rK_2(Q)\to \rK_1(k(\eta_1))=k(\eta_1)^\times,
$$ 
and $\vartheta_1$
has components $\partial_1(\eta_1,\eta_2)={\rm ord}_{\eta_2}\colon k(\eta_1)^\times\to \Z$.

The sequence \eqref{eq:GConj} is exact when the integral regular local ring $ R$  is a DVR
by Dennis-Stein \cite{DennisStein}, when $ R$ is essentially of finite type over a field 
by Quillen \cite[Theorem 5.11]{Quillen}, when $ R$ is essentially of finite type and smooth over a mixed characteristic DVR by Gillet-Levine \cite{GilletLevine} and Bloch \cite{BlochonGersten}, and when $ R=A\lps t_1,\ldots,  t_r\rps$ is a formal power series ring over a complete DVR $A$ by work of Reid-Sherman \cite{ReidSherman}. In these last two cases, by
examining the proof of \cite[Corollary 6]{GilletLevine} (see also  \cite[Corollary 3]{ReidSherman}), one sees that 
the main theorems  of \cite{GilletLevine} and \cite{ReidSherman} allow one to reduce  the proof 
to the case of a DVR.

By the result of Dennis and Stein quoted above for the DVR ${ R}_{\eta_1}$, we 
also have
\begin{equation}\label{eq:DS2}
1\to \rK_2( R_{\eta_1})\to \rK_2(Q)\xrightarrow{\partial_2} k(\eta_1)^\times\to 1.
\end{equation}
Continuing to assume \eqref{eq:GConj} is exact,
we then obtain that $\vartheta_1$ induces an isomorphism
\begin{equation}
\frac{\bigoplus_{\eta_1\in Y^{(1)}}
 k(\eta_1)^\times}{\vartheta_2(\rK_2(Q))}\xrightarrow{\sim } \rZ^2(Y).
\end{equation}
Combining this with \eqref{eq:DS2}, we obtain an isomorphism
\begin{equation}\label{eq:adelicCH2}
\bar\nu_2\colon \frac{\prod'_{\eta_1\in Y^{(1)}} \rK_2(Q)}{\rK_2(Q)\cdot \prod_{\eta_1\in Y^{(1)}}\rK_2( R_{\eta_1})}\xrightarrow{\ \sim\ } \rZ^2(Y)=\bigoplus_{\eta_2\in Y^{(2)}} \Z\cdot \eta_2
\end{equation}
where the various terms are as in the following paragraph.

In the numerator, the restricted product $\prod'_{\eta_1\in Y^{(1)}} \rK_2(Q)$ is the subgroup of the direct product 
in which all but a finite number of components
 belong to $\rK_2(  R_{\eta_1})$. In the denominator, we have the product of the subgroups $\prod_{\eta_1\in Y^{(1)}}\rK_2( R_{\eta_1})$ and  $\rK_2(Q)$, the second group embedded diagonally in $\prod'_{\eta_1\in Y^{(1)}} \rK_2(Q)$.
Note that by the description of elements in $\rK_2(Q)$ as symbols, this diagonal embedding of $\rK_2(Q)$ lies in the restricted product.
The map giving the isomorphism is obtained by 
\begin{equation}
 \nu_2\colon \prod'_{\eta_1\in Y^{(1)}} \rK_2(Q)\to \rZ^2(Y),
\end{equation}
which is defined by summing the maps
\[
\nu_{(\eta_0,\eta_1, \eta_2)}= \partial_1(\eta_1, \eta_2)\circ \partial_2(\eta_0, \eta_1)\colon \rK_2(Q)\to \Z
\]
 as in \eqref{eq:summaps}.  The map $\nu_2$ is well-defined on the restricted product since $\nu_{(\eta_0,\eta_1, \eta_2)}$ is trivial on $\rK_2( R_{\eta_1})$, 
and it makes sense independently of assuming that \eqref{eq:GConj} is exact.

\subsection{Characteristic symbols}
\label{ss:characteristicsymbols}

Suppose that $R$ is a local integral Noetherian ring and  that
 $\calC^\bullet$ is a complex of finitely generated  
$ R$-modules which is exact on codimension $m-1$. 
We can then consider the $m$th  localized  Chern class, as defined in Subsection \ref{ss:generalsetup}
$$
c_m(\calC^\bullet)\in \rZ^m(Y)=\bigoplus_{\eta_m\in Y^{(m)}}\Z\cdot \eta_m.
$$

\begin{definition}
An element  $(a_{\und\eta'})_{\und\eta'}\in \bigoplus_{\und\eta'\in P_{m-1}(Y)} \rK_m(Q)$
such that 
$$
\nu_m((a_{\und\eta'})_{\und\eta'})=c_m(\calC^\bullet)
$$
in $\rZ^m(Y)$ will be called an {\sl $m$th characteristic symbol} for $\calC^\bullet$.
\end{definition}

If $m$ is the smallest integer such that $\calC^\bullet$ is exact on codimension less than $m$, we will
simply say that $(a_{\und\eta'})_{\und\eta'}$ as above is a  characteristic symbol.

\subsection{First and second Chern classes and characteristic symbols}
\label{ss:firstclass}

We now assume that the 
integral Noetherian local ring $R$ is, in addition, regular.

Suppose first that $m=1$, and let $\calC^\bullet$ be a complex of finitely generated 
$R$-modules which is exact on codimension $0$, which is to say that $\calC^\bullet\otimes_{ R}{Q}$ is exact. 
We can then consider the first    Chern class
$
c_1(\calC^\bullet)\in \rZ^1(Y).
$ 
By \eqref{eq:adelicCH1}, we have $\rZ^1(Y)\simeq {Q}^\times/{ R}^\times$
given by the divisor map. In this case, a first characteristic symbol 
(or characteristic element) for $\calC^\bullet$
is an element $f\in {Q}^\times$ such that
$$
{\rm div}(f)= c_1(\calC^\bullet).
$$
This extends the classical notion of a characteristic power series of a torsion module 
in Iwasawa theory, considering the module as a complex of modules supported in degree zero.

In fact, let $M$ be a finitely generated torsion $R$-module. Let $\mathscr{P}$ be a set of representatives in $R$ 
for the equivalence classes of irreducibles under multiplication by units so that
$\mathscr{P}$ is in bijection with the set of height $1$ primes $Y^{(1)}$.  For 
each $\pi \in \mathscr{P}$, let $n_{\pi}(M)$ be the length of the localization of $M$
at the prime ideal of $R$ generated by $\pi$. 
Then $c_1(M)=\sum_{\pi\in \mathscr{P}}n_\pi\cdot (\pi)$. 
In the sections that follow,
we will also use the symbol $c_1(M)$ to denote 
 the ideal generated by $\prod_{\pi \in \mathscr{P}} \pi^{n_{\pi}(M)}$; this should not lead to confusion.  Note that, with this notation, $M$ is pseudo-null if and only if $c_1(M) = R$.  If $R = \zp\ps{t} = \zp\ps{\zp}$, then $c_1(M)$
 is just the usual characteristic ideal of $R$. This explains 
 the statements in the introduction connecting the growth rate in \eqref{eq:growthrate} to first Chern classes
 (e.g., via the proof of Iwasawa's theorem in \cite[Theorem 13.13]{Washington}).
 
Suppose now that $m=2$.
Let $\calC^\bullet$ be a complex of finitely generated 
$ R$-modules which is exact on codimension $\leq 1$. 
We can then consider the second localized Chern class
$ 
c_2(\calC^\bullet)\in \rZ^2(Y).
$ 
In this case, we can also consider characteristic symbols in a restricted product of $\rK_2$-groups.
An element $(a_{\eta_1})_{\eta_1}\in \prod'_{\eta_1\in Y^{(1)}} \rK_2(Q)$
is a second characteristic symbol for $\calC^\bullet$ when we have
$$
\nu_2((a_{\eta_1})_{\eta_1})=c_2(\calC^\bullet).
$$

 \begin{proposition} \label{prop:2ndsymb}
Suppose that $f_1, f_2$ are two prime elements in   $R$. Assume
that $f_1/f_2$ is not a unit of $ R$. Then a second characteristic 
symbol of the $ R$-module ${ R}/(f_1, f_2)$ is given by $a=(a_{\eta_1})_{\eta_1}$ 
with
\begin{equation*}
  a_{\eta_1}=\begin{cases}
\{f_1, f_2\}^{-1},&  \mbox{\rm if\ } \eta_1=(f_1),\\
  1, & \mbox{\rm if\ } \eta_1\neq (f_1).  
  \end{cases}
  \end{equation*}
\end{proposition}

\begin{proof} Notice that, under our assumptions,  ${ R}/(f_1, f_2)$
is supported on codimension $2$.
We have to calculate the image of the Steinberg symbol $\{f_1, f_2\}$ under
$$
\rK_2(Q)\xrightarrow{\partial_2}\rK_1(k(\eta_1))\xrightarrow{{\rm div}_{\eta_2}} \Z
$$
for $\eta_1=(f_1)$ and  $\eta_2\in \overline{\{\eta_1\}}$. (The rest of the contributions to 
$\nu_2((a_{\eta_1})_{\eta_1})$ are obviously trivial.)
We have ${\rm val}_{\eta_1}(f_1)=1$, ${\rm val}_{\eta_1}(f_2)=0$, and so
$$
\partial_2(\{f_1, f_2\})= f_2^{-1}\, {\rm mod}\, (f_1)\in k(\eta_1)^\times.
$$
By definition,
$$
{\rm div}_{\eta_2}(f_2)={\rm length}_{ R'} {(R'/f_2R')},
$$
where  $ R'$ is the localization $  R_{(f_1), \eta_2}=( R/(f_1))_{\eta_2}$.
We have a surjective homomorphism of local rings $ R_{\eta_2}\to  R'$. The $ R_{\eta_2}$-module structure on $
({ R}/(f_1, f_2))_{\eta_2}= R'/f_2\,  R'$ factors through 
$ R_{\eta_2}\to  R'$, so
$$
{\rm length}_{  R'} {(  R'/f_2\, R')} ={\rm length}_{ R_{\eta_2}}(( R/(f_1, f_2))_{\eta_2}).
$$
This, taken  together with the definition of $c_2( R/(f_1, f_2))$, completes the proof.
\end{proof}
 
 \begin{remark} \label{rem:symmetry}
 The same argument shows that a second characteristic 
symbol of the $ R$-module ${ R}/(f_1, f_2)$ is also given (symmetrically) by 
$a'=(a'_{\eta_1})_{\eta_1}$ with 
$ a'_{\eta_1}=\{f_2, f_1\}^{-1} = \{f_1, f_2\}$  if $\eta_1=(f_2)$, and $a'_{\eta_1}=1$ otherwise.
We can actually see directly that the difference $a-a'\in \prod'_{\eta_1\in Y^{(1)}} \rK_2(Q)$
lies in the denominator of the right-hand side of \eqref{eq:adelicCH2}. Indeed, $a-a'$ is equal modulo 
 $\prod_{\eta_1\in Y^{(1)}}\rK_2( R_{\eta_1})$ to
the image of $\{f_1, f_2\}^{-1}\in \rK_2(Q)$ under the diagonal embedding 
$\rK_2(Q)\to \prod'_{\eta_1\in Y^{(1)}} \rK_2(Q)$.
 \end{remark}
 
\section{Some conjectures in Iwasawa theory}
\label{s:conj}

\subsection{Main conjectures}
\label{ss:mainconjectures}

In this subsection, we explain the relationship between the first Chern class (i.e., the case $m=1$ in Section \ref{s:chern}) and main conjectures of Iwasawa theory. First we strip the main conjecture of all its arithmetic content and present an abstract formulation. To make things concrete, we then give two examples. 

For the ring $R$, we take the Iwasawa algebra $\Lambda = \O\ps{\Gamma}$ of the group $\Gamma= \mathbb{Z}_p^r$ for a prime $p$, 
where $\O$ is the valuation ring of a finite extension of $\mathbb{Q}_p$.  That is,
$\Lambda = \varprojlim_U \O[\Gamma/U]$, where $U$ ranges over the open subgroups of $\Gamma$.  In this case, $\Lambda$ is non-canonically isomorphic to $\O\ps{t_1, \ldots, t_r}$, the power series ring in $r$ variables over $\mc{O}$.  We need two ingredients to formulate a ``main conjecture":
\begin{itemize}
\item[(i)] a complex of $\Lambda$-modules $\mathcal{C}^{\bullet}$ quasi-isomorphic to a bounded complex of finitely generated free $\Lambda$-modules that is exact in codimension zero,
and 
\item[(ii)] a subset $\{a_{\rho}\colon \rho \in \Xi\} \subset \overline{\mathbb{Q}}_p$ for a dense set $\Xi$ of continuous characters of $\Gamma$.  
\end{itemize}
Note that every continuous $\rho \colon \Gamma \rightarrow \overline{\mathbb{Q}}_p^{\times}$ induces a homomorphism $\Lambda \rightarrow \overline{\mathbb{Q}}_p$ that can be extended to a map $Q =Q(\Lambda)\rightarrow \overline{\mathbb{Q}}_p \cup \{\infty\}$. We denote this by $\zeta \mapsto \zeta(\rho)$ or by $\zeta \mapsto \int_{\Gamma} \rho d\zeta$. 
A  main conjecture for the data in (i) and (ii) above is the following statement.

\medskip \noindent
{\bf  Main Conjecture for $\mathcal{C}^\bullet$ and $\{a_\rho\}$.} There is an element $\zeta \in Q^\times$ such that 
\begin{itemize}
\item[(a)] $\zeta(\rho) = a_{\rho}$ for all $\rho \in \Xi$,
\item[(b)]  $\zeta$ is a characteristic element for $\mathcal{C}^\bu$, i.e., $c_1(\mathcal{C}^{\bullet}) = {\rm div}(\zeta)$.
\end{itemize}

\medskip \noindent
Here, the Chern class $c_1$ and the divisor ${\rm div}$ are as defined in Section \ref{s:chern}.

\subsection{The Iwasawa main conjecture over a totally real field} 
\label{ss:mainReal}

Let $E$ be a totally real number field.  Let $\chi$ be an even one-dimensional character of the absolute Galois group of $E$ of finite order, and let $E_{\chi}$ denote the fixed field of its kernel.  For a prime $p$ which we take here to be odd, we then set $F = E_{\chi}(\mu_p)$ and $\Delta = \Gal(F/E)$.
We assume that the order of $\Delta$ is prime to $p$.  We denote the cyclotomic $\mathbb{Z}_p$-extension of $F$ by $K$.  Then $\Gal(K/E) \cong \Delta \times \Gamma$, where $\Gamma \cong \mathbb{Z}_p$.  If Leopoldt's conjecture holds for $E$ and $p$, then $K$ is the only $\mathbb{Z}_p$-extension of $F$ abelian over $E$. 
Let $L$ be the maximal abelian unramified pro-$p$ extension of $K$. Then $\Gal(K/E)$ acts continuously on $X=\Gal(L/K)$, as there is a short exact sequence 
\[
	1 \rightarrow \Gal(L/K) \rightarrow \Gal(L/E) \rightarrow \Gal(K/E) \rightarrow 1.
\]
Thus $X$ becomes a module over the Iwasawa algebra $\mathbb{Z}_p\ps{\Gal(K/E)}$. For a character $\psi$ of $\Delta$, define $\O_{\psi}$ to be the $\mathbb{Z}_p$-algebra generated by the values of $\psi$.
The $\psi$-eigenspace 
$$
	X^{\psi} = X \otimes_{\mathbb{Z}_p[\Delta]} \O_{\psi}
$$ 
is a module over $\Lambda_{\psi} = \O_{\psi}\ps{\Gamma}$. By a result of Iwasawa, $X^{\psi}$ is known to be a finitely generated torsion $\Lambda_{\psi}$-module.  

On the other side, we let $\Xi = \{\chi_{\rm cyc}^k \mid k \le 0\}$, where $\chi_{\rm cyc}$ is the $p$-adic cyclotomic character on $\Gamma$. Define
\[
a_{\chi_{\rm cyc}^k} = L(\chi \omega^{k-1}, k) \prod_{\mathfrak{p} \in S_p} (1-\chi\omega^{k-1}(\mathfrak{p})N\mathfrak{p}^{-k}),
\]
where $\omega$ is the Teichm\"{u}ller character, $S_p$ is the set of primes of $E$ above $p$, $N\mathfrak{p}$ is the norm of $\mathfrak{p}$, and $L(\chi \omega^{k-1}, s)$ is the complex $L$-function of $\chi \omega^{k-1}$. Then we have the following Iwasawa main conjecture \cite{Wiles}.

\begin{theorem}[Barsky, Cassou-Nogu\`{e}s, Deligne-Ribet, Mazur-Wiles, Wiles] \label{cyclmc} There is a unique $\mathcal{L} \in Q^{\times}$ such that 
\begin{itemize}
\item[(a)] $\mathcal{L}(\chi_{\cyc}^k) = a_{\chi_{\cyc}^k}$ for every even positive integer $k$,
\item[(b)] $c_1(X^{ \chi^{-1}\omega }) = {\rm div}(\mathcal{L})$.
\end{itemize}
\end{theorem}

\subsection{The two-variable main conjecture over an imaginary quadratic field} 
\label{ss:twovariablemain}

We assume that $p$ is an odd prime that splits into two primes $\mf{p}$ and $\bar{\mf{p}}$ in the imaginary quadratic field $E$. 
Fix an abelian extension $F$ of $E$ of order prime to $p$. Let $K$ be the unique abelian extension of $E$ such that $\Gal(K/F) \cong \mathbb{Z}_p^2$. Let $\Delta= \Gal(F/E)$ and $\Gamma = \Gal(K/F)$. Then we have a canonical isomorphism $\Gal(K/E) \cong \Delta \times \Gamma$. 
Let $\mf{X}_{\mf{p}}$ (resp., $\mf{X}_{\overline{\mf{p}}}$) be the Galois group over $K$ of the  maximal abelian pro-$p$ extension of $K$ unramified outside $\mathfrak{p}$ (resp., $\overline{\mathfrak{p}}$). 
Then, as above, $\mathfrak{X}_{\mathfrak{p}}$ and $\mathfrak{X}_{\overline{\mathfrak{p}}}$ become modules over $\mathbb{Z}_p\ps{\Delta \times \Gamma}$. It is proven in \cite[Theorem 5.3(ii)]{Rubin1} that $\mathfrak{X}_{\mathfrak{p}}$ and $\mathfrak{X}_{\overline{\mathfrak{p}}}$ are finitely generated torsion $\mathbb{Z}_p\ps{\Delta \times \Gamma}$-modules. As in Subsection \ref{ss:mainReal}, for any character $\psi$ of $\Delta$, we let $\O_{\psi}$ be the extension of $\mathbb{Z}_p$ obtained by adjoining values of $\psi$ and let
\[
\mathfrak{X}_{\mathfrak{p}}^{\psi} = \mathfrak{X}_{\mathfrak{p}} \otimes_{\mathbb{Z}_p[\Delta]} \O_{\psi},\quad
\mathfrak{X}_{\overline{\mathfrak{p}}}^{\psi} = \mathfrak{X}_{\overline{\mathfrak{p}}} \otimes_{\mathbb{Z}_p[\Delta]} \O_{\psi}.
\]

The other side takes the following analytic data:  Let $\Xi_{\psi, \mathfrak{p}}$ (resp., $\Xi_{\psi, \overline{\mathfrak{p}}}$) be the set of all Gr\"ossencharacters of $E$ factoring through $\Delta \times \Gamma$ of infinity type $(k,j)$ (resp., $(j,k)$), with $j \le 0 < k$  and with restriction to $\Delta$ equal to $\psi$. Let $\mathfrak{g}$ be the conductor of $\psi$. Let $-d_E$ be the discriminant of $E$. For $\chi \in \Xi_{\psi, \mathfrak{p}}$ (resp., $\chi \in \Xi_{\psi, \overline{\mathfrak{p}}}$) of infinity type $(k,j)$, let
\[
a_{\chi, \mathfrak{p}} = \frac{\varOmega^{j-k}}{\varOmega_p^{j-k}} \left( \frac{\sqrt{d_E}}{2\pi}\right)^j G(\chi) \left(1 - \frac{\chi(\mathfrak{p})}{p}\right) L_{\infty, \mathfrak{g}\overline{\mathfrak{p}}}(\chi^{-1}, 0). 
\]
\[
\Big( \text{resp., } a_{\chi, \mathfrak{\overline{p}}}  = \frac{\varOmega^{j-k}}{\varOmega_p^{j-k}} \left( \frac{\sqrt{d_E}}{2\pi}\right)^j G(\chi) \left(1 - \frac{\chi(\mathfrak{\overline{p}})}{p}\right) L_{\infty, \mathfrak{g}{\mathfrak{p}}}(\chi^{-1}, 0) \Big).
\]
Here, $\varOmega$ and $\varOmega_p$ are complex and $p$-adic periods of $E$, respectively, and $G(\chi)$ is a Gauss sum. Moreover, $L_{\infty, \mathfrak{f}}$ refers to the $L$-function with the Euler factor at $\infty$ but without
the Euler factors at the primes dividing $\mathfrak{f}$.
(For more explanation, see \cite[Equation (36), p. 80]{deShalit}.) Let $W$ be the ring of Witt vectors of $\overline{\mathbb{F}}_p$. Using work of Yager, deShalit proves in \cite[Theorem 4.14]{deShalit} that there are $\mathcal{L}_{\mathfrak{p},\psi }, \mathcal{L}_{\mathfrak{\overline{p}},\psi} \in W\ps{\Gamma}$ such that  
\begin{eqnarray*}
\mathcal{L}_{\mathfrak{p},\psi} (\chi) = a_{\chi, \mathfrak{p}}, \text{   for every   } \chi \in \Xi_{\psi, \mathfrak{p}},
&\mr{and} &
\mathcal{L}_{\mathfrak{\overline{p}},\psi}(\chi) = a_{\chi, \mathfrak{\overline{p}}}, \text{   for every   } \chi \in \Xi_{\psi, \mathfrak{\overline{p}}}.
\end{eqnarray*}

We have the following result of Rubin \cite{Rubin1} on the two-variable main conjecture over $E$.

\begin{theorem}[Rubin] \label{thm:imquadmc} With the notation as above, we have
$${\rm div}(\mathcal{L}_{\mathfrak{p},\psi}) = c_1(W\ps{\Gamma} \cotimes{\O_{\psi}\ps{\Gamma}} \mathfrak{X}_{\mathfrak{p}}^{\psi}).$$
The above is also true with $\mathfrak{p}$ replaced by $\overline{\mathfrak{p}}$. 
\end{theorem}

Let $\sigma$ denote the nontrivial element of $\Gal(E/\mathbb{Q})$. We obtain an action of $\sigma$ on $\Delta \times \Gamma$ via conjugation by any lift of $\sigma$ to $\Gal(K/\mathbb{Q})$. We extend this action $\mathbb{Z}_p$-linearly to a map
\[
\sigma \colon \mathbb{Z}_p\ps{\Delta \times \Gamma} \rightarrow \mathbb{Z}_p\ps{\Delta \times \Gamma}.
\]
This homomorphism $\sigma$ maps $\O_{\psi}\ps{\Gamma}$ isomorphically to $\O_{\psi \circ \sigma}\ps{\Gamma}$. 

\begin{lemma} The two Katz $p$-adic $L$-functions are related by 
\label{relationlplpbar}
\begin{itemize}
\item[(a)]
$\mathcal{L}_{\mathfrak{\overline{p}},\psi} = \sigma(\mathcal{L}_{{\mathfrak{p}},\psi\circ\sigma}).$
\item[(b)]
$
\mathcal{L} _{\mathfrak{\overline{p}},\psi }(\chi) = (\text{$p$-adic unit}) \cdot \mathcal{L}_{\mathfrak{p}, \psi^{-1}\omega}(\chi^{-1}\chi_{\cyc}),
$
where $\chi_{\cyc}$ is the $p$-adic cyclotomic character on $\Delta \times \Gamma$. 
\end{itemize}
\end{lemma}
\begin{proof} Assertion (a) is proven simply by interpolating both sides at all elements in $\Xi_{\psi, \overline{\mathfrak{p}}}$. We first note that both $\sigma(\mathcal{L}_{\mathfrak{p}, \psi\circ \sigma})$ and $\mathcal{L}_{\overline{\mathfrak{p}},\psi}$ lie in $W\cotimes{\O_{\psi}}\O_{\psi}\ps{\Gamma}$.  Then
\begin{align*}
\sigma(\mathcal{L}_{\mathfrak{p}, \psi \circ \sigma} )(\chi) & = \mathcal{L}_{\mathfrak{p}, \psi \circ \sigma} (\chi \circ \sigma) \\
& = \frac{\varOmega^{j-k}}{\varOmega_p^{j-k}} \left( \frac{\sqrt{d_E}}{2\pi}\right)^j G(\chi \circ \sigma) \left(1 - \frac{(\chi \circ \sigma)(\mathfrak{p})}{p}\right) L_{\infty, \mathfrak{g}\overline{\mathfrak{p}}}((\chi \circ \sigma)^{-1}, 0) \\
& = \frac{\varOmega^{j-k}}{\varOmega_p^{j-k}} \left( \frac{\sqrt{d_E}}{2\pi}\right)^j G(\chi) \left(1 - \frac{\chi(\mathfrak{\overline{p}})}{p}\right) L_{\infty, \mathfrak{g}{\mathfrak{p}}}(\chi^{-1}, 0) \\
& = \mathcal{L}_{\overline{\mathfrak{p}},\psi} (\chi),
\end{align*}
where we use the fact that the infinity type of $\chi \circ \sigma$ is $(k,j)$, and in the third equality we use the obvious equalities $G(\chi) = G(\chi \circ \sigma)$ and $L_{\infty, \mathfrak{g\overline{p}}}((\chi\circ \sigma)^{-1}, 0) = L_{\infty, \mathfrak{g p}}(\chi^{-1}, 0)$.

To prove (b), we use the functional equation for $p$-adic $L$-functions, which says that
\[
\mathcal{L}_{\mathfrak{p},\psi} (\chi) = \text{($p$-adic unit)} \cdot \mathcal{L}_{\mathfrak{p}, \overline{\psi}^{-1}\omega} (\overline{\chi}^{-1}\chi_{\cyc}),
\]
where we write $\overline{\psi}$ and $\overline{\chi}$ instead of $\psi \circ \sigma$ and $\chi \circ \sigma$ for convenience (see \cite[Equation (9), p. 93]{deShalit}). Using (a) and the functional equation, we obtain
\begin{align*}
\mathcal{L}_{\overline{\mathfrak{p}},\psi} (\chi) & = \sigma(\mathcal{L}_{\mathfrak{p}, \overline{\psi}} )(\chi)\\ 
& = \mathcal{L}_{\mathfrak{p}, \overline{\psi}} (\overline{\chi}) \\
& = \text{($p$-adic unit)} \cdot \mathcal{L}_{\mathfrak{p}, \psi^{-1}\omega} (\chi^{-1}\chi_{\cyc}).  
\end{align*} 
\end{proof}

\subsection{Greenberg's Conjecture}
\label{ss:GreenbergsConjecture}

 Let $E$ be an arbitrary number field, and let $\widetilde{E}$ be the compositum of all $\zp$-extensions of $E$.  Let $\Gamma=\Gal(\widetilde{E}/E)$ and $\Lambda=\zp\ps{\Gamma}$.  Then $\Gamma \cong \zp^r$ for some $r \ge r_2(E)+1$, where $r_2(E)$ is the number of complex places of $E$. Leopoldt's Conjecture for $E$ and $p$ is the assertion that $r=r_2(E)+1$. This is known to be true if $E$ is abelian over $\Q$ or over an imaginary quadratic field \cite{Brumer}. The ring $\Lambda$ is isomorphic (non-canonically) to the formal power series ring over $\zp$ in $r$ variables. 

Let $L$ be the maximal abelian, unramified pro-$p$-extension of $\widetilde{E}$, and let $X=\mathrm{Gal}(L/\widetilde{E})$, which is a $\Lambda$-module that we will call  the unramified Iwasawa module over $\widetilde{E}$.  The following conjecture was
first stated in print in \cite[Conjecture 3.5]{Green1}.

\begin{conjecture}[Greenberg]  \label{con:Ralph} With the above notation,  the $\Lambda$-module $X$ is pseudo-null.  That is, its localizations at all codimension $1$ points of $\mathrm{Spec}(\Lambda)$ are trivial.
\end{conjecture}

Note that if $E$ is totally real, and if Leopoldt's conjecture for $E$ and $p$ is valid, then $\widetilde{E}$ is the cyclotomic $\zp$-extension of $E$, and the conjecture states that $X$ is finite.  

In the case that $E$ is totally complex, we have the following reasonable extension of the above conjecture.  Let $K$ be any finite extension of $\widetilde{E}$ which is abelian over $E$. Then $\Gal(K/E) \cong \Delta \times \Gamma$, where $\Delta$ is a finite group and $\Gamma$ (as defined above) is identified with $\Gal(K/F)$ for some finite extension $F$ of $E$.   Let $X$ be the unramified Iwasawa module over $K$.  Then $\Gamma$ acts on $X$ and so we can again regard $X$ as a $\Lambda$-module. The extended conjecture asserts that $X$ is pseudo-null as a $\Lambda$-module. 

Some evidence for Conjecture \ref{con:Ralph} has been given in various special cases.  For instance, in \cite{minardithesis}, Minardi verifies Conjecture \ref{con:Ralph} when $E$ is an imaginary quadratic field and $p$ is a prime not dividing the class number of $E$, and also for many imaginary quadratic fields $E$ when $p=3$ and does divide the class number. In \cite{hubbardthesis},  Hubbard verifies the conjecture when $p=3$ for a number of biquadratic fields $E$.  In \cite{sh-unram}, Sharifi gave a criterion for Conjecture \ref{con:Ralph} to hold for $\Q(\mu_p)$.  By a result of Fukaya-Kato \cite[Theorem 7.2.8]{FKConj} 
on a conjecture of McCallum-Sharifi and related computations in \cite{mcs}, the condition holds for $E =\Q(\mu_p)$ for all $p < 25000$.  The results of \cite{sh-unram} suggest that $X$ should have an annihilator of very high codimension for $E = \Q(\mu_p)$.

One can construct examples of $\zp^r$-extensions $K$ of a suitably chosen number field $F$ such that the unramified Iwasawa module $X$ over $K$ has the ideal $(p)$ in its support as a $\zp\ps{\Gal(K/F)}$-module.  Such examples can be constructed by imitating Iwasawa's construction of $\zp$-extensions with positive $\mu$-invariant. This was pointed out to us by T. Kataoka.   Such a construction can be done for any positive $r$.  However, if one adds the assumption that $K$ contain the cyclotomic $\zp$-extension of $F$, and $r > 1$,   then we actually know of no examples where $X$ is demonstrably not pseudo-null.

Assume that $K$ is a $\Z^r_p$-extension of $F$ with $r \ge 1$ such that $K$ contains all $p$-power roots of unity and, therefore, the cyclotomic $\zp$-extension of $F$.   
The class group $\Cl(K)$ is defined as the direct limit under the obvious maps of the ideal class groups of the finite extensions of $F$ contained in $K$. The assertion that $X$ is pseudo-null as a $\Lambda$-module should conjecturally be equivalent to the assertion that the $p$-primary subgroup $\Cl(K)_p$ of $\Cl(K)$ is actually trivial.    We sketch an argument just in one direction.  One can show that if $X$ is pseudo-null, then so is $\Hom( \Cl(K)_p, \qp/\zp)$,   the Pontryagin dual of $\Cl(K)_p$.   One then employs a standard Kummer theory argument to show that 
$\Hom( \Cl(K)_p, \mu_{p^{\infty}})=\Hom( \Cl(K)_p, \qp/\zp)(1)$ is isomorphic to a $\Lambda$-submodule of $\mf{X} = \Gal(M/K)$, where $M$ denotes the maximal abelian pro-$p$ extension of $K$ unramified outside the primes above $p$, which we call the unramified outside $p$ Iwasawa module over $K$.  One can then use the result that $\mf{X}$ has no nontrivial pseudo-null $\Lambda$-submodules.  That result is a consequence of the fact that $\mf{X}$ has rank $r_2(F)$ as a $\Lambda$-module known as the {\em weak Leopoldt conjecture} for $K/F$, which is satisfied because $K$ contains the cyclotomic $\zp$-extension of $F$. 
(See  \cite[Th\'eoreme 3.1]{nguyen} or   \cite[Proposition 5]{Gr78}.)  For a partial result in the converse direction, see \cite[Th\'eoreme 5.1]{nv}.

If one assumes in addition that the decomposition subgroups of $\Gamma=\Gal(K/F)$ for primes above $p$ are of $\zp$-rank at least $2$, then the assertion that $X$ is pseudo-null is equivalent to the assertion that $\mf{X}$ is torsion-free as a $\Lambda$-module.  This equivalence follows from \cite[Proposition 3.6]{lannuzelnguyen} (see also Remark \ref{secondproof}).  The result in \cite{lannuzelnguyen} is stated in terms of the pseudo-nullity of a certain quotient $X'$ of $X$, namely $X'=\mathrm{Gal}(L'/K)$, where $L'$ is the maximal unramified abelian pro-$p$ extension of $K$ in which all the primes of $K$ lying over $p$ split completely. However, the kernel of the map $X \to X'$ is pseudo-null under our assumption on the decomposition subgroups by Lemma \ref{unramsplit} below. It follows that $X$ is pseudo-null if and only if $X'$ is pseudo-null.

\section{Unramified Iwasawa modules}
\label{s:FieldsWithR2Equal1}

\subsection{The general setup} 
\label{ss:gensetup}

Let $p$ be a prime, $E$ be a number field, $F$ a finite Galois extension of $E$, and $\Delta = \Gal(F/E)$.  Let $K$ be a
Galois
extension of $E$ that is a $\zp^r$-extension of $F$ for some $r \ge 1$, and set $\Gamma = \Gal(K/F)$.  Set $\mc{G} = \Gal(K/E)$, $\Omega = \zp\ps{\mc{G}}$, and $\La = \zp\ps{\Gamma}$.  Note that $K/F$ is unramified outside $p$ as a compositum of $\zp$-extensions.

Let $S$ be a set of primes of $E$ including those over $p$ and $\infty$, and let $S_f$ be the set of finite primes in $S$.  
For any algebraic extension $F'$ of $F$, let $G_{F',S}$ denote the Galois group of the maximal extension $F'_S$ of $F'$ that is unramified outside the primes over $S$.  Let $\mc{Q} = \Gal(F_S/E)$.  
For a compact $\zp\ps{\mc{Q}}$-module $T$, we consider the Iwasawa cohomology group
$$
	\rH^i_{\Iw}(K,T) = \varprojlim_{F' \subset K} \rH^i(G_{F',S},T)
$$ 
that is the inverse limit of continuous Galois cohomology groups under corestriction maps, 
with $F'$ running over the finite extensions of $F$ in $K$.  It has the natural structure of an $\Omega$-module.

We will use the following notation.  
For a locally compact $\La$-module $M$, let us set 
$$
	\rE^i_{\La}(M) = \Ext^i_{\La}(M,\La)
$$ 
for short.  
This again has a $\La$-module
structure with $\gamma \in \Gamma$ acting on $f \in \rE^0_{\La}(M)$ by $(\gamma \cdot f)(m) = f(\gamma^{-1}m) = \gamma^{-1} f(m)$.
We let $M^{\vee}$ denote the Pontryagin dual, 
to which we give a module structure by letting $\gamma$ act by precomposition by $\gamma^{-1}$. 
If $M$ is a (left) $\Omega$-module, then $M^{\vee}$ is likewise a (left) $\Omega$-module.  Moreover, 
$\rE^i_{\La}(M) \cong \Ext^i_{\Omega}(M,\Omega)$ as $\La$-modules
since $\Omega$ is $\La$-projective (cf. \cite[Proposition 5.4.17]{nsw}), through which $\rE^i_{\La}(M)$ acquires an $\Omega$-module structure.
We set $M^* = \rE^0_{\La}(M) = \Hom_{\La}(M,\La)$.

The first of the following two spectral sequences is due to Jannsen \cite[Theorem 1]{jan-spec}, and
the second to Nekov\'a\v{r} \cite[Theorem 8.5.6]{nekovar} (though it is assumed there that $p$ is odd
or $K$ has no real places).  One can find very general versions that imply these in 
\cite[1.6.12]{FukayaKatoConj} and \cite[Theorem 4.5.1]{ls}.  

\begin{theorem}[Jannsen, Nekov\'a\v{r}] \label{thm:ss}
	Let $T$ be a compact $\zp\ps{\mc{Q}}$-module that is finitely generated and free over $\zp$.  
	Set $A = T \otimes_{\zp} \qp/\zp$.
	There are convergent spectral sequences of $\Omega$-modules 
	\begin{eqnarray*}
		&{\rm F}_2^{i,j}(T) = \rE^i_{\La}(\rH^j(G_{K,S},A)^{\vee}) \Rightarrow {\rm F}^{i+j}(T) = \rH^{i+j}_{\Iw}(K,T)&\\
		&\rH_2^{i,j}(T) = \rE^i_{\La}(\rH^{2-j}_{\Iw}(K,T)) \Rightarrow \rH^{i+j}(T) = \rH^{2-i-j}(G_{K,S},A)^{\vee}.&
	\end{eqnarray*}
\end{theorem}

We will be interested in the above spectral sequences in the case that $T = \zp$.
We have a canonical isomorphism
$$
	\rH^1(G_{K,S},\qp/\zp)^{\vee} \cong \mf{X},
$$
where $\mf{X}$ denotes the $S$-ramified Iwasawa module over $K$ (i.e., the Galois group of the maximal
abelian pro-$p$, unramified outside $S$ extension of $K$).  
We study the relationship between $\mf{X}$ and $\rH^1_{\Iw}(K,\zp)$. 

The following nearly immediate consequence of Jannsen's spectral sequence is a mild extension of earlier unpublished results of McCallum \cite[Theorems A and B]{mccallum}.

\begin{theorem}[McCallum, Jannsen] \label{mc}
	There is a canonical exact sequence of $\Omega$-modules 
	$$
		0 \to \zp^{\delta_{r,1}} \to \rH^1_{\Iw}(K,\zp) \to \mf{X}^*
		\to \zp^{\delta_{r,2}} \to \rH^2_{\Iw}(K,\zp), 
	$$
	where $\delta_{j,i} = 1$ if $i = j$ and $\delta_{j,i} = 0$ otherwise.
	If the weak Leopoldt conjecture holds for $K$, which is to say that 
	$\rH^2(G_{K,S},\qp/\zp) = 0$, then this exact sequence extends to
	\begin{equation}
	\label{eq:wantit}
		0 \to \zp^{\delta_{r,1}} \to \rH^1_{\Iw}(K,\zp) \to \mf{X}^*
		\to \zp^{\delta_{r,2}} \to \rH^2_{\Iw}(K,\zp)  \to \rE_{\La}^1(\mf{X}) \to
		\zp^{\delta_{r,3}} \to 0.
	\end{equation}
	If $p$ is odd or $K$ has no real places, then there are canonical isomorphisms 
	$\rE_{\La}^i(\mf{X}) \xrightarrow{\sim} \zp^{\delta_{r,i+2}}$ for $i \ge 2$.
\end{theorem}

\begin{proof}
	The first sequence is just the five-term exact sequence of base terms in Jannsen's spectral
	sequence for $T = \zp$.  For this, we remark that 
	$$
		{\rm F}_2^{i,0}(\zp) = \rE_{\La}^i(\zp) \cong \zp^{\delta_{r,i}}
	$$
	by \cite[Lemma 5]{jannsen} or Corollary \ref{cor:indzp} below.
	Under weak Leopoldt, ${\rm F}_2^{0,2}(\zp)$ 
	is zero, so the exact sequence continues as written, 
	the next term being $\rH^3_{\Iw}(K,\zp)=0$.  If $p$ is odd or $K$ is totally imaginary, then 
	$G_{F',S}$ has $p$-cohomological dimension $2$ for some finite extension $F'$ of $F$ 
	in $K$, so $\rH^j_{\Iw}(K,\zp)$ vanishes for $j \ge 3$, in which case the spectral sequence also yields
	the remaining isomorphisms.
\end{proof}

\begin{remark}
	The weak Leopoldt conjecture for $K$ is well-known to hold in the case that $K(\mu_p)$ contains all
	$p$-power roots of unity (see \cite[Theorem 10.3.25]{nsw}).
\end{remark}

\begin{remark}
	For $p$ odd, McCallum proved everything but the exactness at $\zp^{\delta_{r,2}}$ in Theorem \ref{mc},
	supposing both hypotheses listed therein.
\end{remark}

\begin{cor} \label{dualseq}
	There is a canonical isomorphism $\mf{X}^{**} \to \rH^1_{\Iw}(K,\zp)^*$ of $\Omega$-modules.
\end{cor}

\begin{proof}
	This follows from Theorem \ref{mc}, which provides an isomorphism if $r \ge 3$, or if $r = 2$
	and the map $\mf{X}^* \to \zp$ is zero.  If $r = 1$, then
	we obtain an exact sequence
	$$
		0 \to \mf{X}^{**} \to \rH^1_{\Iw}(K,\zp)^* \to \rE_{\La}^0(\zp),
	$$
	and the last term is zero.  If $r = 2$ and the map $\mf{X}^* \to \zp$ is nonzero, then we obtain
	an exact sequence
	$$
		0 \to \rE_{\La}^0(\zp) \to \mf{X}^{**} \to \rH^1_{\Iw}(K,\zp)^* \to \rE_{\La}^1(\zp),
	$$
	and $\rE_{\La}^0(\zp) = \rE_{\La}^1(\zp) = 0$ since $r = 2$.
\end{proof}

Using the second spectral sequence in Theorem \ref{thm:ss}, we may use this to obtain the following.

\begin{cor}
\label{cor:iwseq}
	Suppose that $p$ is odd or $K$ is totally imaginary.
	There is an exact sequence
	\begin{equation} \label{keyseq}
		0 \to \rE_{\La}^1(\rH^2_{\Iw}(K,\zp)) \to \mf{X} \to \mf{X}^{**} \to \rE_{\La}^2(\rH^2_{\Iw}(K,\zp)) \to \zp
	\end{equation}
	of $\Omega$-modules.
	In particular,
	$\rE_{\La}^1(\rH^2_{\Iw}(K,\zp))$ is isomorphic to the $\La$-torsion submodule of $\mf{X}$.
\end{cor}

\begin{proof} By hypothesis, $G_{F',S}$ has $p$-cohomological dimension $2$ for some finite extension $F'$ of $F$ 
in $K$.  Therefore, Nekov\'a\v{r}'s spectral sequence
	is a first quadrant spectral sequence for any $T$.  For $T = \zp$, it provides an exact sequence
	\begin{multline} \label{ptseq}
		0 \to \rE_{\La}^1(\rH^2_{\Iw}(K,\zp)) \to \rH^1(G_{K,S},\qp/\zp)^{\vee} 
		\to \rH^1_{\Iw}(K,\zp)^* \to \rE_{\La}^2(\rH^2_{\Iw}(K,\zp)) \\
		\to 
		\rH^0(G_{K,S},\qp/\zp)^{\vee} \to \rE_{\La}^1(\rH^1_{\Iw}(K,\zp)) \to \rE_{\La}^3(\rH^2_{\Iw}(K,\zp)) \to 0
	\end{multline}
	of $\Omega$-modules.
	In particular, applying Corollary \ref{dualseq} to get the third term, 
	we have the exact sequence of the statement.  
\end{proof}

\begin{remark}
	In the case that $r = 1$, Corollary \ref{cor:iwseq} is in a sense implicit in the work of Iwasawa
	\cite{iwasawa73} (see 
	Theorem 12 and its proof of Lemma 12).  In this case, second $\Ext$-groups are finite,
	so the map to $\zp$ in the corollary is zero.  
\end{remark}

\begin{remark}
	It is natural to ask how \eqref{keyseq} is related to the more abstract exact sequence of 
	\cite[(1.8.1)]{jannsen} (see also \cite[Prop.~5.4.9]{nsw}).  Under the assumption that $K$ contains all $p$-power roots of unity, 
	we may answer this as follows.  Jannsen defines a functor $\mathrm{D}$ on the homotopy category of $\La$-modules.
	Up to homotopy, $\mathrm{D}\mf{X} = \coker(P_0^* \to P_1^*)$ for a choice of projective resolution 
	\begin{equation} \label{projres}
		0 \to P_{r+1} \to \cdots \to P_1 \to P_0 \to \mf{X} \to 0
	\end{equation}
 	of $\mf{X}$.  
	Taking Ext-groups of the four-term exact sequence defining $\mathrm{D}\mf{X}$
	leads to the exact sequence
	\begin{equation} \label{jannsenD}
		0 \to \rE^1_{\La}(\mathrm{D}\mf{X}) \to \mf{X} \to \mf{X}^{**} \to \rE^2_{\La}(\mathrm{D}\mf{X}) \to 0.
	\end{equation}
	
	By definition, one has an injective $\La$-module homomorphism $\alpha \colon \rE^1_{\La}(\mf{X}) \to \mathrm{D}\mf{X}$
	with cokernel equal to the kernel of $P_2^* \to P_3^*$ (taking $P_3 = 0$ if $r = 1$).
	By Theorem \ref{mc}, the Ext-groups computed 
	by the $\La$-dual of \eqref{projres} satisfy $\rE_{\La}^i(\mf{X}) \cong \zp^{\delta_{i+2,r}}$ for $i \ge 2$.   
	We also have a map $\beta \colon \rH^2_{\Iw}(K,\zp) \to \rE^1_{\La}(\mf{X})$ that by \eqref{eq:wantit} 
	is an isomorphism if $r = 1$ or $r \ge 4$, an injection with cokernel $\zp$ if $r = 3$, and a surjection 
	with kernel a quotient $C$ of $\zp$ if $r = 2$.  Setting $C = 0$ if $r = 1$ and $C = \zp$ if $r \ge 3$,
	it follows from these facts that the map $\rho = \alpha \circ \beta \colon \rH^2_{\Iw}(K,\zp) \to \mathrm{D}\mf{X}$ fits
	in a complex
	\begin{equation} \label{complex}
		\rH^2_{\Iw}(K,\zp) \xrightarrow{\rho} \mathrm{D}\mf{X} \to P_2^* \to \cdots \to P_{r+1}^*
	\end{equation}
	which has cohomology $C$ concentrated in degree $r-2$ counting from $0$.
	
	If $r \ge 3$, or if $r = 2$ and $C = \zp$, the complex \eqref{complex} just described allows us to compute that we have isomorphisms
	\begin{equation} \label{ext_isoms}
		\rE_{\La}^i(\mathrm{D}\mf{X}) \to \rE_{\La}^i(\rH^2_{\Iw}(K,\zp))
	\end{equation}
	for $i \notin \{2,3\}$ and an exact sequence 
	$$
		0 \to \rE_{\La}^2(\mathrm{D}\mf{X}) \to \rE^2_{\La}(\rH^2_{\Iw}(K,\zp)) \to \zp \to \rE^3_{\La}(\mathrm{D}\mf{X}) \to 
		\rE^3_{\La}(\rH^2_{\Iw}(K,\zp)).
	$$
	If $r = 2$ and $C$ is finite, the map of \eqref{ext_isoms} is an isomorphism for $i = 2$ and
	an injection with cokernel $C$ for $i = 3$.  
	In all cases, we obtain a map from the exact sequence \eqref{jannsenD} to the exact sequence \eqref{keyseq} of the form
	$$
		\SelectTips{cm}{} \xymatrix{
			0 \ar[r] & \rE_{\La}^1(\mathrm{D}\mf{X}) \ar[r] \ar[d]_{\wr} \ar[r] & \mf{X} \ar@{=}[d] \ar[r] & \mf{X}^{**}
			\ar@{=}[d] \ar[r] & \rE^2_{\La}(\mathrm{D}\mf{X}) \ar[r] \ar@{^{(}->}[d] & 0\\
			0 \ar[r] & \rE_{\La}^1(\rH^2_{\Iw}(K,\zp)) \ar[r] & \mf{X} \ar[r] & \mf{X}^{**} \ar[r] &
			\rE_{\La}^2(\rH^2_{\Iw}(K,\zp)) \ar[r] & \zp,
		}
	$$
	where the leftmost and rightmost vertical maps are induced by $\pm \rho$.  (We have only checked commutativity up
	to the ambiguous signs, as the remark is not used.)
\end{remark}

\begin{remark} \label{rem:stand_map}
	In Corollary \ref{cor:iwseq}, the map $\mf{X} \to \mf{X}^{**}$ can be taken to be 
	the standard map from $\mf{X}$ to its double dual. 
	That is, both the map $\mf{X} \to \rH^1_{\Iw}(K,\zp)^*$ in \eqref{ptseq}
        and the map $\rH^1_{\Iw}(K,\zp) \to \mf{X}^*$ of Theorem \ref{mc} arise in the standard manner from a
        $\La$-bilinear pairing
        $$
        	\mf{X} \times \rH^1_{\Iw}(K,\zp) \to \La 
        $$
        defined as follows.  Write $\La = \invlim{F'} \La_{F'}$, where $\La_{F'} = \zp[\mathrm{Gal}(F'/F)]$ and $F'$ runs over the
        finite extensions of $F$ in $K$.  Take $\sigma \in \mf{X}$ and $f \in \rH^1_{\Iw}(K,\zp)$. 
        Write $f$ as an inverse limit of homomorphisms $f_{F'} \in \rH^1(G_{F',S},\zp)$.  Then our pairing is given by
   $$
        	(\sigma,f) \mapsto \invlim{F'} \sum_{\tau \in \Gal(F'/F)} 
        	f_{F'}(\tilde{\tau}^{-1}\sigma\tilde{\tau}) [\tau]_{F'},
        $$
        where $\tilde{\tau}$ denotes a lift of $\tau$ to $G_{F,S}$, and $[\tau]_{F'}$ denotes the group element of
        $\tau$ in $\Lambda_{F'}$.  
        Thus, the composition of  $\mf{X} \to \rH^1_{\Iw}(K,\zp)^*$ with
	the map $\rH^1_{\Iw}(K,\zp)^* \to \mf{X}^{**}$ of Corollary \ref{dualseq} is the usual map $\mf{X} \to \mf{X}^{**}$. 
\end{remark}

\begin{definition}
\label{def:K}
    For $\mf{p}$ in the set $S_f$ of finite primes in $S$, let
    $\mc{G}_{\mf{p}}$ denote the decomposition group in $\mc{G}$ 
    at a place over the prime $\mf{p}$ in $K$, and
    set $\mc{K}_{\mf{p}} = \zp\ps{\mc{G}/\mc{G}_{\mf{p}}}$, which has the natural structure of a left
    $\Omega$-module.
    We then set 
    \begin{eqnarray*}
    	\mc{K} = \bigoplus_{\mf{p} \in S_f} \mc{K}_{\mf{p}} &\mr{and}& \mc{K}_0 = \ker(\mc{K} \to \zp),
    \end{eqnarray*}
    where the map is the sum of augmentation maps.
\end{definition}

\begin{remark}
\label{rem:needlater}
    If $K$ contains all $p$-power roots of unity,
    then the group $\rH^2_{\Iw}(K,\zp)$ is the twist by $\zp(-1)$ of $\rH^2_{\Iw}(K,\zp(1))$.
    As explained in the proof of \cite[Lemma 2.1]{SharifiLfn}, Poitou-Tate duality provides a canonical exact sequence
    \begin{equation} \label{H2}
    	0 \to X' \to \rH^2_{\Iw}(K,\zp(1)) \to \mc{K}_0 \to 0,
    \end{equation}
    where $X'$ is the completely split 
    Iwasawa module over $K$ (i.e., the Galois group of the maximal abelian pro-$p$ extension $K$ that is completely split at all 
    places above $S_f$). 
\end{remark}
	
We next wish to consider local versions of the above results.  Let $T$ and $A = T \otimes_{\zp} \qp/\zp$ be as in
Theorem \ref{thm:ss}.  For $\mf{p} \in S_f$, let 
$$
	\rH^i_{\Iw,\mf{p}}(K,T) = 
	\varprojlim_{\substack{F'/E \text{ finite} \\ F' \subset K}} \bigoplus_{\mf{P} \mid \mf{p}} \rH^i(G_{F'_\mf{P}},T),
$$	
where $G_{F'_{\mf{P}}}$ denotes the absolute Galois group of
the completion $F'_{\mf{P}}$.  
If $M$ is a discrete $\zp\ps{\Gal(F_S/E)}$-module, let
$\rH^i(G_{K,\mf{p}},M)$ denote the direct sum of the groups $\rH^i(G_{K_{\mf{P}}},M)$ over the primes $\mf{P}$ 
in $K$ over $\mf{p}$.   
We have the local spectral sequence
$$
	{\rm P}_{2,\mf{p}}^{i,j}(T) = \rE_{\La}^i( \rH^j(G_{K,\mf{p}},A)^{\vee}) \Rightarrow {\rm P}^{i+j}_{\mf{p}}(T) = \rH^{i+j}_{\Iw,\mf{p}}(K,T)
$$
(cf. \cite[Theorem 4.2.2]{ls}).
Note that $\rH^j(G_{K,\mf{p}},A)^{\vee} \cong \rH^{2-j}_{\Iw,\mf{p}}(K,T^{\dagger})$ by Tate duality, where $T^{\dagger} = \Hom_{\zp}(T,\zp(1))$.

\begin{remark} \label{rem:ssmaps}
Tate and Poitou-Tate duality provide maps between the sum of these local spectral sequences over all $\mf{p} \in S_f$ 
and the global spectral sequences, supposing for simplicity that
$p$ is odd or $K$ is purely imaginary (in general, for real places, one uses Tate cohomology).  On $E_2$-terms, these form a complex
$$
	\cdots \to {\rm F}_2^{i,j}(T) \to \bigoplus_{\mf{p} \in S_f} {\rm P}_{2,\mf{p}}^{i,j}(T) \to \rH_2^{i,j}(T^{\dagger}) \to {\rm F}_2^{i,j+1}(T)
	\to \cdots.
$$
These spectral sequences can be seen in the derived category of complexes of finitely generated 
$\Omega$-modules, where they form an exact triangle (see \cite[Theorem 4.5.1]{ls}).  To see this, one uses the regularity of $\zp$ in order to 
replace the dualizing complex with $\zp$. The cohomology
groups in question can then be identified with those we have written by the isomorphisms of \cite[Lemmas 5.3.1 and 5.3.2, Theorem 5.4.1]{lim}. 
\end{remark}

Let $\Gamma_{\mf{p}} = \mathcal{G}_{\mf{p}}\cap \Gamma$ be the decomposition group in $\Gamma$ at a prime over $\mf{p}$ in $K$, and let $r_{\mf{p}} = \rank_{\zp} \Gamma_{\mf{p}}$.
For an $\Omega$-module $M$, we let $M^{\iota}$ denote the $\Omega$-module which as a compact $\zp$-module is  $M$ 
and on which $g \in \mc{G}$ now acts by $g^{-1}$. 

\begin{lemma} \label{extKp}
	For $j \ge 0$, we have isomorphisms
	$\rE_{\La}^j(\mc{K}_{\mf{p}}) \cong (\mc{K}_{\mf{p}}^{\iota})^{\delta_{r_{\mf{p}},j}}$
	of $\Omega$-modules.
\end{lemma}

\begin{proof}
	This is immediate from Corollary \ref{cor:indzp}.
\end{proof}

Let $\mf{D}_{\mf{p}}$ denote the Galois group of the maximal abelian, pro-$p$ quotient
of the absolute Galois group of the completion $K_{\mf{p}}$ of $K$ at a prime over $\mf{p}$, and consider the completed tensor product
$$
	D_{\mf{p}} =  \Omega \cotimes{\zp\ps{\mc{G}_{\mf{p}}}} \mf{D}_{\mf{p}},
$$
which has the structure of an $\Omega$-module by left multiplication.

\begin{theorem} \label{thm:localseq} 
	Suppose that $K$ contains all $p$-power roots of unity.  For each $\mf{p} \in S_f$,
	we have a commutative diagram of exact sequences 
	$$
		\SelectTips{cm}{} \xymatrix{
		0 \ar[r] & \rE^1_{\La}(\mc{K}_{\mf{p}})(1) \ar[r] \ar[d] &
		D_{\mf{p}}  \ar[r] \ar[d] & D_{\mf{p}}^{**} \ar[r] \ar[d] & 
		\rE^2_{\La}(\mc{K}_{\mf{p}})(1) \ar[d] \ar[r] & 0 &\\
		0 \ar[r] & \rE^1_{\La}(\rH^2_{\Iw}(K,\zp)) \ar[r] & \mf{X} \ar[r] & \mf{X}^{**} \ar[r] & 
		\rE^2_{\La}(\rH^2_{\Iw}(K,\zp))\ar[r] &\zp
		}
	$$
	of $\Omega$-modules in which the vertical maps are the canonical ones.
\end{theorem}

\begin{proof}  
    We have
    \begin{eqnarray*}
    	\rH^2_{\Iw,\mf{p}}(K,\zp) \cong \mc{K}_{\mf{p}}(-1) &\text{and}& \rH^1(G_{K,\mf{p}},\qp/\zp)^{\vee} \cong D_{\mf{p}},
    \end{eqnarray*} 
    the first using our assumption on $K$.  We also have $\rH^2(G_{K,\mf{p}},\qp/\zp) = 0$.
.
    
    The analogue of Theorem \ref{mc} is the exact sequence
    \begin{equation}
    \label{eq:differentname}
    	0 \to \rE_{\La}^1(\mc{K}_{\mf{p}}) \to \rH^1_{\Iw,\mf{p}}(K,\zp)
	\to D_{\mf{p}}^* \to \rE_{\La}^2(\mc{K}_{\mf{p}}) \to \rH^2_{\Iw,\mf{p}}(K,\zp).
    \end{equation}
    We remark that the map 
    $$
   \rE_{\La}^2(\mc{K}_{\mf{p}})  = (\mc{K}_{\mf{p}}^{\iota})^{\delta_{r_{\mf{p}},2}} \to \rH^2_{\Iw,\mf{p}}(K,\zp) \cong \mc{K}_{\mf{p}}(-1)
    $$ 
    is zero since $\Gamma_{\mf{p}}$ acts trivially on $\mc{K}_{\mf{p}}^{\iota}$ but not on any nonzero element
    of $\mc{K}_{\mf{p}}(-1)$.  Applying Lemma \ref{extKp} to \eqref{eq:differentname},
    dualizing, and using the fact that $r_{\mf{p}} \ge 1$ by assumption on $K$, we obtain an isomorphism
    $\rH^1_{\Iw,\mf{p}}(K,\zp)^* \xrightarrow{\sim} D_{\mf{p}}^{**}$ compatible with Corollary \ref{dualseq}.
    The analogue of Corollary \ref{cor:iwseq} is then the exact sequence
    \begin{equation}
    \label{eq:torturous}
    	0 \to \rE_{\La}^1(\mc{K}_{\mf{p}})(1) \to D_{\mf{p}} \to D_{\mf{p}}^{**} \to \rE_{\La}^2(\mc{K}_{\mf{p}})(1) \to 
	\mc{K}_{\mf{p}}. 
    \end{equation}  
    As above, the map $\rE_{\La}^2(\mc{K}_{\mf{p}})(1)
    \to \mc{K}_{\mf{p}}$ is zero.  
    
    The map of exact sequences follows from Remark \ref{rem:ssmaps}.
\end{proof}

One might ask whether or not the map $\mf{X}^* \to \zp$ in Theorem \ref{mc} is zero in the case $r = 2$.

\begin{proposition} \label{nozp1}
	Suppose that $K$ contains all $p$-power roots of unity.  If Leopoldt's conjecture holds for $F$,
	then $X'$ has no $\Lambda$-quotient or $\La$-submodule isomorphic to $\zp(1)$.
\end{proposition}

\begin{proof}
	We claim that if $M$ is a finitely generated $\La$-module such that the invariant group
	$M^{\Gamma}$ has positive $\zp$-rank, then the coinvariant group $M_{\Gamma}$ does as well.  
	To see this, let $I$ be the augmentation ideal in $\La$.  The annihilator of $M^{\Gamma}$ is $I$,
	so the annihilator of $M$ is contained in $I$.  By \cite[Proposition 2.1]{Gr06} and its proof,
	there is an ideal $J$ of $\Lambda$ contained in the annihilator of $M$ such that any prime
	ideal $P$ of $\Lambda$ containing $J$ satisfies $\rank_{\La/P} M/PM$ is
	positive.  We then apply this to $P = I$ to obtain the claim.
	
	Applying this to $X'(-1)$, we may suppose that $X'$ has a quotient isomorphic to $\zp(1)$.   Such a quotient
	is in particular a locally trivial $\zp(1)$-quotient of the Galois group of $\mf{X}$.  In other words,
	we have a subgroup of $\rH^1(G_{K,S},\mu_{p^{\infty}})$ isomorphic to $\zp$ and which maps
	trivially to $\rH^1(G_{K,\mf{p}},\mu_{p^{\infty}})$ for all $\mf{p} \in S_f$.  
	
	The maps
	\begin{eqnarray*}
		\rH^1(G_{F,S},\mu_{p^{\infty}}) \to \rH^1(G_{K,S},\mu_{p^{\infty}})^{\Gamma} &\mr{and}&
		\rH^1(G_{F,\mf{p}},\mu_{p^{\infty}}) \to  \rH^1(G_{K,\mf{p}},\mu_{p^{\infty}})^{\Gamma}
	\end{eqnarray*}
	have $p$-torsion kernel and cokernel.  For instance, the kernel (resp., cokernel) 
	of the first map is (resp., is contained in) $\rH^i(\Gamma,\mu_{p^{\infty}})$ for $i = 1$ (resp., $i = 2$).
	If $\Phi \cong \zp$ is a subgroup of $\Gamma$ that does not fix $\zp(1)$,
	then the Hochschild-Serre spectral sequence 
	$$
		\rH^i(\Gamma/\Phi, \rH^j(\Phi,\mu_{p^{\infty}})) \Rightarrow \rH^{i+j}(\Gamma,\mu_{p^{\infty}})
	$$
	gives finiteness of all $\rH^k(\Gamma,\mu_{p^{\infty}})$, as $\rH^j(\Phi,\mu_{p^{\infty}})$
	is finite for every $j$ (and zero for every $j \neq 0$).
	
	We may now conclude that $\rH^1(G_{F,S},\mu_{p^{\infty}})$ has a subgroup isomorphic
	to $\zp$ with finite image under the localization map
	$$
		\rH^1(G_{F,S},\mu_{p^{\infty}}) \to \bigoplus_{\mf{p} \in S_f} \rH^1(G_{F,\mf{p}},\mu_{p^{\infty}}).
	$$
	In other words, Leopoldt's conjecture must fail (see \cite[Theorem 10.3.6]{nsw}).
\end{proof}	

\begin{remark}
	Proposition \ref{nozp1} also holds for the unramified Iwasawa module $X$ over $K$ in place of $X'$.
\end{remark}

\begin{proposition} \label{r2case}
	Suppose that $r = 2$ and $K$ contains all $p$-power roots of unity.  If Leopoldt's conjecture
	holds for $F$, then the sequences
	\begin{eqnarray}
	\label{eq:twoseq}
		&0 \to \rH^1_{\Iw}(K,\zp) \to \mf{X}^* \to \zp \to 0&\\
		\label{eq:threeseq}
		&0 \to \rE_{\La}^1(\rH^2_{\Iw}(K,\zp)) \to \mf{X} \to \mf{X}^{**} \to \rE_{\La}^2(\rH^2_{\Iw}(K,\zp)) \to 0&
	\end{eqnarray}
	of Theorem \ref{mc} and Corollary \ref{cor:iwseq} are exact.
\end{proposition}

\begin{proof}
	Suppose that Leopoldt's conjecture holds for $F$.  
	Consider first the map $\phi \colon \zp \to \rH^2_{\Iw}(K,\zp)$ of Theorem \ref{mc}.
	The image of $\zp$ is contained in $\rH^2_{\Iw}(K,\zp)^{\Gamma}$.  There is an exact sequence 
	$$
		0 \to X'(-1)^{\Gamma} \to \rH^2_{\Iw}(K,\zp)^{\Gamma} \to \mc{K}_0(-1)^{\Gamma}.
	$$
	Primes in $S_f$ are finitely decomposed in the cyclotomic $\zp$-extension $F_{\cyc}$, 
	and the action of the summand $\Gamma_{\cyc} = \Gal(F_{\cyc}/F)$ of $\Gamma$ on $\mathbb{Z}_p(-1)$ is faithful.  
	It follows that $\mc{K}_0(-1)^{\Gamma} = (\mc{K}_0^{\Gal(K/F_{\cyc})}(-1))^{\Gamma_{\cyc}}$ is trivial.	
	Proposition \ref{nozp1} then implies that $\phi$ must have finite image,
	and we have the first exact sequence.
	
	By Corollary \ref{cor:indzp},  $\rE_{\La}^1(\zp) = 0$ and $\rE_{\La}^2(\zp) \cong \zp$.  The long exact sequence of $\Ext$-groups for  \eqref{eq:twoseq} reads
	$$
		0 \to \rE_{\La}^1(\mf{X}^*) \to \rE_{\La}^1(\rH^1_{\Iw}(K,\zp)) \to \zp \to \rE_{\La}^2(\mf{X}^*).
	$$
	 By Corollary \ref{EEconseq}(b), this implies
	that $\rE_{\La}^1(\rH^1_{\Iw}(K,\zp)) \to \zp$ is surjective with finite kernel.  
	The map $\zp \to \rE_{\La}^1(\rH^1_{\Iw}(K,\zp))$ of \eqref{ptseq} is then also forced to be injective, being that it
	 is of finite (i.e., codimension at least $3$) cokernel $\rE^3_{\La}(\rH^2_{\Iw}(K,\zp))$, for instance by Proposition 
	 \ref{prop:bjork}.
	Therefore, the map $\rE_{\La}^2(\rH^2_{\Iw}(K,\zp)) \to \zp$ in \eqref{ptseq} is trivial,
	and \eqref{eq:threeseq} is exact.
\end{proof}

\subsection{Useful lemmas}

It is necessary for our purposes to account for discrepancies between decomposition and inertia groups,
and the unramified Iwasawa module $X$ and $H^2_{\Iw}(K,\zp(1))$.  The following lemmas are designed for
this purpose.
For a prime $\mf{p} \in S_f$, we set 
$$
	I_{\mf{p}} = \Omega \cotimes{\zp\ps{\mc{G}_{\mf{p}}}} \mf{I}_{\mf{p}},
$$ 
where $\mf{I}_{\mf{p}}$ denotes the inertia subgroup of $\mf{D}_{\mf{p}}$.  Then $I_{\mf{p}}$ 
is an $\Omega$-submodule of $D_{\mf{p}}$.

\begin{remark} \label{unramdesc}
	The unramified Iwasawa module $X$ over $K$
	is the cokernel of the map $\bigoplus_{\mf{p} \in S_f} I_{\mf{p}} \to \mf{X}$, independent of $S$ containing
	the primes over $p$.  Its completely split-at-$S_f$ quotient
	is the cokernel of $\bigoplus_{\mf{p} \in S_f} D_{\mf{p}} \to \mf{X}$.  The latter $\Omega$-module is the
	completely split Iwasawa module $X'$ if $K$ contains the cyclotomic $\zp$-extension of $F$.
\end{remark}

In the following, we suppose that primes over $\mf{p}$
do not split completely in $K/F$, which occurs, for instance, if $\mf{p}$ lies over $p$ or $K$ contains the cyclotomic 
$\zp$-extension of $F$.

\begin{lemma} \label{inertia}
	Suppose that $\Gamma_{\mf{p}} \neq 0$.
	Let $\epsilon_{\mf{p}} = 0$ (resp., 1) if the completion $K_{\mf{p}}$ at a prime over $\mf{p}$ 
	contains (resp., does not contain) the unramified $\zp$-extension of $E_{\mf{p}}$.  Let 
	$\epsilon'_{\mf{p}} = \epsilon_{\mf{p}}\delta_{r_{\mf{p}},1}$, and if $\epsilon'_{\mf{p}} = 1$, suppose
	that $K$ contains all $p$-power roots of unity.
	We have a commutative diagram
	\begin{equation}
	\label{eq:tworows}
		\SelectTips{cm}{} \xymatrix@R=15pt{
		0 \ar[r] & I_{\mf{p}} \ar[r] \ar[d] & D_{\mf{p}} \ar[r] \ar[d] & \mc{K}_{\mf{p}}^{\epsilon_{\mf{p}}} \ar[r] \ar[d] & 0\\
		0 \ar[r] & I_{\mf{p}}^{**} \ar[r] & D_{\mf{p}}^{**} \ar[r] & \mc{K}_{\mf{p}}^{\epsilon_{\mf{p}}'} \ar[r] & 0
		}
	\end{equation}
	where the right-hand vertical map is the identity if $\epsilon_{\mf{p}}' = 1$.
\end{lemma}

\begin{proof}
	We have exact sequences $0 \to \mf{I}_{\mf{p}} \to \mf{D}_{\mf{p}} \to \zp^{\epsilon_{\mf{p}}} \to 0$ by the theory
	of local fields.  These yield the upper exact sequence upon taking the tensor product with $\Omega$ over 
	$\zp\ps{\mc{G}_{\mf{p}}}$.  Since $\Gamma_{\mf{p}} \neq 0$, we have that $\mc{K}_{\mf{p}}$ is a torsion $\La$-module.  
	Taking $\Ext$-groups, we obtain an exact sequence
	\begin{equation} \label{starter}
		0 \to D_{\mf{p}}^* \to I_{\mf{p}}^* \to \rE_{\La}^1(\mc{K}_{\mf{p}}^{\epsilon_{\mf{p}}}) \to \rE_{\La}^1(D_{\mf{p}}).
	\end{equation}
	If $r_{\mf{p}} > 1$ or $\epsilon_{\mf{p}} = 0$, 
	then we are done by Lemma \ref{extKp} after taking a dual.  
	
	Suppose that $r_{\mf{p}} = \epsilon'_{\mf{p}} = 1$.
	We claim that the last map in \eqref{starter} is trivial.  This map is, by Lemma \ref{lem:indext}, 
	just the map of $\Omega$-modules
	$$
		\Omega^{\iota} \otimes_{\zp\ps{\mc{G}_{\mf{p}}}} \Ext^1_{\La_{\mf{p}}}(\zp,\La_{\mf{p}}) 
		\to \Omega^{\iota} \otimes_{\zp\ps{\mc{G}_{\mf{p}}}} \Ext^1_{\La_{\mf{p}}}(\mf{D}_{\mf{p}},\La_{\mf{p}}),
	$$
	where $\La_{\mf{p}} = \zp\ps{\Gamma_{\mf{p}}}$.
	For the claim, we may then assume that $r = 1$ and
	$K$ is the cyclotomic $\zp$-extension of $F$.
	We then have an exact sequence
	$$
		0 \to \mc{K}_{\mf{p}}^{\iota}(1) \to D_{\mf{p}} \to D_{\mf{p}}^{**} \to 0
	$$
	from Theorem \ref{thm:localseq} and Lemma \ref{extKp}.
	Taking $\Ext$-groups yields an exact sequence 
	$$
		0 \to \rE_{\La}^1(D_{\mf{p}}^{**}) \to \rE_{\La}^1(D_{\mf{p}}) \to \mc{K}_{\mf{p}}(-1) \to \rE_{\La}^2(D_{\mf{p}}^{**}),
	$$
	and the first and last term are trivial by Corollary \ref{EEconseq}.
	As there is no nonzero $\La$-module homomorphism $\zp \to
	\zp(-1)$, there is no nonzero homomorphism $\mc{K}_{\mf{p}}^{\iota} \to \mc{K}_{\mf{p}}(-1)$,
	hence the claim.
	Finally, taking $\Ext$-groups once again, we have an exact sequence
	$$
		0 \to I_{\mf{p}}^{**} \to D_{\mf{p}}^{**} \to \mc{K}_{\mf{p}} \to \rE_{\La}^1(I_{\mf{p}}^*).
	$$
	By Corollary \ref{EEconseq}, $\rE_{\La}^1(I_{\mf{p}}^*) = 0$, so we have shown 
	the exactness of the second row of \eqref{eq:tworows}.
\end{proof}

Using Lemma \ref{inertia}, one can derive exact sequences as in Theorem \ref{thm:localseq} with $I_{\mf{p}}$
in place of $D_{\mf{p}}$ if we suppose that $K$ contains all $p$-power roots of unity.  When $F$  contains
$\mu_p$, this hypothesis is equivalent to $K$ containing the cyclotomic $\zp$-extension $F_{\mr{cyc}}$ of $F$.  

\begin{lemma} \label{unramsplit} \
	\begin{itemize}
		\item[(a)] If $K_{\mf{p}}$ contains a $\zp^2$-extension of 
		$E_{\mf{p}}$ for all $\mf{p} \in S_f$ lying over $p$, then the kernel of the quotient map $X \to X'$ 
		is pseudo-null.  
		\item[(b)] If $K_{\mf{p}}$ contains the unramified $\zp$-extension of $E_{\mf{p}}$ for all $\mf{p} \in S_f$ lying over 
		$p$, the the quotient map $X \to X'$ is an isomorphism.
	\end{itemize}
\end{lemma}

\begin{proof}
	Take $S$ to be the set of primes over $p$ and $\infty$. We have a canonical surjection 
	$$
		\bigoplus_{\mf{p} \in S_f} (\Omega \cotimes{\zp\ps{\mc{G}_\mf{p}}} \mf{D}_{\mf{p}}/\mf{I}_{\mf{p}}) \to \ker(X \to X')
	$$
	and $\mf{D}_{\mf{p}}/\mf{I}_{\mf{p}}$ is zero or $\zp$ according as to whether $K_{\mf{p}}$ does or does not
	contain the unramified $\zp$-extension of $E_{\mf{p}}$, respectively.  This implies part (b) immediately.  It
	also implies part (a), since $\Omega \cotimes{\zp\ps{\mc{G}_\mf{p}}} \mf{D}_{\mf{p}}/\mf{I}_{\mf{p}}$ is of finite $\zp\ps{\mc{G}/\mc{G}_{\mf{p}}}$-rank and $\zp\ps{\mc{G}/\mc{G}_{\mf{p}}}$ is pseudo-null in case (a). 
\end{proof}

The following lemma describes the structure of the $\Ext$-groups of $\mc{K}_0$ in terms of 
those of $\mc{K}$.

\begin{lemma} \label{K0}
	Let $\mf{p} \in S_f$. 
	\begin{itemize}
		\item[(a)] For $0 \le j < r-1$, we have
		$\rE_{\La}^j(\mc{K}_0) \cong \rE_{\La}^j(\mc{K})$.  For $j \ge r+1$, we have $ \rE_{\La}^j(\mc{K}) = E_{\La}^j(\mc{K}_0)  = 0$.
		\item[(b)] If $r \neq r_{\mf{p}}$ for all $\mf{p} \in S_f$, then $\rE_{\La}^r(\mc{K}) = \rE_{\La}^r(\mc{K}_0) = 0$, and we have an exact sequence
		$$
			0 \to \rE_{\La}^{r-1}(\mc{K}) \to \rE_{\La}^{r-1}(\mc{K}_0)
			\to \zp \to 0.
		$$
		\item[(c)] If $r = r_{\mf{p}}$ for some $\mf{p} \in S_f$, then $\rE_{\La}^{r-1}(\mc{K}_0) \cong \rE_{\La}^{r-1}(\mc{K})$, and we have an exact sequence
		$$
			0 \to \zp \to \rE_{\La}^r(\mc{K}) \to \rE_{\La}^r(\mc{K}_0) \to 0.
		$$
	\end{itemize}
\end{lemma}

\begin{proof}	
	Note that 
	$$
		\rE_{\La}^j(\mc{K}) \cong  \bigoplus_{\mf{p} \in S_f} (\mc{K}_{\mf{p}}^{\iota})^{\delta_{j,r_{\mf{p}}}}
	$$
	by Lemma \ref{extKp}.  
	Moreover, Corollary \ref{cor:indzp} tells us that $\rE_{\La}^j(\zp) \cong \zp^{\delta_{r,j}}$.  We are quickly reduced to the case
	that $r = r_{\mf{p}}$ for some $\mf{p}$.  The map $\rE_{\La}^r(\zp) \to \rE_{\La}^r(\mc{K}_{\mf{p}})$ for such a $\mf{p}$
	is the map $\zp \to \mc{K}_{\mf{p}}^{\iota}$ that takes $1$ to the norm element, hence is injective.
\end{proof}

If $K$ contains all $p$-power roots of unity, then from \eqref{H2} we have an exact sequence
\begin{equation} \label{H2Ext}
	\cdots \to \rE_{\La}^j(\mc{K}_0)(1) \to \rE_{\La}^j(\rH^2_{\Iw}(K,\zp)) \to \rE_{\La}^j(X')(1) \to \rE_{\La}^{j+1}(\mc{K}_0)(1) \to \cdots
\end{equation}
for all $j$.  Lemmas \ref{unramsplit} and \ref{K0} then allow one to study the relationship between the higher $\Ext$-groups 
of $\rH^2_{\Iw}(K,\zp)$ occuring in Theorem \ref{thm:localseq} and the higher $\Ext$-groups of $X$.

\begin{remark} \label{secondproof}
	At the end of Section \ref{ss:GreenbergsConjecture}, we asserted that if $K$ contains $\mu_{p^{\infty}}$ and 
	$r_{\mf{p}} \ge 2$ for all $\mf{p} \in S_f$, then $\mf{X}$ is torsion-free if and only if $X$ is pseudo-null.
	This may also be seen as follows.  By Corollary \ref{cor:iwseq}, the $\La$-torsion subgroup of $\mf{X}$ is
	isomorphic to $\rE_{\La}^1(\rH^2_{\Iw}(K,\zp))$.  By assumption and Lemmas \ref{extKp} and \ref{K0}, we have 
	$\rE_{\La}^1(\mc{K}_0) = 0$.  Thus, by the exact sequence \eqref{H2Ext}, the triviality of $\rE_{\La}^1(\rH^2_{\Iw}(K,\zp))$
	and the triviality of $\rE_{\La}^1(X')$ are equivalent.  Since $X'$ is $\La$-torsion, $\rE_{\La}^1(X') = 0$ if and only if $X'$ is
	pseudo-null, which by Lemma \ref{unramsplit} is equivalent to the pseudo-nullity of $X$.
\end{remark}

\subsection{Eigenspaces}
\label{ss:ranksnake}

We end with a discussion of the rank of the $\Delta$-eigenspaces of the global and local Iwasawa modules $\mf{X}$
and $D_{\mf{p}}$.  Let us suppose now that $\mc{G} = \Gamma \times \Delta$, and for simplicity, that $\Delta$ is abelian. 
Without loss of generality, we shall suppose here that $F$ contains $\Q(\mu_p)$, and we let $\omega \colon \Delta \to \zp^{\times}$
denote the Teichm\"uller character.

Let $\psi$ be a $\bar{\Q}_p^{\times}$-valued character of $\Delta$.  For a $\zp[\Delta]$-module $M$, we let 
$$
	M^{\psi} = M \otimes_{\zp[\Delta]} \mc{O}_{\psi},
$$ 
where $\mc{O}_{\psi}$ is the $\zp$-algebra generated by the values of $\psi$, and $\zp[\Delta] \to \mc{O}_{\psi}$ is the surjection induced by $\psi$.  We set $\La = \zp\ps{\Gamma}$ and $\La_{\psi} = \mc{O}_{\psi}\ps{\Gamma}$.
Note that $\Omega^{\psi} \cong \La_{\psi}$ as compact $\mc{O}_{\psi}$-algebras, but $\Omega^{\psi}$ has the extra structure of an $\Omega$-module on which $\Delta$ acts by $\psi$.  

Let $r_2(E)$ denote the number of complex places of $E$ and $r_1^{\psi}(E)$ the number of real places of $E$ at
which $\psi$ is odd.  We have the following consequence of Iwasawa-theoretic 
global and local Euler-Poincar\'e characteristic formulas, as found in \cite[5.2.11, 5.3.6]{nekovar}.

\begin{lemma} \label{lem:rankone} \
	\begin{itemize}
		\item[(a)] If weak Leopoldt holds for $K$, then
		$\rank_{\La_{\psi}} \mf{X}^{\psi} = r_2(E) + r_1^{\psi}(E)$.
		\item[(b)] If either $\Gamma_{\mf{p}} \neq 0$ or $\psi|_{\Delta_{\mf{p}}} \neq 1$,  
		then $\rank_{\La_{\psi}} D_{\mf{p}}^{\psi} = [E_{\mf{p}}:\qp]$.
	\end{itemize}
\end{lemma}

\begin{proof}
	Let $\Sigma$ be the union of $S$ and the primes that ramify in $F/E$.
	Since the primes in $\Sigma \setminus S$ can ramify at most tamely in $\mf{X}_{\Sigma}$,
	the $\La_{\psi}$-modules $\mf{X}^{\psi}$ and $\mf{X}_{\Sigma}^{\psi}$ (the $\Sigma$-ramified Iwasawa
	module over $K$) have the same rank.
	Endow $\mc{O}_{\psi^{-1}}$ (which equals $\mc{O}_{\psi}$ as a $\zp$-module) 
	with a $G_{E,\Sigma}$-action through $\psi^{-1}$.
	Let $B_{\psi} = (\Omega^{\psi^{-1}})^{\vee} \cong \Hom_{\zp,\cont}(\La,
	\mc{O}_{\psi^{-1}}^{\vee})$, which is a discrete $\La_{\psi}\ps{G_{E,\Sigma}}$-module.
	Restriction and Shapiro's lemma (see \cite[8.3.3]{nekovar} and \cite[5.2.2, 5.3.1]{lim}) provide
	$\La_{\psi}$-module homomorphisms
	$$
		\rH^1(G_{E,\Sigma},B_{\psi}) 
		\xrightarrow{\Res} \rH^1(G_{F,\Sigma}, B_{\psi})^{\Delta}
		\xrightarrow{\sim} \rH^1(G_{K,\Sigma},\mc{O}_{\psi^{-1}}^{\vee})^{\Delta}
		\xrightarrow{\sim} (\mf{X}_{\Sigma}^{\psi})^{\vee},
	$$
	restriction having cotorsion kernel and cokernel.  (The last step passes through the intermediate
	module $\Hom_{\zp[\Delta]}(\mf{X}_{\Sigma} \otimes_{\zp} \mc{O}_{\psi^{-1}},\qp/\zp)$.)
	We are therefore reduced to computing the $\La_{\psi}$-corank
	of $\rH^1(G_{E,\Sigma},B_{\psi})$.  The global Euler characteristic formula tells us that
	$$
		\sum_{j=0}^2 (-1)^{j-1}\rank_{\La_{\psi}} \rH^j(G_{E,\Sigma},B_{\psi})^{\vee}
		= \sum_{v \in S-S_f} \rank_{\La_{\psi}} (\Omega^{\psi}(1))^{G_{E_v}},
	$$
	and $\rH^j(G_{E,\Sigma},B_{\psi})$ is $\La_{\psi}$-cotorsion for $j = 0$ and $j = 2$,
	the latter by weak Leopoldt for $K$.
	
	Recall that $D_{\mf{p}} = \rH^1(G_{K,\mf{p}},\qp/\zp)^{\vee}$.  Restriction and Shapiro's lemma 
	\cite[5.3.2]{lim} again 
	reduce the computation
	of the $\La_{\psi}$-corank of $H^1(G_{E,\mf{p}},B_{\psi})$, and 
	the local Euler characteristic formula tells us that
	$$
		\sum_{j=0}^2 (-1)^j \rank_{\La_{\psi}} \rH^j(G_{E_{\mf{p}}},B_{\psi})^{\vee} 
		= [E_{\mf{p}}:\qp]\cdot \rank_{\La_{\psi}} \Omega^{\psi} = [E_{\mf{p}}:\qp].
	$$
	As $\rH^2(G_{E_{\mf{p}}},B_{\psi})^{\vee}$ is trivial, and 
	$\rH^0(G_{E_{\mf{p}}},B_{\psi})^{\vee} \cong (\Omega^{\psi})_{G_{E_{\mf{p}}}}$ is torsion
	by virtue of the fact that either $\Gamma_{\mf{p}}$ or $\psi|_{\Delta_{\mf{p}}}$ is nontrivial, we are done.
\end{proof}

Let us suppose in the following three lemmas that $\psi$ has order prime to $p$.
These following lemmas are variants of the lemmas of the previous section in ``good eigenspaces''.  
The proofs are straightforward from what has already been done and as such are left to the reader.
For the second lemma, one can use the following simple fact: for an $\Omega$-module $M$, we have 
\begin{equation} \label{reflect}
	(\rE_{\La}^j(M)(1))^{\psi} \cong \rE_{\La}^j(M^{\omega\psi^{-1}})(1).
\end{equation}

\begin{lemma} \label{Kpsi}
	We have $\mc{K}_{\mf{p}}^{\psi} = 0$ if $\psi|_{\Delta_{\mf{p}}} \neq 1$.  We have $\mc{K}^{\psi} \cong \mc{K}_0^{\psi}$
	if $\psi \neq 1$.
\end{lemma}

\begin{lemma} \label{inertpsi}
	Suppose that $K$ contains the cyclotomic $\zp$-extension $F_{\mr{cyc}}$ of $F$.  
	If $\psi|_{\Delta_{\mf{p}}} \neq 1$ (resp., $\omega\psi^{-1}|_{\Delta_{\mf{p}}} \neq 1$), 
	then $I_{\mf{p}}^{\psi} \to D_{\mf{p}}^{\psi}$ (resp., $D_{\mf{p}}^{\psi} \to (D_{\mf{p}}^{\psi})^{**}$) is an isomorphism. 
\end{lemma}

\begin{lemma}
	Suppose that $K$ contains $F_{\mr{cyc}}$.  
	Then the maps $X^{\psi} \twoheadrightarrow (X')^{\psi} \hookrightarrow \rH^2(K,\zp(1))^{\psi}$ 
	are isomorphisms if $\psi|_{\Delta_{\mf{p}}} \neq 1$ for all $\mf{p}$ lying over $p$.
\end{lemma}

\section{Reflection-type theorems for Iwasawa modules}
\label{s:reflectiontype}

In this section, we prove results that relate an Iwasawa module in a given eigenspace
with another Iwasawa module in a ``reflected'' eigenspace.  These modules typically appear on opposite sides
of a short exact sequence, with the middle term being measured by $p$-adic $L$-functions.
The method in all cases is the same:
we take a sum of the maps of exact sequences at primes over $p$ found in Theorem \ref{thm:localseq} and apply the snake
lemma to the resulting diagram.  Here, we focus especially on cases in which eigenspaces of the unramified outside $p$ Iwasawa modules $\mf{X}$ have
rank $1$, in order that the corresponding eigenspace of the double dual is free of rank one.  
Our main result is a 
symmetric exact sequence for an unramified Iwasawa module and its reflection 
in the case of an imaginary quadratic field. This sequence gives rise to a computation of  second Chern classes (see
Subsection \ref{ss:imaginaryquad}).

We maintain the notation of Section \ref{s:FieldsWithR2Equal1}.
We suppose in this section that $p$ is odd, and
we let $S$ be the set of primes of $E$ over $p$ and $\infty$.
We let $\psi$ denote a one-dimensional character of the absolute Galois group of $E$ of finite order prime to $p$, and
let $E_{\psi}$ denote the fixed field of its kernel.
We then set $F = E_{\psi}(\mu_p)$ and $\Delta = \Gal(F/E)$.
Let $\omega$ denote the 
Teichm\"uller character of $\Delta$. 

We now take $\widetilde{E}$ to be the compositum of all $\zp$-extensions of $E$, and we set $r = \rank_{\zp} \Gal(\widetilde{E}/E)$.
If Leopoldt's conjecture holds for $E$, then $r = r_2(E)+1$.
We set $K = F\widetilde{E}$.  As before, we
take $\mc{G} = \Gal(K/E)$ and $\Gamma = \Gal(K/F)$ and set $\Omega = \zp\ps{\mc{G}}$ and $\La = \zp\ps{\Ga}$.  
 
For a subset $\Sigma$ of $S_f$, let us set $\mc{K}_{\Sigma} = \bigoplus_{\mf{p} \in \Sigma} \mc{K}_{\mf{p}}$.  We set
$$\mc{H}_{\Sigma} = \ker(\rH^2_{\Iw}(K,\zp(1)) \to \mc{K}_{S_f - \Sigma}),$$ 
which for $\Sigma \neq \varnothing$ 
fits in an exact sequence
\begin{equation} \label{Hseq}
	0 \to X' \to \mc{H}_{\Sigma} \to \mc{K}_{\Sigma} \to \zp \to 0.
\end{equation}
For $\Sigma = \varnothing$, we have $\mc{H}_{\Sigma} \cong X'$.  We shall study the diagram that arises from the sum
of exact sequences in Theorem \ref{thm:localseq} over primes in $T = S_f - \Sigma$.  Setting $D_T = \bigoplus_{\mf{p} \in T} D_{\mf{p}}$, 
it reads
\begin{equation} \label{Tdiag}
		\SelectTips{cm}{} \xymatrix{
		0 \ar[r] & \rE^1_{\La}(\mc{K}_T)(1) \ar[r] \ar[d] &
		D_T  \ar[r] \ar[d] & D_T^{**} \ar[r] \ar[d]^{\phi_T} & 
		\rE^2_{\La}(\mc{K}_T)(1) \ar[d] \ar[r] & 0 &\\
		0 \ar[r] & \rE^1_{\La}(\rH^2_{\Iw}(K,\zp)) \ar[r] & \mf{X} \ar[r] & \mf{X}^{**} \ar[r] & 
		\rE^2_{\La}(\rH^2_{\Iw}(K,\zp))\ar[r] &\zp,
		}
\end{equation}
where $\phi_T = \sum_{\mf{p} \in T} \phi_{\mf{p}}$ is the sum of maps $\phi_{\mf{p}} \colon D_{\mf{p}}^{**} \to \mf{X}^{**}$.
We take $\psi$-eigenspaces, on which the map to $\zp$ in the diagram will vanish if $\psi \neq 1$, $r = 1$, or $r = 2$
and Leopoldt's conjecture holds for $F$, the latter by Proposition \ref{r2case}.  The cokernel of $D_T \to \mf{X}$ is the
Iwasawa module $\mf{X}_{\Sigma,T}$ which is the Galois group over $K$ of the maximal pro-$p$ abelian extension of $K$ which is unramified outside of $\Sigma$ and totally split over $T = S_f - \Sigma$.  The group $I_T
=  \bigoplus_{\mf{p} \in T} I_{\mf{p}}$ has the property that the cokernel of $I_T \to \mf{X}$
is the unramified outside of $\Sigma$-Iwasawa module $\mf{X}_{\Sigma}$.

In this section, we focus on examples for which $\rank_{\La_{\psi}} \mf{X}^{\psi} = 1$, which forces $r \le 2$ under Leopoldt's conjecture by Lemma \ref{lem:rankone}.  We have that $(\mf{X}^{\psi})^{**}$ is free of rank one over $\Lambda_{\psi}$, by Lemma \ref{lem:double}.  If $\mf{p}$
is split in $E$, then $(D_{\mf{p}}^{\psi})^{**}$ is also isomorphic to $\La_{\psi}$, so 
$$
	\phi_{\mf{p}}^{\psi} \colon (D_{\mf{p}}^{\psi})^{**} \to (\mf{X}^{\psi})^{**}
$$ 
is identified
with multiplication by an element of $\La_{\psi}$, well-defined up to unit.  We shall exploit this fact throughout.
At times, we will have to distinguish between decomposition and inertia groups, which we will deal with below as the need arises.
In our examples, $T$ is always a set of degree one primes, so $r_{\mf{p}} = 1$ for $\mf{p} \in T$.  The assumptions on $\psi$
and $T$ make most results cleaner and do well to illustrate the role of second Chern classes, but the methods can be applied for any $\zp^r$-extension containing the cyclotomic $\zp$-extension and any set of primes over $p$.

As in Definition \ref{def:adjoint}, for an $\Omega$-module $M$ that is finitely generated over $\La$ 
with annihilator of height at least $r$, we define the adjoint $\alpha(M)$ of $M$ to be $\rE^r_{\La}(M)$.
 
\subsection{The rational setting}
\label{ss:rational}

Let us demonstrate the application of the results of Subsection \ref{ss:gensetup} in the setting of the classical Iwasawa main conjecture.   Suppose that $E = \Q$ and that  $\psi$ is odd.  For simplicity, we assume $\psi \ne \omega$.
We study the unramified Iwasawa module $X$ over $K$.

\begin{theorem} \label{thm:cyclcase}
	If $(X')^{\omega\psi^{-1}}$ is finite, then there is an exact sequence of $\Omega$-modules
	$$
		0 \to X^{\psi} \to \Omega^{\psi}/(\mc{L}_\psi) \to ((X')^{\omega\psi^{-1}})^{\vee}(1) \to 0
	$$
	with $\mc{L}_\psi$ interpolating the $p$-adic $L$-function for $\chi=\omega\psi^{-1}$ as in
	Theorem \ref{cyclmc}.
\end{theorem}

\begin{proof}
	Note that $\mc{K} = \mc{K}_p$.
	By Lemma \ref{extKp}, we have $\rE_{\La}^1(\mc{K}) \cong \mc{K}^{\iota}$ and 
	$\rE_{\La}^2(\mc{K}) = 0$. 
	Since $\psi \neq \omega$, Lemma
	\ref{Kpsi} tells us that $\mc{K}^{\omega\psi^{-1}} \cong \mc{K}_0^{\omega\psi^{-1}}$.  
	By \eqref{H2Ext} and Lemma \ref{K0} and our assumption of pseudo-nullity of $X^{\omega\psi^{-1}}$,
	we have that the natural maps
	\begin{eqnarray*}
		\rE_{\La}^1(\mc{K})(1)^{\psi} \xrightarrow{\sim} \rE_{\La}^1(\rH^2_{\Iw}(K,\zp))^{\psi}
		&\mr{and}& \rE_{\La}^2(\rH^2_{\Iw}(K,\zp))^{\psi} \xrightarrow{\sim} \rE_{\La}^2(X')(1)^{\psi}
	\end{eqnarray*}
	are isomorphisms. By \eqref{reflect} and Proposition \ref{Er}(a), $\rE_{\La}^2(X')(1)^{\psi} \cong ((X')^{\omega\psi^{-1}})^{\vee}(1)$.
	The diagram of Theorem \ref{thm:localseq} with $\mf{p} = p$ reads
	 $$
    		\SelectTips{cm}{} \xymatrix@C=15pt{		
   		0 \ar[r] & \rE^1_{\La}(\mc{K}_p)(1)^{\psi} \ar[r]  \ar[d]^{\wr} &
    		D_p^{\psi} \ar[r] \ar[d] & (D_p^{\psi})^{**}
    		\ar[r]  \ar[d] & 0&\\
    		0 \ar[r] & \rE_{\La}^1(\rH^2_{\Iw}(K,\zp))^{\psi} \ar[r] & 
		\mf{X}^{\psi} \ar[r] & (\mf{X}^{\psi})^{**} \ar[r] & ((X')^{\omega\psi^{-1}})^{\vee}(1) \ar[r]&0.
    		}
  	$$
	It follows from Lemma \ref{inertia} and the fact that there is no nonzero map
	$\mc{K}_p^{\iota}(1) \to \mc{K}_p$ of $\Omega$-modules
	that we can replace $D_p$ by $I_p$ in the diagram.
	By applying the snake lemma to the resulting diagram, we obtain an exact sequence
	\begin{equation} \label{intermedseq}
		0 \to X^{\psi} \to \coker \theta \to ((X')^{\omega\psi^{-1}})^{\vee}(1) \to 0,
	\end{equation}
	where $\theta \colon (I^{\psi} _p)^{**} \to (\mf{X}^{\psi})^{**}$ is the canonical map (restricting $\phi_p^{\psi}$).  
	The map $\theta$ is of free rank one $\La_{\psi}$-modules by Lemmas \ref{lem:rankone} and \ref{lem:double},
	and it is nonzero, hence injective, as $X$ is torsion.  We may identify its
	image with a nonzero submodule of $\Omega^{\psi}$.  Since $(X')^{\omega\psi^{-1}}$ is pseudo-null, 
	\eqref{intermedseq} tells us that
    	$$
    		c_1(X^{\psi}) = c_1( \Omega^{\psi}/\im \theta), 
    	$$
    	and this forces the image of $\theta$ to be $c_1(X^{\psi})$.
	By the main conjecture of Theorem \ref{cyclmc}, we have $c_1(X^{\psi}) = (\mc{L}_{\psi})$.
\end{proof}

\begin{remark} \label{nofinite}
	If we do not assume that $(X')^{\omega\psi^{-1}}$ is finite, one may still derive an exact sequence
	of $\La_{\psi}$-modules
	$$
		0 \to \alpha(X^{\omega\psi^{-1}})(1) \to X^{\psi} \to \Omega^{\psi}/(\mc{M}) \to 
		((X'_{\fin})^{\omega\psi^{-1}})^{\vee}(1) \to 0
	$$ 
	for some $\mc{M} \in \Omega^{\psi}$ such that $(\mc{M})c_1((X^{\omega\psi^{-1}})^{\iota}(1)) = (\mc{L}_{\psi})$.
\end{remark}

\subsection{The imaginary quadratic setting}
\label{ss:imaginaryquad}

In this subsection, we take our base field $E$ to be imaginary quadratic.  Let $p$ be an odd prime that splits 
into two primes $\mf{p}$ and $\bar{\mf{p}}$ in $E$, so $S_f = \{\mf{p},\bar{\mf{p}}\}$.
Since Leopoldt's conjecture holds for $E$, we have $\Ga = \Gal(K/F) \cong \zp^2$.   
We let $\mf{X}_{\mf{p}}$ denote the $\mf{p}$-ramified (i.e., unramified outside of the primes over $\mf{p}$) Iwasawa module over $K$, and similarly for $\bar{\mf{p}}$.

We will prove the following result and derive some consequences of it.

\begin{theorem}
\label{thm:mainresult1}  
Suppose that $E$ is imaginary quadratic and $p$ splits in $E$.
If $X^{\omega\psi^{-1}}$ is pseudo-null as a $\La_{\psi}$-module, then 
there is a canonical exact sequence
of $\Omega$-modules
\begin{equation}
\label{eq:maineq1} 
	0 \to   
	(X/X_{\mathrm{fin}})^{\psi} \to 
	\frac{\Omega^{\psi}}{c_1(\mf{X}_{\mf{p}}^{\psi}) +  c_1(\mf{X}_{\bar{\mf{p}}}^{\psi}) }
	\to \alpha(X^{\omega\psi^{-1}})(1) \to 0.
\end{equation}
Moreover, we have $X_\mathrm{fin}^{\psi} = 0$ unless $\psi = \omega$, and $X_\mathrm{fin}^{\omega}$ is cyclic.   
\end{theorem}

We require some lemmas. 

\begin{lemma} \label{lem:noassum}
	The completely split Iwasawa module $X'$ over $K$ is equal to the unramified Iwasawa module $X$ over $K$, and
	the map $I_{\mf{p}} \to D_{\mf{p}}$ is an isomorphism.
	Moreover, we have $\rE_{\La}^1(\mc{K}_{\mf{p}}) = 0$ and  $\rE_{\La}^2(\mc{K}_{\mf{p}}) \cong \mc{K}_{\mf{p}}^{\iota}$.
\end{lemma}

\begin{proof}
	The prime $\mf{p}$ is infinitely ramified and has infinite residue
	field extension in $\widetilde{E}$, so $r_{\mf{p}} = 2$.  
	The statements follow from Lemma \ref{unramsplit}(b), Lemma \ref{inertia}, and Lemma \ref{extKp} respectively.
\end{proof}

Note that $\mc{K} = \mc{K}_{\mf{p}} \oplus \mc{K}_{\bar{\mf{p}}}$ and $\mc{K}_0 = \ker(\mc{K} \to \zp)$ by the definition of
Remark \ref{rem:needlater}. 

\begin{lemma} \label{lem:H2seqnew}
	If $X^{\omega\psi^{-1}}$ is pseudo-null as a $\La_{\psi}$-module, then so is $\rH^2_{\Iw}(K,\zp)^{\omega\psi^{-1}}$,
	and we have an exact sequence of $\Omega$-modules
	$$
		0 \to \zp(1)^{\psi} \to (\mc{K}^{\omega\psi^{-1}})^{\iota}(1) \to \rE^2_{\La}(\rH^2_{\Iw}(K,\zp))^{\psi} \to 
		\rE_{\La}^2(X^{\omega\psi^{-1}})(1) \to 0.
	$$
\end{lemma}

\begin{proof} 
    	Lemmas \ref{lem:noassum} and \ref{K0} tell us that $E_{\La}^i(\mc{K}_0)(1) = 0$ for $i \neq 2$ and provide
    	an exact sequence 
	\begin{equation}
	\label{eq:betterdisplay}
	0 \to \zp(1) \to \mc{K}^{\iota}(1) \to \rE_{\La}^2(\mc{K}_0)(1) \to 0.
	\end{equation}  
	The exact sequence \eqref{H2Ext} has the form
	$$
		0 \to \rE_{\La}^1(\rH^2_{\Iw}(K,\zp)) \to \rE_{\La}^1(X)(1) \to \rE_{\La}^2(\mc{K}_0)(1) 
		\to \rE_{\La}^2(\rH^2_{\Iw}(K,\zp)) \to \rE_{\La}^2(X)(1) \to 0,
	$$
	noting that $X = X'$ by Lemma \ref{lem:noassum}.
	The $\psi$-eigenspaces of the first two terms are zero by \eqref{reflect} and the pseudo-nullity
	of $X^{\omega\psi^{-1}}$, yielding the first assertion and
	leaving us with a short exact sequence.  Splicing this together with the
	$\psi$-eigenspace of the sequence \eqref{eq:betterdisplay} and applying \eqref{reflect} to the last
	term, we obtain the exact sequence of the statement.
\end{proof}

The main conjecture for imaginary quadratic fields is concerned with the unramified outside $\mf{p}$
Iwasawa module $\mf{X}_{\mf{p}}$ over $K$.  For it, we have the following result on first Chern classes.

\begin{proposition} \label{unramoutp}
	If $X^{\omega\psi^{-1}}$ is pseudo-null as a $\La_{\psi}$-module, then there is an injective pseudo-isomorphism
	$\mf{X}_{\mf{p}}^{\psi} \to \Omega^{\psi}/c_1(\mf{X}_{\mf{p}}^{\psi})$
	of $\Omega$-modules.
\end{proposition}

\begin{proof}
	We apply the snake lemma to the $\psi$-eigenspaces of the 
	diagram of Theorem \ref{thm:localseq}.  By Lemma \ref{lem:noassum} and the pseudo-nullity in 
	Lemma \ref{lem:H2seqnew}, one has a commutative diagram
        \begin{equation} \label{onepseq}
    		\SelectTips{cm}{} \xymatrix@C=15pt{		
   		 0 \ar[r] &
    		I_{\bar{\mf{p}}}^{\psi} \ar[r] \ar[d] & (I_{\bar{\mf{p}}}^{\psi})^{**}
    		\ar[r]  \ar[d]^{\phi_{\bar{\mf{p}}}^{\psi}} & \rE_{\La}^2(\mc{K}^{\omega\psi^{-1}}_{\bar{\mf{p}}})(1) \ar[r] \ar[d] & 0 \\
    		0 \ar[r] & 
		\mf{X}^{\psi} \ar[r] & (\mf{X}^{\psi})^{**} \ar[r] & \rE_{\La}^2(\rH^2_{\Iw}(K,\zp))^{\psi} \ar[r]&0,
    		}
  	\end{equation}
	the right exactness of the bottom row following from Proposition \ref{r2case}.  	
	We immediately obtain an exact sequence
	$$
		0 \to \mf{X}_{\mf{p}}^{\psi} \to \coker (\phi_{\bar{\mf{p}}}^{\psi}) \to C \to 0
	$$
	for $\phi_{\bar{\mf{p}}}^{\psi}$ defined to be as in the diagram \eqref{onepseq},
	with $C$ a pseudo-null $\Omega$-module that by Lemmas \ref{extKp} and \ref{lem:H2seqnew} fits in an exact sequence
	$$
		0 \to \zp(1)^{\psi} \to (\mc{K}_{\mf{p}}^{\omega\psi^{-1}})^{\iota}(1) \to C \to 
		\alpha(X^{\omega\psi^{-1}})(1) \to 0.
	$$
	The map
    	$\phi_{\bar{\mf{p}}}^{\psi} \colon (I^{\psi}_{\bar{\mf{p}}})^{**} \to (\mf{X}^{\psi})^{**}$ is an injective homomorphism
    	of free rank one $\La_{\psi}$-modules.  Since $C$ is pseudo-null, the image of $\phi_{\bar{\mf{p}}}^{\psi}$ is 
	$c_1(\mf{X}_{\mf{p}}^{\psi})$, as required.
\end{proof}

We now prove our main result.

\begin{proof}[Proof of Theorem \ref{thm:mainresult1}]   
    Consider \eqref{Tdiag} for $T = \{\mf{p},\bar{\mf{p}}\}$. From \eqref{onepseq} one has
    \begin{equation}
    \label{eq:diagram1new}
    		\SelectTips{cm}{} \xymatrix{		
    0 \ar[r]  &
    		I_{\mf{p}}^{\psi} \oplus I_{\bar{\mf{p}}}^{\psi} \ar[r] \ar[d] & (I_{\mf{p}}^{\psi})^{**} \oplus 
		(I_{\bar{\mf{p}}}^{\psi})^{**}
    		\ar[r]  \ar[d]^{\phi_{\mf{p}}^{\psi} + \phi_{\bar{\mf{p}}}^{\psi}}& 
    		\rE_{\La}^2(\mc{K}^{\omega\psi^{-1}})(1) \ar[d] \ar[r] & 0&\\
    		0 \ar[r] & \mf{X}^{\psi} \ar[r] & (\mf{X}^{\psi})^{**} \ar[r] & \rE^2_{\La}(\rH^2_{\Iw}(K,\zp))^{\psi} \ar[r]&0.\\
    		}
    \end{equation}
    The snake lemma applied to \eqref{eq:diagram1new} produces an exact sequence
    \begin{equation}
    \label{eq:maineq} 
	\zp(1)^{\psi} \to 
    	X^{\psi} \to 
    	\frac{(\mf{X}^{\psi})^{**} }{(I_{\mf{p}}^{\psi})^{**}+ ( I_{\bar{\mf{p}}}^{\psi})^{**}}
    	\to \alpha(X^{\omega\psi^{-1}})(1)  \to 0,
    \end{equation}
    where the first and last terms follow from the exact sequence of Lemma \ref{lem:H2seqnew}.
    In the proof of Proposition \ref{unramoutp}, we showed that $\phi_{\bar{\mf{p}}}^{\psi}$ is injective with 
    image $c_1(\mf{X}_{\mf{p}}^{\psi})$
    in the free rank one $\Omega^{\psi}$-module $(\mf{X}^{\psi})^{**}$, and similarly upon switching
    $\mf{p}$ and $\bar{\mf{p}}$.
    Thus, we have an isomorphism
    \begin{equation}
    \label{eq:fix1}
    \frac{(\mf{X}^{\psi})^{**} }{(I_{\mf{p}}^{\psi})^{**}+ ( I_{\bar{\mf{p}}}^{\psi})^{**}} \cong \frac{\Omega^{\psi}}{c_1(\mf{X}_{\mf{p}}^{\psi}) + c_1(\mf{X}_{\bar{\mf{p}}}^{\psi})}.
    \end{equation}
    
    If $\psi \neq \omega$, then  $\zp(1)^{\psi} = 0$.  For $\psi = \omega$, 
    we claim that the image of the map $\zp(1) \to X^{\omega}$ of \eqref{eq:maineq} is finite cyclic.  
    Since
    $\Omega^{\psi}/(c_1(\mf{X}_{\mf{p}}^{\psi}) + c_1(\mf{X}_{\bar{\mf{p}}}^{\psi}))$ has no
    nontrivial finite submodule by Lemma \ref{lem:dumbfinite}, the result then follows from \eqref{eq:maineq} 
    and \eqref{eq:fix1}.
    
    To prove the claim, we identify
    $I_{\mf{p}}^{\omega}$ and $I_{\bar{\mf{p}}}^{\omega}$ with their isomorphic images in $\mf{X}^{\omega}$, so
    the kernel of $I_{\mf{p}}^{\omega} \oplus I_{\bar{\mf{p}}}^{\omega}  \to \mf{X}^{\omega}$ is identified 
    with $(I_{\mf{p}} \cap I_{\bar{\mf{p}}})^{\omega}$, and similarly with the double duals.
    By the exact sequence    
    $$0 \to I_{\mf{p}}^{\omega} \to (I_{\mf{p}}^{\omega})^{**} \to \zp\ps{\Gamma/\Gamma_{\mf{p}}}^{\iota}(1) \to 0$$
    that follows from \eqref{eq:torturous},
    we see that $I_{\mf{p}}^{\omega}$ is contained in the ideal $I$ of $\Lambda \cong (I_{\mf{p}}^{**})^{\omega}$ with $\Lambda/I
    \cong \zp(1)$.  This means that 
    the intersection $(I_{\mf{p}}\cap I_{\bar{\mf{p}}})^{\omega}$ is contained in $I$ times the
    free rank one $\Lambda$-submodule $(I_{\mf{p}}^{**}\cap I_{\bar{\mf{p}}}^{**})^{\omega}$ of $(\mf{X}^{\omega})^{**}$.
    As the kernel of $\zp(1) \to X^{\omega}$ is isomorphic to 
    $(I_{\mf{p}}^{**}\cap I_{\bar{\mf{p}}}^{**})^{\omega}/(I_{\mf{p}}\cap I_{\bar{\mf{p}}})^{\omega}$, which has
    $\zp(1)$ as a quotient, the claim follows.  
\end{proof}

Let $\Omega_W = W\ps{\mc{G}}$ and $\Lambda_W = W\ps{\Gamma}$, where $W$ denotes the Witt vectors of $\bar{\mb{F}}_p$.
Let $\mc{L}_{\mf{p},\psi}$ denote the element of $\La_W \cong \Omega_W^{\psi}$ that determines the two-variable $p$-adic $L$-function for $\mf{p}$ and $\omega\psi^{-1}$.  Let $X_W^{\psi}$ denote the completed tensor product of $X^{\psi}$ with $W$
over $\mc{O}_{\psi}$.
Together with the Iwasawa main conjecture for $K$,  
Theorem \ref{thm:mainresult1} implies the following result.

\begin{theorem}
\label{thm:imagquad} Suppose that $E$ is imaginary quadratic and $p$ splits in $E$.
If both $X^{\psi}$ and $X^{\omega\psi^{-1}}$
are pseudo-null $\La_{\psi}$-modules,  
then there is an equality of second Chern classes
\begin{equation}
\label{eq:chernequal}
c_2\left (\frac{\Lambda_W}{(\mc{L}_{\mf{p},\psi}, \mc{L}_{\bar{\mf{p}}, \psi} )}\right ) =
c_2(X_W^{\psi}) + c_2((X_W^{\omega\psi^{-1}})^{\iota}(1)).
\end{equation}
These Chern classes 
have a characteristic symbol
with component at a codimension one prime $P$ of $\La_W$ equal to the Steinberg symbol
$\{ \mc{L}_{\mf{p}, \psi}, \mc{L}_{\bar{\mf{p}}, \psi} \}$
if $\mc{L}_{\bar{\mf{p}}, \psi}$ is not a unit at $P$, and with other components trivial.
\end{theorem}

\begin{proof} 
By \cite[Corollary III.1.11]{deShalit}, 
the main conjecture as proven in \cite[Theorem 2(i)]{Rubin2} implies that 
that $\mc{L}_{\mf{p}, \psi}$ generates $c_1(\mf{X}_{\mf{p}}^{\psi})\La_W$.
We have $$c_2(\alpha(X_W^{\omega\psi^{-1}})(1)) = c_2((X_W^{\omega\psi^{-1}})^{\iota}(1))$$
by Proposition \ref{prop:adjoint}.  We then apply Proposition \ref{prop:2ndsymb}.
\end{proof}

\begin{remark}
	Supposing that both $X^{\psi}$ and $X^{\omega\psi^{-1}}$ are pseudo-null, 
	the Tate twist of the result of applying $\iota$ to the sequence \eqref{eq:maineq1} reads exactly as the analogous sequence
	for the character $\omega\psi^{-1}$ in place of $\psi$.
	The functional equation of
	Lemma \ref{relationlplpbar}(b) yields an isomorphism
	$$
		(\Omega_W^{\psi}/(\mc{L}_{\mf{p}, \psi}, \mc{L}_{\bar{\mf{p}}, \psi} ))^{\iota}(1)
		\cong 
		\Omega_W^{\omega \psi^{-1}}/(\mc{L}_{\bar{\mf{p}}, \omega\psi^{-1}},\mc{L}_{\mf{p}, \omega\psi^{-1}}),
	$$
	 of the middle terms of these sequences.
\end{remark}

This implies the following codimension two Iwasawa-theoretic analogue of the Herbrand-Ribet theorem
in the imaginary quadratic setting, as mentioned in the introduction.  Note that in this 
analogue we must treat the eigenspaces $X^{\psi}$ and $X^{\omega \psi^{-1}}$ together.  

\begin{cor}
\label{cor:Herbrand-Ribet} 
Suppose that $\psi \neq 1, \omega$.
The Iwasawa modules $X^{\psi}$ and $X^{\omega\psi^{-1}}$  are both trivial if and only if at least one of
$\mc{L}_{\mf{p}, \psi}$ or $\mc{L}_{\bar{\mf{p}}, \psi}$ is a unit in $\La_W$.  
\end{cor}

\begin{proof}

If $X^{\psi}$ is not pseudo-null, then so are both $\mf{X}_{\mf{p}}^{\psi}$ and $\mf{X}_{\bar{\mf{p}}}^{\psi}$.
So, by the main conjecture proven by Rubin (see Theorem \ref{thm:imquadmc}), neither $\mc{L}_{\mf{p},\psi}$
nor $\mc{L}_{\bar{\mf{p}},\psi}$ are units.  If $X^{\omega \psi^{-1}}$ is not pseudo-null, then $\mc{L}_{\mf{p},\omega \psi^{-1}}$ and $\mc{L}_{\bar{\mf{p}},\omega \psi^{-1}}$ are similarly not units.  By the functional
equation of Lemma \ref{relationlplpbar}(b), this implies that $\mc{L}_{\bar{\mf{p}},\psi}$
nor $\mc{L}_{\mf{p},\psi}$ are non-units as well.

If $X^{\psi}$ and $X^{\omega\psi^{-1}}$ are both pseudo-null, then the exact sequence \eqref{eq:maineq1} 
of Theorem \ref{thm:imagquad} shows that $X^{\psi}$ and $X^{\omega \psi^{-1}}$ are both finite
if and only if the quotient $\Omega^{\psi}/(c_1(\mf{X}_{\mf{p}}^{\psi}) +  c_1(\mf{X}_{\bar{\mf{p}}}^{\psi}))$
is finite, which cannot happen unless it is trivial
by Lemma \ref{lem:dumbfinite}. Since $\psi \neq \omega$, again noting \eqref{eq:maineq1}, this happens if
and only if both $X^{\psi}$ and $X^{\omega\psi^{-1}}$ are trivial as well.  
By the main conjecture, the quotient is trivial if and only at least one of $\mc{L}_{\mf{p},\psi}$ and $\mc{L}_{\bar{\mf{p}},\psi}$ is a unit in $\La_W$. 
\end{proof}

\begin{example} \label{imquadexs}
	Suppose that $\psi$ is cyclotomic, so extends to an abelian character of $\widetilde{\Delta} = \Gal(F/\Q)$, 
	that $E \not\subseteq \Q(\mu_p)$, and that
	$\psi \neq 1,\omega$.  Then $X^{\psi}$ is nontrivial if and only if the 
	$\psi$-eigenspace under $\Delta$ of the unramified Iwasawa module $X_{\cyc}$ over the cyclotomic $\zp$-extension
	$F_{\cyc}$ of $F$ is nontrivial.  That is, since all primes over $p$ are unramified in $K/F_{\cyc}$, the map
	from the $\Gal(K/F_{\cyc})$-coinvariants of $X$ to $X_{\cyc}$ is injective with cokernel isomorphic to 
	$\Gal(K/F_{\cyc})$, which has trivial $\Delta$-action.
	We extend $\psi$ and $\omega$ in a unique way to odd characters $\tilde{\psi}$ and $\tilde{\omega}$ of 
	$\widetilde{\Delta}$.  Identify the quadratic character $\kappa$ of $\Gal(E/\Q)$ with a character of $\widetilde{\Delta}$
	that is trivial on $\Delta$.  
	
	The $\psi$-eigenspace  of $X_{\mr{cyc}}$ under $\Delta$
	is the direct sum of the two eigenspaces $X_{\mr{cyc}}^{\tilde{\psi}}$ and $X_{\mr{cyc}}^{\tilde{\psi}\kappa}$ under $\widetilde{\Delta}$.
	By the cyclotomic main conjecture (Theorem \ref{cyclmc}), the Iwasawa module $X_{\mr{cyc}}^{\tilde{\psi}}$ is nontrivial if and only if the appropriate Kubota-Leopoldt $p$-adic $L$-function is not a unit.
	This in turn occurs if and only if $p$ divides the Kubota-Leopoldt $p$-adic $L$-value 
	$$
		L_p(\tilde{\omega} \tilde{\psi}^{-1},0) = (1-\tilde{\psi}^{-1}(p))L(\tilde{\psi}^{-1},0).
	$$  
	The value $L(\tilde{\psi}^{-1},0)$ is the negative of the generalized Bernoulli number $B_{1,\tilde{\psi}^{-1}}$. 
	We have $\tilde{\psi}^{-1}(p) = 1$ if and only if $\tilde{\psi}$ is locally trivial at $p$, in which case the $p$-adic $L$-function is said to have
	an exceptional zero.  By the usual reflection principle (see also Remark \ref{nofinite}), 
	if $X_{\mr{cyc}}^{\tilde{\psi}\kappa}$ is nonzero, then so is 
	$X_{\mr{cyc}}^{\tilde{\omega}\tilde{\psi}^{-1}\kappa}$.
	
	Similarly, the unique extension of $\omega\psi^{-1}$ to an odd character
	of $\widetilde{\Delta}$ is $\tilde{\omega}\tilde{\psi}^{-1}\kappa$, and 
	$X_{\mr{cyc}}^{\tilde{\omega}\tilde{\psi}^{-1}\kappa}$ is nontrivial if and only if $\mc{L}_{\tilde{\omega}\tilde{\psi}^{-1}\kappa}$ is not 
	a unit, which is to say that  $p$ divides $L_p(\tilde{\psi}\kappa,0)$, or equivalently
	that either $p \mid B_{1,\tilde{\omega}^{-1}\tilde{\psi}\kappa}$ or $\tilde{\omega}\tilde{\psi}^{-1}\kappa$ is locally trivial at $p$.
	If $X_{\mr{cyc}}^{\tilde{\omega}\tilde{\psi}^{-1}} \neq 0$, then $X_{\mr{cyc}}^{\tilde{\psi}} \neq 0$.
	
	Typically, when $X_{\mr{cyc}}^{\tilde{\psi}}$ is nonzero, $X_{\mr{cyc}}^{\tilde{\omega}\tilde{\psi}^{-1}}$ and
	$X_{\mr{cyc}}^{\tilde{\omega}\tilde{\psi}^{-1}\kappa}$ are trivial.
	For example, if $p = 37$, then $37 \mid B_{1,\tilde{\omega}^{31}}$ (and $X_{\cyc}^{\tilde{\omega}^{31}} = 0$), 
	but $37 \nmid B_{1,\tilde{\omega}^5\kappa}$
	for $\kappa$ the quadratic character of $\Gal(\Q(i)/\Q)$.
	However, it can occur, though relatively infrequently, that both $B_{1,\tilde{\psi}^{-1}}$ and $B_{1,
	\tilde{\omega}^{-1}\tilde{\psi}\kappa}$ 
	are divisible by $p$.  A cursory 
	computer search using this $\kappa$ revealed many examples in the case one of the $p$-adic $L$-functions has an exceptional 
	zero, e.g., for $p = 5$ and $\tilde{\psi}$ a character of conductor $28$ and order $6$, and other examples in the 
	cases that neither does, e.g., with $p = 5$ and $\tilde{\psi}$ a character of conductor $555$ and order $4$.  
\end{example}

\subsection{Two further rank one cases} 
\label{ss:onecomplexplace}

We will briefly indicate generalizations of Theorem \ref{thm:mainresult1} which can be proved in the remaining
two cases when $\rank_{\La_{\psi}} \mf{X}^{\psi} = 1$.  Our field $E$ will have at most one complex place, but it
will not be $\Q$ or imaginary quadratic.  In view of Lemma \ref{lem:rankone}, the two cases to consider
are when (i) $E$ has exactly one complex place and the character $\psi$ is even at all real places, and (ii) $E$
is totally real and $\psi$ is odd at exactly one real place.

For any set of primes $T$ of $E$, we let $\mf{X}_T$ denote the $T$-ramified
Iwasawa module over $K$.  If $T = \{\mf{p}\}$, we set $\mf{X}_{\mf{p}} = \mf{X}_T$.
Suppose that we are given $n$ degree one primes 
$\mf{p}_1, \ldots, \mf{p}_n$ of $E$ over $p$, and set $T = \{\mf{p}_1,\ldots,\mf{p}_n\}$, $\Sigma = S_f - T$ and 
$\Sigma_i = S_f-\{\mf{p}_i\}$ for $i \in \{1, \ldots, n\}$.

\begin{theorem} \label{1cplx}
	Let $E$ be a number field with exactly one complex place and at least one real place, and suppose
	that $\psi$ is even at all real places of $E$.  Assume that Leopoldt's conjecture holds for $E$, so $r = 2$.
	Furthermore, suppose that $\mf{X}_{\Sigma_i}^{\psi}$ is $\Lambda_{\psi}$-torsion for all $i \in \{1,\ldots,n\}$.
	Assume that $r_{\mf{p}} = 2$ for all $\mf{p} \in S_f$.
	If $X^{\omega\psi^{-1}}$ is pseudo-null, then there is an exact sequence of $\Omega$-modules
	$$
		0 \to \mf{X}_{\Sigma}^{\psi} \to 
		\frac{\Omega^{\psi}}{\sum_{i=1}^n c_1(\mf{X}_{\Sigma_i}^{\psi}) } \to
		\alpha(\mc{H}_{\Sigma}^{\omega\psi^{-1}})(1) \to 0
	$$
	where $\mc{H}_{\Sigma}$ is as in \eqref{Hseq}.
\end{theorem}

\begin{proof} As in the imaginary quadratic case, the strategy is to control the terms and vertical homomorphisms of the 
diagram \eqref{Tdiag}.  The three steps needed to do this are (i) show that decomposition groups can be replaced
by inertia groups, (ii) show that the appropriate eigenspaces of the $\rE_\Lambda^1$ groups in \eqref{Tdiag} are 
trivial, and (iii) use Iwasawa cohomology groups to relate the $\rE_\Lambda^2$ groups in \eqref{Tdiag} to
$\mc{H}_{\Sigma}$, which is an extension of an unramified Iwasawa module.  

	Note that $\zp(1)^{\psi} = 0$ since $\psi$ is even at a real place.
	For $\mf{p} \in T$, the field $K_{\mf{p}}$ contains the unramified $\zp$-extension of $\qp$
	because $\mf{p}$ has degree $1$ and $r_{\mf{p}} = 2 = r$.  
	By assumption and Lemmas \ref{unramsplit}(a), \ref{inertia}, and \ref{extKp}, the map $X \to X'$ has pseudo-null kernel, 
	$I_T = D_T$ 
	and $\rE_{\La}^i(\mc{K}_T) = 0$ for $i \neq 2$.
	By our pseudo-nullity assumption, we have $\rE_{\La}^1(X^{\omega\psi^{-1}}) = 0$.  It follows from
	Lemma \ref{K0} and \eqref{H2Ext} that $\rE_{\La}^1(\rH^2_{\Iw}(K,\zp))^{\psi}= 0 $.  Since 
	$\mc{H}_{\Sigma}^{\omega\psi^{-1}}$ is a submodule of the pseudo-null module 
	$\rH^2_{\Iw}(K,\zp(1))^{\omega\psi^{-1}}$, we have 
	$\rE_{\La}^1(\mc{H}_{\Sigma}^{\omega\psi^{-1}}) = 0$ as well.

	 As 
	 $\zp(1)^{\psi} = 0$, we have $(\mc{K}_T)_0^{\omega\psi^{-1}} = \mc{K}_T^{\omega\psi^{-1}}$.
	 Hence, we have an exact sequence
		$$
		0 \to (\rE_{\La}^2(\mc{K}_T)(1))^{\psi} \to (\rE_{\La}^2(\rH^2_{\Iw}(K,\zp)))^{\psi} \to (\rE_{\La}^2(\mc{H}_{\Sigma})(1))^{\psi} \to 0.
	$$
	Taking $\psi$-eigenspaces of the terms of diagram \eqref{Tdiag} and applying the snake lemma, we obtain 
	an exact sequence
	$$
		0 \to \mf{X}_T^{\psi} \to \coker \phi_T^{\psi} \to \rE_{\La}^2(\mc{H}_{\Sigma}^{\omega\psi^{-1}})(1) \to 0.
	$$
	As $\mf{X}^{\psi}$ and $D_{\mf{p}}^{\psi}$ for all $\mf{p} \in T$ have 
	$\Lambda_{\psi}$-rank one by Lemma \ref{lem:rankone}, the argument is now just as before,
	the assumption on $\mf{X}_{\Sigma_i}^{\psi}$ ensuring the injectivity of $\phi_{\mf{p}_i}^{\psi}$.
\end{proof}

This yields the following statement on second Chern classes.

\begin{cor}
	Let the notation and hypotheses be as in Theorem \ref{1cplx}.  Suppose in addition that $n = 2$ and
	$\mf{X}_{\Sigma}^{\psi}$ is pseudo-null.  
	Let $f_i$ be a generator for the ideal 
	$c_1((\mf{X}_{\Sigma_i})^{\psi})$ of $\Lambda_{\psi}$. 
	Then the sum of second Chern classes
	$$
		c_2(\mf{X}_{\Sigma}^{\psi}) + c_2((\mc{K}_{\Sigma}^{\omega\psi^{-1}})^{\iota}(1)) +
		c_2(((X')^{\omega\psi^{-1}})^{\iota}(1))
	$$
	has a characteristic symbol 
	with component at a codimension one prime $P$ of $\Lambda_{\psi}$ the Steinberg symbol
	$ \{ f_1, f_2 \}$
	if $f_2$ is not a unit at $P$, and trivial otherwise.
\end{cor}

\begin{proof}
	By the exact sequence \eqref{Hseq} for $\mc{H}_{\Sigma}$, 
	Lemma \ref{lem:smallseq}, and the fact that $\zp^{\omega\psi^{-1}} = 0$, we have
	$$
		c_2(\alpha(\mc{H}_{\Sigma}^{\omega\psi^{-1}})(1)) = c_2(\alpha(\mc{K}_{\Sigma}^{\omega\psi^{-1}})(1)) +
		c_2(\alpha((X')^{\omega\psi^{-1}})(1)).
	$$
	The result then follows from Theorem \ref{1cplx}, as in the proof of Theorem \ref{thm:imagquad}.
\end{proof}

In the following, it is not necessary to suppose Leopoldt's conjecture, if one simply allows $K$ to 
be the cyclotomic $\zp$-extension of $F$.

\begin{theorem}
	Let $E$ be a totally real field other than $\Q$, and let $\psi$ be odd at exactly one real place of $E$.
	Assume that Leopoldt's conjecture holds, so $r = 1$.  
	Furthermore, suppose that $\mf{X}_{\Sigma_i}^{\psi}$ is $\Lambda_{\psi}$-torsion for all $i \in \{1,\ldots,n\}$.
	If $(X')^{\omega\psi^{-1}}$ is finite, then there is an
	exact sequence of $\Omega$-modules 
	$$
		(\mc{K}_{\Sigma}^{\omega\psi^{-1}})^{\iota}(1) \to \mf{X}_{\Sigma}^{\psi} 
		\to \frac{\Omega^{\psi}}{\sum_{i=1}^n c_1(\mf{X}_{\Sigma_i}^{\psi}) } 
		\to ((X')^{\omega\psi^{-1}})^{\vee}(1) \to 0.
	$$
\end{theorem}

\begin{proof}
	The argument is much as before: we take the $\psi$-eigenspace of the terms of diagram \eqref{Tdiag} with 
	$T = \{\mf{p}_1,\ldots,\mf{p}_n\}$.
	We have $\rE^2_{\La}(\mc{K}) = 0$ and $\rE^1_{\La}(\mc{K}) \cong \mc{K}^{\iota}$.  The map $D_T \to
	D_T^{**}$ in the diagram can be replaced by $I_T \to I_T^{**}$, as in the proof of Theorem \ref{thm:cyclcase}.
	Applying the snake lemma gives the stated sequence.  
\end{proof}

Although this is somewhat less strong than our other results in general (when $\omega\psi^{-1}|_{\Delta_{\mf{p}}} = 1$
for some $\mf{p} \in \Sigma$), we have the following interesting corollary.

\begin{cor}
	Suppose that $E$ is real quadratic, $p$ is split in $E$ into two primes $\mf{p}_1$ and $\mf{p}_2$, and
	the character $\psi$ is odd at exactly one place of $E$.  If $X^{\psi}$ and
	$(X')^{\omega\psi^{-1}}$ are finite, then there is an
	exact sequence of finite $\Omega$-modules 
	$$
		0 \to X^{\psi} 
		\to \frac{\Omega^{\psi}}{c_1(\mf{X}_{\mf{p}_1}^{\psi}) + c_1(\mf{X}_{\mf{p}_2}^{\psi}) } 
		\to ((X')^{\omega\psi^{-1}})^{\vee}(1) \to 0.
	$$
\end{cor}

We can make this even more symmetric, replacing $X$ by $X'$ on the left, if we also replace $\mf{X}_{\mf{p}_i}^{\psi}$
by its maximal split-at-$\mf{p}_{3-i}$ quotient for $i \in \{1,2\}$, and supposing only that $(X')^{\psi}$ is finite.  
We of course have the corresponding statement on second Chern classes.

\section{A non-commutative generalization}
\label{s:generalization}

The study of non-commutative generalizations of the first Chern
class main conjectures discussed in Section \ref{s:conj} has been 
very fruitful.  See \cite{Coates1}, for example, and its references.  
 We now indicate briefly a non-commutative generalization of 
 Theorems \ref{thm:mainresult1} and \ref{thm:imagquad} concerning
 second Chern classes.  
 
 We make the same assumptions as in Subsection \ref{ss:imaginaryquad}.
 Namely, $E$ is imaginary quadratic, and $p$ is an odd prime that splits 
into two primes $\mf{p}$ and $\bar{\mf{p}}$ in $E$.
Let $\psi$ be a one-dimensional $p$-adic character of
the absolute Galois group of $E$ of finite order prime to $p$ with fixed field of its kernel $E_{\psi}$.
Let $F = E_{\psi}(\mu_p)$.  
Let $\omega$ denote the Teichm\"uller character of $\Delta = \Gal(F/E)$.
Let $\widetilde{E}$ denote the compositum of all $\zp$-extensions of $E$, and 
let $K$ be the compositum of $\widetilde{E}$ with $F$.    
Let $S$ be the set of primes of $E$ above $p$ and $\infty$, so $S_f = \{\mf{p},\bar{\mf{p}}\}$.

We suppose in addition that $F$ is Galois over $\Q$.  Let $\sigma$ be a complex conjugation
 in $\mathrm{Gal}(F/\Q)$, and let $H = \{e,\sigma\}$.  Then  $\widetilde{\Delta} = \mathrm{Gal}(F/\mathbb{Q})$ is a semi-direct product of the abelian group $\Delta$ with $H$.  
 The group $H$ acts on $\Delta$ and $\Gamma = \mathrm{Gal}(K/F) \cong \mathbb{Z}_p^2$ by conjugation. 
 Let $\tau$ be the character of an $n$-dimensional irreducible $p$-adic representation of $\widetilde{\Delta}$.  Then $n \in \{1,2\}$.    If $n =1$, then $\tau$ restricts to a one-dimensional character $\psi$ of $\Delta$.  If $n = 2$, then the representation corresponding to $\tau$ restricts to a direct sum of two one-dimensional representations $\psi$ and $\psi \circ \sigma$ of $\Delta$.  So, the orbit of $\psi$ under the action of $\sigma$ has order $n$.
 
Let $A_{\tau}$ denote the direct factor of $\mc{O}_{\psi} \otimes_{\zp} \Omega=\mc{O}_{\psi}\ps{\mc{G}}$ obtained by applying the idempotent in 
$\mc{O}_{\psi}[\widetilde{\Delta}]$-attached
to $\tau$, where $\mc{O}_{\psi}$ is as before.  Then $A_{\tau} = \Omega^{\psi}$ if $n = 1$ and $A_{\tau} = \Omega^{\psi} \times \Omega^{\psi \circ \sigma}$ if $n = 2$.
 The $H$-action on $A_{\tau}$ is compatible with the
 $\Omega$-module structure and the action of $H$  on $\Omega$.   Thus,
 $A_{\tau}$ is a module over the twisted group ring $B_{\tau} = A_{\tau}\langle H \rangle$, 
 which itself is a direct factor of $\mc{O}_{\psi}\ps{\Gal(K/\Q)}$.  

 The following non-commutative generalization of Theorem \ref{thm:mainresult1}
 follows from the compatibility with the $H$-action of the arguments used in the proof of said theorem.
 
 \begin{supproposition}
 \label{prop:star}
     Suppose that $X^{\omega\psi^{-1}}$ is pseudo-null as a $\La_{\psi}$-module.
     If $n = 1$, then the sequence \eqref{eq:maineq1} for $\psi$ 
     is an exact sequence of modules for the non-commutative ring $B_{\tau}$.
     If $n = 2$, then the direct sum of the sequences \eqref{eq:maineq1} for $\psi$ and $\psi \circ \sigma$
     is an exact sequence of $B_{\tau}$-modules.
 \end{supproposition} 
   
 To generalize Theorem \ref{thm:imagquad}, we first extend the approach
to Chern classes used in Subsection \ref{ss:generalsetup} to the context of non-commutative algebras which are finite over
 their centers.  
 (For related work on non-commutative Chern classes, see \cite{CPT}.)  
 
The twisted group algebra $B_\tau$ is a free rank four module over its center $Z_{\tau} = 
A_{\tau}^H$.  
  Suppose that $M$ is a finitely generated module for $B_\tau$ 
 with support as a $Z_{\tau}$-module of codimension at least $2$.
 Let $Y = \mathrm{Spec}(Z_{\tau})$,
 and let $Y^{(2)}$ be the set of codimension two primes in $Y$.  
 The localization $M_y = (Z_{\tau})_y \otimes_{Z_{\tau}} M$ of $M$ at $y \in Y^{(2)}$ has finite length over the localization 
 $(B_{\tau})_y$.  Let 
 $k(y)$ be the residue field of $y$, and let
 $B_\tau(y) = k(y) \otimes_{Z_{\tau}} B_\tau$.
 Then $B_\tau(y)$ has dimension $4$ as a $k(y)$-algebra.  From a composition
 series for $M_y$ as a $B_\tau(y)$-module, we can define a class $[M_y]$ in the Grothendieck group
 $\rK_0'(B_\tau(y))$ of all finitely generated $B_\tau(y)$-modules.  This leads to a second
 Chern class 
 \begin{equation}
 \label{eq:secondCnon}
 c_{2,B_\tau}(M)=\sum_{y\in Y^{(2)}}  [M_y] \cdot y 
\end{equation}
in the group 
$$\rZ^2(B_\tau)= \bigoplus_{y\in Y^{(2)}} \rK_0'(B_\tau(y)).$$

For $y \in Y^{(2)}$, note that $A_{\tau}(y) = k(y) \otimes_{Z_{\tau}} A_{\tau}$ is a $k(y)$-algebra of dimension $2$ with an action of $H$
over $k(y)$, and $B_\tau(y)$ is the twisted group algebra $A_{\tau}(y)\langle H \rangle$.
Moreover, we have
$$
	k(y) \otimes_{Z_{\tau}} Z_{\tau}[H] \cong k(y)[H],
$$ 
and in this way, both $A_{\tau}(y)$ and $k(y)[H]$ are commutative $k(y)$-subalgebras of $B_{\tau}(y)$.

\begin{suplemma}
\label{lem:restrictg}
    For $y \in Y^{(2)}$, the forgetful functors on finitely generated module categories produced by
    restricting operators from $B_\tau(y)$ to either $A_\tau(y)$ or $k(y)[H]$ induce injections
    \begin{eqnarray}
    \label{eq:injB}
    	&\rK_0'(B_{\tau}(y)) \to \rK_0'(A_{\tau}(y)) = \mathbb{Z} \oplus \mathbb{Z}& \text{ if $n = 2$, and } \\
    \label{eq:injC} 
    	&\rK_0'(B_{\tau}(y)) \to \rK_0'(k(y)[H]) = \mathbb{Z} \oplus \mathbb{Z}& \text{ if $n = 1$.}
    \end{eqnarray}
\end{suplemma}

\begin{proof}
Since $H$ has order $2$ and $A_\tau$ is a $\mathbb{Z}_p$-algebra
    for an odd prime $p$, the surjection $A_\tau \to A_\tau(y)$ gives a surjection $Z_\tau = A_\tau^H \to (A_\tau(y))^H$.
    Thus $k(y) = A_\tau(y)^H$.  Therefore either $A_\tau(y)$ is a Galois \'etale $H$-algebra over $k(y)$ or
    $A_\tau(y)$ is isomorphic to the dual numbers $k(y)[\epsilon]/(\epsilon^2)$ in such a way that
    the generator $\sigma$ for $H$ sends $\epsilon$ to $-\epsilon$.

    There is a homomorphism $\pi^* \colon K'_0(k(y)) \to K_0'(B_\tau(y))$ which sends a 
    finitely generated $k(y)$-module $M$ to the $B_\tau(y)$-module with underlying
    $A_\tau(y)$-module $A_\tau(y) \otimes_{k(y)} M$ and action of $H$ induced by the action of $H$
    on $A_\tau(y)$. 
    
    Suppose to begin with that $A_{\tau}(y)$ is an \'etale algebra over $k(y)$, so $A_\tau(y)$ is a Galois \'etale $k(y)$-algebra with Galois group  $H$. 
 It is shown by descent in the paragraph just before \cite[Lemma 8.4]{tameness}
    that $\pi^*$ is an isomorphism.  The composition of $\pi^*$ with the 
    forgetful homomorphism $f_1 \colon \rK_0'(B_{\tau}(y)) \to \rK_0'(A_{\tau}(y))$ is the homomorphsm
    $h \colon K_0'(k(y)) \to K_0(A_\tau(y))$ induced by tensoring with $A_\tau(y)$ over $k(y)$.  Since $h$ is injective, $f_1$ is injective. 
    
    If $n = 2$, then $H$ permutes the two algebra components of
    $A_{\tau}$, and $A_\tau(y)$ is isomorphic to the \'etale $k(y)$-algebra $k(y) \times k(y)$,
    so we have shown (\ref{eq:injB}).  
    
    Suppose now that $n = 1$ and that $A_\tau(y)$ is 
    \'etale over $k(y)$.   We have shown $A_\tau(y)$ is then a Galois \'etale $k(y)$-algebra with Galois group  $H$. 
    Thus $A_\tau(y)$ has two one-dimensional $k(y)$-eigenspaces under the action of $H$.
    We have also shown that the map $\pi^* \colon K'_0(k(y)) \to K_0'(B_\tau(y))$ induced by tensoring with $A_\tau(y)$
    over $k(y)$ is an isomorphism.
    So from the definition of $\pi^*$, we see that restricting operators from $B_\tau(y)$ to
    $k(y)[H]$ gives an injection
    $$K_0'(B_\tau(y)) \to K_0(k(y)[H]) = \mathbb{Z} \oplus \mathbb{Z}$$
    whose image is the diagonal embedding of $\mathbb{Z}$.  
    
    Finally, suppose $n = 1$ and that $A_\tau(y)$ is not \'etale over $k(y)$.
    We have shown that $A_\tau(y)$ is isomorphic in this case to the dual numbers $k(y)[\epsilon]/(\epsilon^2)$ in such a way that
    a generator for $H$ sends $\epsilon$ to $-\epsilon$.  All simple
    $B_{\tau}(y)$-modules are annihilated by $\epsilon$.  Since $B_{\tau}(y)/\epsilon B_{\tau}(y)$ is
    isomorphic to $k(y)[H]$, we find that the map \eqref{eq:injC} is injective.
\end{proof}

As in Theorem \ref{thm:imagquad}, we must take completed tensor products
over $\mathcal{O}_\psi$ with the Witt vectors $W$ over $\overline{\mathbb{F}}_p$.
In what follows, we abuse notation and omit this $W$ from the notation of the completed tensor products.
That is, from now on 
we let $Z_{\tau}$ denote $W \cotimes{\mc{O}_{\psi}} Z_{\tau}$, and similarly with $A_{\tau}$ and $B_{\tau}$.
We then let $Y = \mathrm{Spec}(Z_{\tau})$, and
we use $k(y)$ to denote the residue field of $Z_{\tau}$ at $y \in Y$, and we define $A_{\tau}(y)$ and $B_{\tau}(y)$
as before.
Note that the analogue of Lemma \ref{lem:restrictg} holds for $y \in Y^{(2)}$,
with $Y^{(2)}$ the subset of codimension $2$ primes in $Y$.

We suppose for the remainder of this section that $X^{\psi}$ and $X^{\omega\psi^{-1}}$ are pseudo-null as $\Lambda_{\psi}$-modules.
In view of Proposition \ref{prop:star}, we have by Theorem \ref{thm:imagquad} the following identity among
non-commutative second Chern classes
\begin{equation}
\label{eq:fancynon}
	c_{2,B_{\tau}}\left( \bigoplus_{\chi \in T}  
	\frac{\Omega_W^{\chi}}{(\mc{L}_{\mf{p},\chi}, \mc{L}_{\bar{\mf{p}},\chi} )} \right ) = 
	c_{2,B_{\tau}}\Biggl(\bigoplus_{\chi \in T} X_W^{\chi}\Biggr) + 
	c_{2,B_{\tau}}\Biggl(\bigoplus_{\chi \in T} (X_W^{\omega\chi^{-1}})^{\iota}(1)\Biggr),
\end{equation}
where $T$ denotes the orbit of $\psi$ (of order $1$ or $2$).
In view of Lemma \ref{lem:restrictg}, to compute \eqref{eq:fancynon} in terms of $p$-adic $L$-functions, it suffices to compute the
analogous abelian second Chern classes via $L$-functions when $B_\tau$ is replaced
by $A_{\tau}$ and by $Z_{\tau}[H]$ and we view the latter two as quadratic algebras over $Z_{\tau}$.

In the case that $n = 2$, a prime $y \in Y^{(2)}$ gives rise to one prime in each of the two factors of
$A_{\tau} = \Omega_W^{\psi} \times \Omega_W^{\psi \circ \sigma}$ by projection.
Note that we can identify $\Omega_W^{\psi}$ and $\Omega_W^{\psi \circ \sigma}$ 
with $\Lambda_W$ so that $Z_{\tau}$ is identified with the diagonal in $\Lambda_W^2$,
and these two primes of $\Lambda_W$ are then equal.  We have
\begin{equation} \label{eq:niceA}
	\rK_0'(A_{\tau}(y)) \cong \Z \oplus \Z,
\end{equation}
the terms being $\rK_0'$ of the residue fields of $\Omega_W^{\psi}$ and $\Omega_W^{\psi \circ \sigma}$
for $y$, respectively.

\begin{supproposition}
\label{prop:2star}
	If $n = 2$, then under the injective map
	$$
		\bigoplus_{y \in Y^{(2)}} \rK_0'(B_{\tau}(y)) \to 
		\bigoplus_{y \in Y^{(2)}} (\Z \oplus \Z)
	$$
	induced by \eqref{eq:injB} and \eqref{eq:niceA}, 
	the class in \eqref{eq:fancynon} is sent to an element with both components
	having a characteristic symbol which at $P \in Y^{(1)}$ is
	equal to the Steinberg symbol
	\begin{eqnarray*}
	& \{ \mc{L}_{\mf{p}, \psi}, \mc{L}_{\bar{\mf{p}}, \psi} \} \in \rK_2(\mr{Frac}(\Lambda_W))
	& \text{if } \mc{L}_{\bar{\mf{p}}, \psi} \text{ is not a unit at $P$,}
	\end{eqnarray*}
	and is zero otherwise.
\end{supproposition}

\begin{proof}
	This is immediate from Theorem \ref{thm:imagquad} in the first coordinate.  The second
	coordinate is the same by Lemma \ref{relationlplpbar}(a) 
	and the above identification of $\Omega_W^{\psi \circ \sigma}$ with $\Lambda_W$, recalling
	Remark \ref{rem:symmetry}.
\end{proof}

In the case that $n = 1$, so $\psi = \tau|_{\Delta}$, 
we have an algebra decomposition 
\begin{equation} \label{eq:algdec}
	Z_{\tau}[H] =  Z_{\tau}^+ \times Z_{\tau}^-
\end{equation}
with the summands corresponding to the trivial and nontrivial one-dimensional
characters of $H$.  
These summands are isomorphic to $Z_{\tau}$ as $Z_{\tau}$-algebras.
We then have a decomposition
\begin{equation}
\label{eq:niceB}
    \rK_0'(k(y)[H]) \cong \mathbb{Z} \oplus \mathbb{Z}
\end{equation}
with the terms being $\rK_0'$ of the residue fields of $Z_{\tau}^+$ and $Z_{\tau}^-$, respectively, for the
images of $y$.

There are pro-generators $\gamma_1, \gamma_2 \in \mathrm{Gal}(K/F)$ 
such that $\sigma(\gamma_1) = \gamma_1$ and $\sigma(\gamma_2) = \gamma_2^{-1}$.
The ring $A_{\tau} = \Omega^{\psi}$ is $Z_{\tau}[\lambda]$, where $\lambda = \gamma_2-\gamma_2^{-1}$.
Note that $\sigma(\lambda) = -\lambda$ and $\lambda^2 \in Z_{\tau}$.
It follows that we have an isomorphism $Z_{\tau}[H] \xrightarrow{\sim} \Omega^{\psi}$ of $Z_{\tau}[H]$-modules 
taking $1$ to $(1+ \lambda)/2$ and $\sigma$ to $(1 - \lambda)/2$.

The element $\sigma$ permutes the $p$-adic $L$-functions $\mc{L}_{\mf{p},\psi}$ and
 $\mc{L}_{\bar{\mf{p}},\psi}$.  Define 
\begin{eqnarray*}
	\mc{L}_{\tau}^+ = \mc{L}_{\mf{p},\psi} + \mc{L}_{\bar{\mf{p}},\psi} 
    	&\mathrm{and}&
    	\mc{L}_{\tau}^- = \mc{L}_{\mf{p},\psi} - \mc{L}_{\bar{\mf{p}},\psi}.
\end{eqnarray*}

 \begin{supproposition}
 \label{prop:1star}
 	If $n = 1$, then under the injective map
	$$
		\bigoplus_{y \in Y^{(2)}} \rK_0'(B_{\tau}(y)) \to \bigoplus_{y \in Y^{(2)}} (\Z \oplus \Z)
	$$
	induced by \eqref{eq:injC} and \eqref{eq:niceB},
	the class in \eqref{eq:fancynon} is sent to an element
	that in the first and second components, respectively, has a characteristic symbol
	which at $P \in Y^{(1)}$ is equal to 
	\begin{eqnarray*} 
	\label{eq:stein1}
   	 \{\lambda\mc{L}_{\tau}^-, \mc{L}_{\tau}^+ \} \in \rK_2(\mathrm{Frac}(Z_{\tau}^+))
	 && \mbox{if $\mc{L}_{\tau}^+$ is not a unit at $P$}, \\
    	\label{eq:stein2}
    	\{ \mc{L}_{\tau}^-, \lambda \mc{L}_{\tau}^+  \} \in \rK_2(\mathrm{Frac}(Z_{\tau}^-))&&
    	\mbox{if $\lambda \mc{L}_{\tau}^+$ is not a unit at $P$},
    	\end{eqnarray*}
	and is zero otherwise.
\end{supproposition}

\begin{proof}
    The decomposition \eqref{eq:algdec} and the isomorphism $\Omega^{\psi} \cong Z_{\tau}[H]$
    induce an isomorphism
    $$
    	\frac{\Omega^{\psi}}{(\mc{L}_{\mf{p},\psi}, \mc{L}_{\bar{\mf{p}},\psi} )} 
    	\cong \frac{Z_{\tau}^+}{(\mc{L}_{\tau}^+, \lambda\mc{L}_{\tau}^- )} \oplus
     	\frac{Z_{\tau}^-}{(\lambda \mc{L}_{\tau}^+, \mc{L}_{\tau}^- )}.
    $$
    From the two summands on the right, together with 
    Proposition \ref{prop:2ndsymb}, we arrive at the two components of
    the non-commutative second Chern class of \eqref{eq:fancynon}, as in the statement of the proposition.
\end{proof}

 \appendix
 \section{Results on Ext-groups}
\label{appendix}

In this appendix, we derive some facts about modules over power series rings.  
For our purposes, let $\mc{O}$ be the valuation ring of a finite extension of $\qp$.
Let $\Gamma = \zp^r$ for some $r \ge 1$, and denote its standard topological generators by $\gamma_i$ for $1 \le i \le r$.  Set  $\La = \mc{O}\ps{\Gamma} = \mc{O}\ps{t_1,\ldots,t_r}$, where $t_i = \gamma_i-1$.  

As in Subsection \ref{ss:gensetup}, we use the following notation for a finitely generated $\La$-module $M$.
We set $\rE^i_{\La}(M) = \Ext^i_{\La}(M,\La)$, and we set $M^* = \rE^0_{\La}(M) = \Hom_{\La}(M,\La)$. 
Moreover, $M^{\vee}$ denotes the Pontryagin dual, and $M_{\mr{tor}}$ denotes the $\La$-torsion submodule of $M$. 

We will be particularly concerned with $\La$-modules of large codimension, but we first recall a known result on much larger modules.

\begin{suplemma}
\label{lem:double} 
	Let $M$ be a $\La$-module of rank one. Then $M^{**}$ is free.
\end{suplemma}

\begin{proof}
	The canonical map
	$(M/M_{\mr{tor}})^* \to M^*$ is an isomorphism, so we may assume that $M_{\mr{tor}} = 0$.
	We may then identify $M$ with a nonzero ideal of $\La$.  
	The dual
	of a finitely generated module is reflexive, so we are reduced to showing that a reflexive ideal $I$ of 
	$\La$ is principal.  
	For each height one ideal $P$ of $\La$, let $\pi_P$ be a 
	uniformizer of $\La_P$, and let $n_P \ge 0$ be such that $\pi_P^{n_P}$ generates $I_P$.  Let $s$ be the 
	finite product of the $\pi_P^{n_P}$.  Then the principal 
	ideal $J = s\La$ 
	is obviously reflexive and has the same localizations at height one primes as $I$.  As $I$ and $J$ are
	reflexive, they are the intersections of their localizations at height one primes, so $I = J$.
\end{proof}

For a finitely generated $\La$-module $M$, we have $\rE_{\La}^i(M) = 0$ for all $i > r+1$.  Since $\La$ is Cohen-Macaulay (in fact, regular), the minimal $j = j(M)$ such that $\rE_{\La}^j(M) \neq 0$ is also the height of the annihilator of $M$.  (We take $j = \infty$
for $M = 0$.)  In particular, $M$ is torsion (resp., pseudo-null) if $j \ge 1$ (resp., $j \ge 2$), and $M$ is finite if $j = r+1$.

\begin{suplemma} \label{lem:groth}
	For $j \ge 0$, let $G_0(j)$ be the Grothendieck group of the category of finitely generated $\La$-modules 
	$M$ with $j(M) \ge j$.  The quotient of $G_0(j)$ by the image of the natural homomorphism $G_0(j+1) \to G_0(j)$ is generated by 
	the classes of modules of the form $\La/P$ with $P$ a prime ideal of height $j$.
\end{suplemma}

\begin{proof} Suppose $M$ is a finitely generated $\La$-module with $j(M) \ge j$.   The codimension of the support
of $M$ is then at least $j$, and the localization of $M$ at every prime $P$ of codimension $j$ is of finite length over $\La_{P}$.  
If $P$ is  in the support of $M$, then
$P$ is an associated prime of $M$ by \cite[(7.D) Thm. 9]{Matsumura}.  Hence there is an $m \in M$ such that $\La \cdot m$ is isomorphic to
$\La/P$.  Thus there is an exact sequence
$$0 \to \La/P \to M \to M' \to 0$$
in which $j(M') \ge j$ and the sum $s(M')$ of the lengths of $M'$ at codimension $j$ primes of $\La$ is one less than $s(M)$.  The lemma now follows by induction on $s(M)$.
\end{proof}

We also have the following.

\begin{suplemma}
\label{lem:dumbfinite} 
	Let $1 \le d \le r$, and let $f_i$ for $1 \le i \le d$ be elements of $\La$
	such that $(f_1,\ldots,f_d)$ has height $d$. Then $M = \La/(f_1,\ldots,f_d)$
	has no nonzero $\La$-submodule $N$ with $j(N) \ge d+1$. 
\end{suplemma}

\begin{proof}Since $\La$ is a Cohen-Macaulay local ring, we know from  \cite[Theorem 17.4(iii)]{Matsumura}
that the ideal $(f_1,\ldots,f_d)$ has height $d$ if and only if $f_1,\ldots,f_d$ form a regular sequence
in $\La$. Then $M$ is a Cohen-Macaulay module by \cite[Theorem 17.3(ii)]{Matsumura}, and it
has no embedded prime ideals by \cite[Theorem 17.3(i)]{Matsumura}.  If $M$ has a nonzero $\La$-submodule
$N$ with $j(N)\ge d+1$, then a prime ideal of $\La$ of height strictly greater than $d$
will be the annihilator of a nonzero element of $M$.  This contradicts the
fact that $M$ has no embedded primes.    
\end{proof}

Let $\mc{G}$ be a profinite group containing $\Gamma$ as an open normal subgroup, and set $\Omega = \mc{O}\ps{\mc{G}}$.   For a left (resp., right) $\Omega$-module $M$, the groups $\rE_{\La}^i(M)$ have the structure of right (resp., left) $\Omega$-modules (see \cite[Proposition 2.1.2]{ls}, for instance).

We will say that a finitely generated $\Omega$-module $M$ 
is small if $j(M) \ge r$ as a (finitely generated) $\La$-module.   We use the notation $\finite$ to denote an unspecified finite module occurring in an exact sequence, and the notation $M_{\fin}$ to denote the maximal finite $\La$-submodule of $M$. 
Let $M^{\dagger} = (M \otimes_{\zp} \bigwedge^r \Gamma)^{\vee}$, which is isomorphic to $M^{\vee}$ if $\mc{G}$ is abelian.

We derive the following from the general study of Jannsen \cite{jannsen}.   In \cite[Lemma 5]{jan-spec}, a form of this is proven for modules finitely generated over $\zp$.  Its part (b) gives an explicit description of the 
Iwasawa adjoint of a small $\Omega$-module $M$ in the case that $\Omega$ has sufficiently large center.  We do not use
this in the rest of the paper, but for comparison with the classical theory, the explicit description appears to be of interest.
 
\begin{supproposition} \label{Er}  
	Let $M$ be a small (left) $\Omega$-module.  
	\begin{itemize}
	\item[(a)] 
	There exist canonical right $\Omega$-module isomorphisms $\rE^{r+1}_{\La}(M) \cong M_{\fin}^{\dagger}$, 
	and these are natural in $M$.
	\item[(b)]
	Given a non-unit $f \in \La$ that is central in $\Omega$ and 
	not contained in any height $r$ prime ideal in the support of $M$,
	there exists a canonical right $\Omega$-module homomorphism
	$$
		\rE_{\La}^r(M) \cong \varprojlim_n\, (M/f^nM)^{\dagger},
	$$
	the inverse limit taken with respect to maps $(M/f^{n+1}M)^{\vee} \to (M/f^nM)^{\vee}$ induced by
	multiplication by $f$.  
	The maximal finite submodule of $\rE_{\La}^r(M)$ is zero.
	\end{itemize}
\end{supproposition}

\begin{proof}
	For $i \ge 0$ and a locally compact $\Omega$-module $A$, set
	$$
		{\rm D}_i(A) = \dirlim{U} \rH^i_{\cont}(U,A)^{\vee},
	$$
	where the direct limit is with respect to duals of restriction maps 
	over all open subgroups $U$ of finite index in $\Gamma$.
	The group $\Gamma$ is a duality group (see \cite[Theorem 3.4.4]{nsw}) of 
	strict cohomological dimension $r$, and its dualizing module is the
	$\Omega$-bimodule
	$$
		\textstyle {\rm D}_r(\zp) \cong \varinjlim_{U} \Hom_{\zp}(\La^r U,\zp)^{\vee} \cong \bigwedge^r \Gamma \otimes_{\zp} \qp/\zp.
	$$
	We have ${\rm D}_i(M^{\vee}) = 0$ for $i > r$ and, by duality, we have the first isomorphism in
	$$
		\textstyle {\rm D}_r(M^{\vee}) \cong \Hom_{\Lambda}(M^{\vee},{\rm D}_r(\zp)) \cong 
		(\dirlim{U} M^U) \otimes_{\zp} \bigwedge^r \Gamma.
	$$ 
	By \cite[Theorem 2.1]{jannsen}, we then have canonical and natural isomorphisms
	\begin{equation}
	\label{eq:E3stuff}
		\rE^{r+1}_{\La}(M) \cong ({\rm D}_r(M^{\vee})[p^{\infty}])^{\vee} \cong ((\dirlim{U} M^U)[p^{\infty}])^{\dagger}.
	\end{equation}
	Moreover, by \cite[Corollary 2.6b]{jannsen}, we have that $\rE^i_{\La}(M) = 0$ for $i \ne r+1$ if $M$ happens
	to be finite.
	
	We claim that 
	\begin{equation}
	\label{eq:horror}
	(\dirlim{U} M^U)[p^{\infty}] = M_{\fin}
	\end{equation}
	which will finish the proof of part (a).
	As $\Gamma$ acts continuously on $M$, the left-hand side contains $M_{\fin}$, so 
	it suffices to show that $(\dirlim{U} M^U)[p^{\infty}]$ is finite.  As $M$ is compact, there exist
	an open subgroup $V$ and $k \ge 1$ such that
	$M^V[p^k] = \dirlim{U} M^U[p^{\infty}]$.  As $V \cong \Gamma$,
	it suffices to show that $M^{\Gamma}[p]$ is finite.
	
	By Lemma \ref{lem:groth} and the right exactness of $\rE^{r+1}_{\La}$,
	we are recursively reduced to considering $M$ of the form $\Lambda/P$
	with $P$ a prime ideal of height $r$.  If $p \notin P$, 
	then $M$ has no $p$-power torsion and \eqref{eq:horror} is clear. 
	If $p \in P$, then $\Lambda/P$ is isomorphic to 
	$\mathbb{F}_q\ps{\overline{t}_1,\ldots,\overline{t}_r}/P'$ for a prime ideal $P'$ of height $r-1$ and some finite
	field $\mathbb{F}_q$ of characteristic $p$, 
	where the $\overline{t}_i$ are the images of the $t_i = \gamma_i - 1$ for topological generators
	$\gamma_i$ of $\Gamma$.
	The $\Gamma$-invariants of $(\Lambda/P)^{\Gamma}$ are annihilated by all $\overline{t}_i$.  If this invariant group 
	had a nonzero element, it would be annihilated by all $\overline{t}_i$.  The primality of $P'$ would then 
	force $P'$ to contain all $\overline{t}_i$. Since $P'$ is not maximal, this proves the claim, and hence part (a).

	Suppose we are given an element $f \in \La$ which is not a unit in $\La$ and is central in $\Omega$
	but is not in any prime ideal of codimension $r$ in the support of $M$.  
	As $M/f^nM$ and $M[f^n]$ are supported in codimension $r+1$, these $\La$-modules are finite.  
	It follows that we have isomorphisms $\rE^r_{\La}(M/M[f^n]) \cong \rE^r_{\La}(M)$ and
	then exact sequences
	$$
		0 \to \rE^r_{\La}(M) \xrightarrow{f^n} \rE^r_{\La}(M) \to \rE^{r+1}_{\La}(M/f^nM) \to 0.
	$$
	We write
$$
		\rE^r_{\La}(M) \cong \invlim{n} \rE^{r}_{\La}(M)/f^n \rE^{r}_{\La}(M)\cong \invlim{n} 
		\rE^{r+1}_{\La}(M/f^nM) 
		\cong \invlim{n} (M/f^nM)^{\dagger},
$$
	where multiplication by $f$ induces the map $(M/f^{n+1}M)^{\dagger} \to (M/f^nM)^{\dagger}$, which
	is the twist by the inverse of $\bigwedge^r \Gamma$ of
	$M^{\vee}[f^{n+1}] \to M^{\vee}[f^n]$.  It is clear from the latter description that $\rE^r_{\La}(M)$
	can have no nonzero finite submodule (and for this, it suffices to prove the statement as a $\La$-module,
	in which case the existence of $f$ is guaranteed), so we have part (b).
\end{proof}

\begin{supremark}
	A non-unit $f \in \La$ as in Proposition \ref{Er}(b) always exists.  That is, consider the finite set of height $r$
	prime ideals conjugate under $\mc{G}$ to a prime ideal in the support of $M$.  
	The union of these primes is not the maximal ideal of $\La$, so 
	we may always find a non-unit $b \in \La$ not contained in any prime in the set.  
	The product of the distinct $\mc{G}$-conjugates of $b$ is the desired $f$.
	Given a morphism $M \to N$ of small $\Omega$-modules, we obtain a canonical morphism between the isomorphisms 
	of Proposition \ref{Er}(b) for $M$ and $N$ by choosing $f$ to be the same element for both modules.
\end{supremark}

For a small $\Omega$-module $M$ and an $f$ as  in Proposition \ref{Er}(b), the quotient $M/fM$ is finite, $M$ itself is finitely generated and torsion over $\zp\ps{f}$.  The description of $\rE_{\La}^r(M)$
in Proposition \ref{Er}(b) then coincides (up to choice of a $\zp$-generator of $\bigwedge^r \Gamma$) with the usual definition
of the Iwasawa adjoint as a $\zp\ps{f}$-module.  In view of this, we make the following definition.
\begin{supdefinition} \label{def:adjoint}
	The Iwasawa adjoint $\alpha(M)$ of a small $\Omega$-module $M$ is $\rE_{\La}^r(M)$.
\end{supdefinition}
We then have the following simple lemma (cf. \cite[Proposition 15.29]{Washington}).

\begin{suplemma} \label{lem:smallseq}
	Let $0 \to M_1 \to M_2 \to M_3 \to 0$
	be an exact sequence of small $\Omega$-modules.  The long exact sequence of $\Ext$-groups 
	yields an exact sequence 
	$$
		0 \to \alpha(M_3) \to \alpha(M_2) \to \alpha(M_1) \to \finite
	$$
	of right $\Omega$-modules, where $\finite$ is the zero module if $(M_3)_{\fin} = 0$.
\end{suplemma}

We recall the following consequence of Grothendieck duality \cite[Chapter V]{Hartshorne},
noting that $\La$ is its own dualizing module in that $\La$ is regular (and that $\Omega$ is a finitely generated, free
$\La$-module).

\begin{supproposition} \label{prop:bjork}
	For a finitely generated $\Omega$-module $M$, there is a convergent
	spectral sequence 
	\begin{equation} \label{EEss}
	\rE_{\La}^p(\rE_{\La}^{r+1-q}(M)) \Rightarrow M^{\delta_{p+q,r+1}},
	\end{equation}
	natural in $M$, of right $\Omega$-modules,
	where $\delta_{i,j} = 1$ if $i = j$ and $\delta_{i,j} = 0$ if $i \neq j$.
	Moreover, $\rE_{\La}^i(\rE_{\La}^j(M)) = 0$ for $i < j$ and for $i > r+1$.  
\end{supproposition}

This implies the following.

\begin{supcor} \label{EEconseq}
	Let $M$ be a finitely generated $\Omega$-module. 
	\begin{itemize}
		\item[(a)] For $r = 1$, one has $\rE_{\La}^i(M^*) = 0$ for all $i \ge 1$.  Hence, $M^*$ is $\La$-free for any $M$.
		\item[(b)] For $r = 2$, one has $\rE_{\La}^2(M^*) = 0$ 
		and $\rE_{\La}^1(M^*) \cong \rE_{\La}^3(\rE_{\La}^1(M))$, so $\rE_{\La}^1(M^*)$ is finite.
		\item[(c)] If $M$ is small, then there is an exact sequence of $\Omega$-modules
		$$
			0 \to \rE_{\La}^{r+1}(\rE_{\La}^{r+1}(M)) \to M \to \rE_{\La}^r(\rE_{\La}^r(M)) \to 0.
		$$
		That is, $\alpha(\alpha(M)) \cong M/M_{\fin}$ as $\Omega$-modules.
	\end{itemize}
\end{supcor}

For a left (resp., right) 
$\Omega$-module $M$,  we let $M^{\iota}$ denote the right (resp., left) $\Omega$-module that is $M$ as an $\mc{O}$-module and on which $g \in \mc{G}$ acts as $g^{-1}$ does on $M$. 
The following is a consequence of the theory of Iwasawa adjoints for $r = 1$ (see \cite[Lemma 3.1]{jannsen}), in which case $\La$-small means $\La$-torsion.

\begin{suplemma} \label{lem:extcycl}
	Let $d \ge 1$, and let $f_i$ for $1 \le i \le d$ be elements of $\Lambda$
	such that $(f_1,\ldots,f_d)$ has height $d$.  Set $M = \Lambda/(f_1,\ldots,f_d)$. 
	Then $\rE_{\La}^i(M) \cong (M^{\iota})^{\delta_{i,d}}$ for
	all $i \ge 0$.
\end{suplemma}

\begin{proof}
	This is clearly true for $d = 0$.  Let $d \ge 1$, and set $N= \La/(f_1, \ldots, f_{d-1})$ so that $M = N/(f_d)$.  The exact sequence 
	$$
		0 \to N \xrightarrow{f_d} N \to M \to 0
	$$
	that is a consequence of Lemma \ref{lem:dumbfinite}
	gives rise to a long exact sequence of Ext-groups.  By induction on $d$, the only nonzero terms of that
	sequence form a short exact sequence 
	$$
		0 \to N^{\iota} \xrightarrow{(f_d)^\iota} N^{\iota} \to \rE_{\La}^d(M) \to 0,
	$$
	and the result follows.
\end{proof}

For more general $\Omega$-modules, we can for instance prove the following.

\begin{supproposition} \label{prop:adjoint}
	Suppose that $\mc{G} \cong \Gamma \times \Delta$, where $\Delta$ is abelian of order prime to $p$.
	Let $M$ be a small $\Omega$-module. 
	Then $(M/M_{\fin})^{\iota}$ and $\alpha(M)$ have the same class in the quotient of the Grothendieck group 
	of the category of 
	small right $\Omega$-modules by the image of the Grothendieck group of the category of finite right $\Omega$-modules.
  In particular, as $\La$-modules, their $r$th localized Chern classes agree.
\end{supproposition}

\begin{proof}
	By taking $\Delta$-eigenspaces of $M$ (passing to a coefficient ring containing $|\Delta|$th roots of unity), 
	we can reduce to the case that $\Delta$ is trivial.  
	It then suffices by Lemmas \ref{lem:groth} and \ref{lem:smallseq} to show the first statement 
	for $M = \La/P$, where $P$ is a height $r$ prime. Let  $\iota \colon
	\La \to \La$ be the involution determined by inversion of group elements.  We can compute
	$\alpha(M) =   \rE^r_{\La}(M) = \mathrm{Ext}^r_{\La}(M,\La)$ by an injective resolution of $\La$
	by $\La$-modules.  Every group in the resulting complex of homomorphism groups will
	be killed by $\iota(P)$, so $\alpha(M)$ will be annihilated by
	$\iota(P)$.  Clearly,
	$\iota(P)$ is the only codimension $r$ prime possibly in the support of $\rE^r_{\La}(M)$, and it is in the
	support since $\alpha(\alpha(M)) \cong M/M_{\fin}$ by Corollary \ref{EEconseq}(c). 
\end{proof}

The following particular computation is of interest to us.  Let $\mc{G}'$ be a closed subgroup of $\mc{G}$, and 
	let $M$ be a finitely generated left $\Omega' = \zp\ps{\mc{G}'}$-module.	
	Set $\Ga' = \mc{G}' \cap \Ga$, and let $\La' = \zp\ps{\Ga'}$.
	For $i = 0$ we have a right action of $g \in \Ga'$ on $f \in \rE_{\La'}^0(M) = \mathrm{Hom}_{\La'}(M,\La')$
	given by setting $(fg)(m) = f(m)g$ and a left action of $g$ on $f$ given by $(gf)(m) = f(m)g^{-1}$.
	This extends functorially to right and left actions of $\Ga'$ on $\rE_{\La'}^i(M)$ for all $i$.

\begin{suplemma} \label{lem:indext}
	With the above notation,	we have for all $i \ge 0$ an isomorphism of right $\Omega$-modules   
\begin{equation}
\label{eq:Rclaim1}
		\rE_{\La}^i(\Omega \otimes_{\Omega'} M) \cong \rE_{\La'}^i(M) \otimes_{\Omega'} \Omega
		\end{equation}
	and an isomorphism of left $\Omega$-modules 
\begin{equation}
\label{eq:Rclaim2}
	\rE_{\La}^i(\Omega \otimes_{\Omega'} M) \cong \Omega^{\iota} \otimes_{\Omega'} \rE_{\La'}^i(M).
\end{equation}
\end{suplemma}

\begin{proof}    Let us first show that $\Omega$ is flat over $\Omega'$.  This follows, for instance, from \cite[Lemma 2.4.3(a)]{ls}, since
	$\Omega$ is a free profinite $\Omega'$-module, i.e., a topological direct product of copies of $\Omega'$, 
	on a set right coset representatives for $\mc{G}'$ in $\mc{G}$.  By \cite[Lemma 2.1.6]{ls} and \cite[Lemma 2.1.7]{ls}
there are isomorphisms of right $\Omega$-modules
	$$
		\Ext_{\Omega}^i(\Omega \otimes_{\Omega'} M,\Omega) \cong \Ext_{\Omega'}^{i}(M,\Omega)
		\cong \Ext_{\Omega'}^{i}(M,\Omega') \otimes_{\Omega'} \Omega.
	$$
	The isomorphism (\ref{eq:Rclaim1}) follows, as the left term is
	$\rE_{\La}^i(\Omega \otimes_{\Omega'} M)$ and the right term is $\rE_{\La'}^i(M) \otimes_{\Omega'} \Omega$.
	The  isomorphism (\ref{eq:Rclaim2}) follows from (\ref{eq:Rclaim1}).
\end{proof}

\begin{supcor} \label{cor:indzp}
	Let $\mc{G}'$ be a closed normal subgroup of $\mc{G}$, and let $N$ denote the left $\Omega$-module 
	$\zp\ps{\mc{G}/\mc{G'}}$.
	Let $r' = \rank_{\zp}(\mc{G}' \cap \Gamma)$.  Then $\rE_{\La}^i(N) \cong (N^{\iota})^{\delta_{i,r'}}$
	as right $\Omega$-modules.
\end{supcor}

\begin{proof}
	Let $\Omega' = \zp\ps{\mc{G'}}$.  Note that $N \cong \Omega \otimes_{\Omega'} \zp$, so 
	$N^{\iota} \cong \zp \otimes_{\Omega'} \Omega$ as right
	$\Omega$-modules.
	By Lemmas \ref{lem:indext} and \ref{Er}, we have 
	$$
		\rE_{\La}^i(N) \cong \rE_{\La'}^i(\zp) \otimes_{\Omega'} \Omega \cong (\zp
		\otimes_{\Omega'} \Omega)^{\delta_{i,r'}}.
	$$
\end{proof}

 \bibliographystyle{plain}

\end{document}